\documentclass[a4paper,11pt]{article}
\usepackage{graphicx} % Required for inserting images
\usepackage[english]{babel}
\usepackage[utf8]{inputenc}
\usepackage{amsmath}
\numberwithin{equation}{section}
\usepackage{amsthm}
\usepackage{amssymb}
\usepackage{url}
\usepackage{hyperref}
\usepackage{mathabx}
\usepackage{biblatex}
\usepackage{enumerate}
\usepackage{appendix}
\usepackage{float}
\usepackage[colorinlistoftodos]{todonotes}
\usepackage[skip=5pt, indent=20pt]{parskip}
\usepackage{setspace}
\usepackage{multirow}
\usepackage[font=small,labelfont=bf]{caption}
\usepackage{geometry}
\newtheorem{lemma}{Lemma}[section]
\newtheorem{theorem}{Theorem}[section]
\newtheorem{remark}{Remark}[section]
\newtheorem{corollary}{Corollary}[section]

\allowdisplaybreaks[4]

\begin{document}

    \title{A Fully Discrete Local Discontinuous Galerkin Method for Quasilinear Stochastic Convection-Diffusion-Type Equations}
%\author{Yiming Chen}
    \author{
            Yiming Chen\thanks{
			Department of Mathematics, The Ohio State University, Columbus, OH 43210,
			USA. Email: \href{mailto:chen.11042@osu.edu}{chen.11042@osu.edu}. }
            \and
            Yunzhang Li\thanks{
			Research Institute of Intelligent Complex Systems, Fudan University, Shanghai 200433, P.R. China. Email: \href{mailto:li_yunzhang@fudan.edu.cn}{li\_yunzhang@fudan.edu.cn}. }
            \and 
			Yulong Xing\thanks{
			Department of Mathematics, The Ohio State University, Columbus, OH 43210,
			USA. Email: \href{mailto:xing.205@osu.edu}{xing.205@osu.edu}. } 
		    }

\date{\vspace{-2ex}}

\singlespacing
%\onehalfspacing

\pagenumbering{arabic}

\maketitle

\begin{abstract}
In this paper, we develop and analyze a fully discrete local discontinuous Galerkin (LDG) method with IMEX-Euler time discretization for a class of multi-dimensional quasilinear stochastic convection-diffusion-type equations driven by multiplicative $\mathcal Q$-Wiener noise. The leading diffusion matrix may depend on the solution as well as the spatial and temporal variables, while the lower-order drift and noise coefficients may depend on both the solution and its gradient. 
Under a suitable stochastic parabolicity condition, we establish unconditional high-moment stability estimates for the fully discrete scheme in the quasilinear setting. In the semilinear setting, where the leading diffusion matrix is independent of the solution but may vary in space and time, we further prove optimal high-moment strong error estimates of order $\mathcal O(h^{r+1})$ in space and $\mathcal O(k^{1/2})$ in time. A pathwise error estimate is then derived by combining the high-moment error bound with a discrete Kolmogorov argument. 
Numerical experiments are presented to illustrate the stability and convergence properties of the proposed method.
\end{abstract}
\vspace{2ex}

\begin{section}{Introduction}
In this paper, we develop and analyze a high order discontinuous Galerkin (DG) method for the quasilinear stochastic convection-diffusion-type equation driven by infinite-dimensional multiplicative noise, with a nonlinear solution-dependent leading diffusion matrix:
\begin{gather}
\label{spde:conv-diff}
    \begin{cases}
        \mathrm{d}u = \left\{\nabla \!\cdot\! [A(\cdot,u)\nabla u] + \psi(\cdot,u,\nabla u) \right\}\mathrm{d}t + g(\cdot,u,\nabla u) \mathrm{d}W_t,  & (\omega,\mathbf{x},t)\in \Omega\times \mathcal{D} \times (0,T], \\
        u(\cdot,0) = u_0,  & (\omega,\mathbf{x}) \in \Omega \times \mathcal{D}.
    \end{cases}
\end{gather}
where $\mathcal{D}=\mathbb{T}^d$ is the $d$-dimensional torus with periodic boundary conditions,  $A=\{a_{ij}(\omega,\mathbf{x},t,u)\}$ is a symmetric, semi-positive definite matrix, and $\{W_t\}_{t\in[0,T]}$ is a space-time $Q$-Wiener process of the form
\[ W_t = W_t(\mathbf{x},\omega) = \sum_{m=1}^{\infty} \sqrt{\gamma_m} e_m(\mathbf{x})\mathcal{B}_m(t,\omega) \quad \mathbf{x}\in \mathcal{D}, \]
where $\{e_m \}$ is an orthonormal basis of $L^2(\mathcal{D})$, $\{ \mathcal{B}_m\}$ is a sequence of independent standard Brownian motions defined on the filtered probability space $(\Omega, \mathcal{F},\{ \mathbb{F}_t\}_{0 \leq t \leq T}, \mathbb{P})$, and $\mathcal{Q}$ is a symmetric Hilbert-Schmidt operator on $L^2(\mathcal{D})$ satisfying $\mathcal{Q} e_m = \gamma_m e_m$, $\gamma_m > 0$, and $\mathrm{Tr}(\mathcal{Q}) := \sum_{m=1}^{\infty} \gamma_m < \infty$.

Convection-diffusion-type equations form a fundamental class of partial differential equations (PDEs) describing the combined effects of transport, diffusion, and reaction, with important applications in fluid dynamics, chemical reaction processes, and pollutant transport etc. However, for many realistic situations, deterministic models often fail to capture the intrinsic uncertainties arising from environmental fluctuations and random forcing. This motivates the study of stochastic PDEs (SPDEs), in which noise terms are incorporated into the original dynamics. Equation \eqref{spde:conv-diff} is a quasilinear second-order stochastic parabolic equation, and such problems have been studied extensively over the past several decades.
Two widely known approaches for the study of SPDEs are the semigroup method \cite{Chow2007_Parabolic_SPDEs, Prato_Zabczyk_2014} and the variational approach \cite{Krylov_Rozovskii_1981, Prevot_Rockner2007}. For the existence and uniqueness of solutions of stochastic parabolic PDEs, we refer the reader to \cite{Chow2007_Parabolic_SPDEs, Prato_Zabczyk_2014} for linear and semilinear equations, and to \cite{Debussche_SIAMMA2015, DuLiu_2019Cauchy, Krylov_AMS1999, Pardoux_Peng_1994} for quasilinear equations. Since analytic solutions to SPDEs are unavailable in most cases, it is an important but challenging task to design efficient and accurate numerical algorithms to approximate their solutions. Considerable progress has been made on numerical methods for stochastic parabolic PDEs, including finite difference \cite{AntonCohen_IMA2020, gyongy_1999Lattice, Millet_Morien_2005, Sauer_Mathcomp2015}, finite element \cite{Debussche2009weak, DuZhang_SIAMNu2002, Walsh_2005_finite_element, Yan_SIAMNu2005, XLV2026} and spectral methods \cite{Brehier_JCM2018, Conus_2019AnnalsAP, Jentzen_Kloeden_2009PRSA, Jentzen_IMA2020, Lord_Shardlow_2007_postprocessing}.

In this paper, we focus on DG methods, which are high-order finite element methods using discontinuous piecewise polynomial spaces. Their applications to hyperbolic conservation laws  were carried out by Cockburn \textit{et al.} in a series of papers \cite{CockburnShu_I_M2AN_1991, cockburnShu_II_Mathcomp1989, cockburnShu_III_JCP1989, cockburnHouShu_IV_Mathcomp1990, CockburnShu_V_JCP1998}. Compared to standard finite element methods, DG methods possess several advantages, including locally high-order approximation, flexibility on complex geometries, efficient $h$-$p$ adaptivity, and suitability for parallel implementations. For time-dependent equations containing higher-order spatial derivatives, including convection-diffusion-type equations, the higher-order operators require additional treatment within the DG framework. Among the various DG formulations developed for such models, the local discontinuous Galerkin (LDG) methods \cite{CockburnShu_SIAMNu1998} have been  applied and analyzed for many deterministic problems, including convection-diffusion equation \cite{Castillo_Cockburn_Mathcomp2002}, KdV-type equations \cite{XuShu_CMAME2007, YanShu_SIAMNu2002}, wave equation \cite{ChouShuXing_JCP2014, XingChouShu_2013energy} and time dependent fourth-order problems \cite{DongShu_SIAMNu2009}, where $L^2$-stability and high-order error estimates are established using Gauss-Radau projections and carefully designed numerical fluxes. More recently, there has been work on the analysis of fully-discrete LDG schemes for deterministic convection-diffusion equations coupled with IMEX Runge-Kutta time stepping \cite{WangShuZhang_SIAMNu2015, WangWangZhangShu_ESAIM2016}. By discretizing the nonlinear convection term explicitly and the diffusion term implicitly, IMEX schemes provide a good balance between computational efficiency and stability.  

There has also been growing interest in applying DG methods for SPDEs. They have been applied to stochastic Schr\"odinger equation \cite{ChenHongJi_IMA2016}, wave equation \cite{HongHouSun_JCP2022, Li_Wu_Xing_2022}, Allen-Cahn equation \cite{YangZhaoZhao_CiCP2024}, Cahn-Hilliard equation \cite{LiQinMingWang_CMA2018, zhouLi_CiCP2022_LDG}, backward SPDEs \cite{Li_SIAMFM2022}, and stochastic parabolic equations \cite{LiShuTang2021_ESAIM, YangZhaoZhao2023_NMPDE}, etc. In particular, Sun \textit{et al.} \cite{SunShuXing_JCP2022, SunShuXing_M2AN2023} proposed LDG methods for stochastic Maxwell equations driven by additive and multiplicative noises, where the methods are shown to preserve the linear growth of the stochastic energy and multi-symplectic structure, and optimal error estimates were analyzed for the semi-discrete method. Li \textit{et al.} developed and analyzed one-dimensional DG, ultra-weak DG and LDG methods for stochastic conservation laws \cite{LiShuTang2020_SISC}, stochastic KdV equation \cite{LiShuTang_JSC2020} and stochastic parabolic equation \cite{LiShuTang2021_ESAIM}. In these works, semi-discrete stability estimates were established for quasilinear problems, whereas optimal error estimates were obtained only in semilinear settings. In particular, the one-dimensional optimal error estimates were carried out only for stochastic KdV and parabolic equations with constant leading coefficients. In multiple dimensions, especially for unstructured meshes, the corresponding fully-discrete error estimates remain less developed. 
Another difficulty arises from time discretization. In contrast with deterministic problems, classical high-order Runge-Kutta methods cannot be applied directly in the strong sense because of the low temporal regularity of Brownian paths; see \cite{Kloeden_PLaten_2007} for a systematic discussion of strong and weak approximations for stochastic ordinary differential equations. It is therefore interesting to investigate fully-discrete DG schemes for SPDEs coupled with practical time-stepping methods.

The present paper extends the one-dimensional semi-discrete analysis in \cite{LiShuTang2021_ESAIM}, which treats equations driven by standard Brownian motion, to a two-dimensional fully-discrete LDG-IMEX-Euler approximation of stochastic convection-diffusion-type equations driven by multiplicative spatially colored \(\mathcal Q\)-Wiener noise on Cartesian meshes. 
The main contributions of the paper are fourfold. 
First, in the quasilinear setting, we establish fully-discrete stability estimates for the numerical solution. This is nontrivial because the leading diffusion operator depends on the solution itself, while the stochastic diffusion term may depend on both the solution and its gradient. As a result, the fully-discrete energy argument must balance the coercive contribution of the implicit diffusion term against the lower-order drift and stochastic terms, which leads to a nontrivial condition on the stochastic parabolicity constant.
Second, in the semilinear setting, where the leading diffusion matrix is independent of $u$, we prove optimal strong error estimates of order $\mathcal O(h^{r+1})$ in space and $\mathcal O(k^{1/2})$ in time. In particular, the leading coefficient matrix is allowed to depend on the spatial and temporal variables, which goes beyond the constant-coefficient setting treated in earlier one-dimensional analyses. The main difficulty in the proof is that the analysis is fully discrete, so the spatial approximation error, the temporal discretization error, and the stochastic error must be controlled simultaneously in a unified argument. In two space dimensions, this is further complicated by additional LDG coupling terms involving the auxiliary variables and projection errors, which must be estimated carefully in order to recover the optimal rates.
Third, we extend both the stability and error analysis to higher moments, thereby obtaining higher-moment fully-discrete stability and optimal higher-moment error estimates. This requires a careful combination of discrete stochastic estimates and parameter choices in order to close the higher-moment bounds. 
Fourth, by combining the higher-moment error bound with a discrete Kolmogorov argument, we derive a pathwise error estimate for the fully-discrete method, in a manner analogous to the pathwise analysis for stochastic Navier-Stokes equations in \cite{FengVo_CiCP2024}. 
Throughout the paper, we carry out the analysis in two space dimensions in order to present the main ideas as clearly as possible; the extension to higher-dimensional Cartesian meshes is straightforward. 
%We refer the reader to \cite{chen_semi-linear_1d} for the corresponding one-dimensional semilinear analysis. 
The treatment of unstructured meshes, as well as higher-order time discretizations such as those in \cite{XLV2026}, is left for future work.

%In model \eqref{spde:conv-diff}, we take into account the space-time mixed color noise rather than the standard Brownian motion considered in \cite{LiShuTang2020_SISC, LiShuTang_JSC2020, LiShuTang2021_ESAIM}. Moreover, we show the fully-discrete stability estimate for the quasilinear equation and optimal error estimates for semi-linear equation with variable leading coefficient. We further extend our analytic estimations to higher moments of the numerical solution and error. The pathwise error estimate is obtained with the help of the proposed discrete Kolmogorov lemma \ref{lem:Kolmogorov}. This is similar to the high moment and pathwise error estimates for stochastic Navier-Stokes equation in \cite{VoFeng2023_CICP}. Lastly, we present all analytic results in two-dimensional space on cartesian meshes, which can be generalized to higher-dimensional spaces straightforwardly. We refer the reader to \cite{chen_semi-linear_1d} for the analysis of 1D version of the model \eqref{spde:conv-diff-2d}. The corresponding fully nonlinear equation is analyzed in \cite{chen_non-linear_1d, chen_non-linear_2d}. The stability and error estimations on unstructured meshes, or for higher-order time discretizations are left for future work.  

Throughout this paper, we assume that $K := \sum_{i=1}^{\infty} \gamma_i \| e_i(x)\|_{\infty}^2$ is finite for the purpose of stability and error analysis. We denote by $(\cdot, \cdot)$ the $L^2$ inner product on $\mathcal{D}$. Given an integer $r \in \mathbb{Z}_+$, we use $\| \cdot\|$ and $\| \cdot\|_r := \| \cdot\|_{H^r}$ to denote the $L^2$ and $H^r$ norm on $\mathcal{D}$ with respect to the spatial variable $x \in \mathcal{D}$. 
We denote by the spaces $ L^{2}(\Omega \times [0,T]; H^{r})$ and $L^{2}(\Omega,L^{\infty}[0,T; H^{r}])$ the space of all adapted strongly continuous processes $\phi: \Omega \times [0,T] \to L^2(\mathcal{D})$ such that 
$$\| \phi\|_{L^{2}(\Omega \times [0,T]; H^{r})}^2 := \mathbb{E}\left[ \int_0^T \| \phi(t) \|_{H^r}^2 \mathrm dt\right] < \infty \quad \text{and} \quad \| \phi\|_{L^{2}(\Omega,L^{\infty}[0,T; H^{r}])}^2 := \mathbb{E}\left[ \sup_{t \in [0,T]} \| \phi(t) \|_{H^r}^2 \right] < \infty.$$ 
For any predictable process $\Phi(t)$ taking values in the space of Hilbert-Schmidt operators 
$\mathcal{L}_2^r =\mathcal L_2\!\left(\mathcal Q^{1/2}L^2(\mathcal D);H^r(\mathcal D)\right)$, we define the operator norm $\| \Phi(t)\|_{\mathcal{L}_2^r}^2 := \sum_{m=0}^{\infty} \gamma_m \| \Phi e_m\|_r^2$. In this setting we assume $\{ e_i(x)\}_{i=0}^{\infty} \subset H^r(\mathcal{D})$. Finally, $C>0$ denotes a generic constant independent of the spatial and temporal discretization parameters, and may change from line to line. In the two-dimensional analysis, we assume periodic boundary condition and $\mathcal{D} = [0,2\pi]^2$. 

The rest of the paper is organized as follows. Section \ref{sec:2d-scheme-prelim} presents the two-dimensional fully-discrete LDG method coupled with IMEX-Euler time discretization. It also provides several preliminary results for subsequent analysis. 
%lemmas for the numerical analysis in Sections \ref{sec:stability-2d} and \ref{sec:error-estimate-2d}. 
In Section \ref{sec:stability-2d}, the higher-moment unconditional stability estimates are established for the fully-discrete method on Cartesian meshes. The optimal higher-moment error estimates and pathwise error estimates are provided in Section \ref{sec:error-estimate-2d}. Numerical tests are reported in Section \ref{sec:numerical-test} to verify the stability and convergence rates of the proposed method. Some proofs are provided in Appendix \ref{appendixA}.

\end{section}

\begin{section}{IMEX-LDG Scheme and Preliminaries}
\label{sec:2d-scheme-prelim}

In this section, we present the fully-discrete IMEX-LDG formulation and several preliminary results for the subsequent analysis. For clarity of presentation, we use the two-dimensional case as a representative example to illustrate the main ideas of the method and the analysis, and the same framework can be extended to higher spatial dimensions on Cartesian meshes in a straightforward manner. 

We consider the following two-dimensional version of \eqref{spde:conv-diff} on the periodic domain $[0,2\pi]^2$:
% \begin{equation}
% \label{spde:conv-diff-2d}
%     \begin{cases}
%         \displaystyle \mathrm{d}u = {\color{red}\left( \sum_{i=1}^2 \frac{\partial}{\partial x_i} \left(\sum_{j=1}^{2} a_{ij} \frac{\partial u}{\partial x_j} \right) \right)}\mathrm{d}t + \psi(\cdot,u,u_x,u_y) \mathrm{d}t \\
%          \qquad \,\, + \, g(\cdot,u,u_x,u_y) \mathrm{d}W_t, \quad & (\omega,x,y,t)\in \Omega\times [0,2\pi]^2 \times (0,T]; \\
%         u(\omega,x,y,0) = u_0(x,y),  & (\omega,x,y) \in [0,2\pi]^2.
%     \end{cases}
% \end{equation}
% The principal part
% \[
% \sum_{i=1}^2 \frac{\partial}{\partial x_i}\Big(\sum_{j=1}^{2} a_{ij} \frac{\partial u}{\partial x_j}\Big)
% \]
\begin{equation}
\label{spde:conv-diff-2d}
    \begin{cases}
        \displaystyle \mathrm{d}u = \nabla \!\cdot\! \bigl(A(\cdot,u)\nabla u \bigr) \mathrm{d}t + \psi(\cdot,u,u_x,u_y) \mathrm{d}t \\
         \qquad \,\, + \, g(\cdot,u,u_x,u_y) \mathrm{d}W_t, \quad & (\omega,x,y,t)\in \Omega \times [0,2\pi]^2 \times (0,T]; \\
        u(\omega,x,y,0) = u_0(x,y),  & (\omega,x,y) \in \Omega \times [0,2\pi]^2.
    \end{cases}
\end{equation}
The principal part
\[  \nabla \!\cdot\! \bigl(A(\cdot,u)\nabla u \bigr)
\]
is the diffusion operator, whose leading diffusion matrix $A=\{a_{ij}(\omega,x,y,t,u)\}$ is allowed to depend on the solution $u$, so that the equation is quasilinear. The term $\psi(\omega,x,y,t,u,u_x,u_y)$ denotes a general lower-order drift term, which may include a convection term $b_1u_x+b_2u_y$ and reaction/source contributions. In this sense, \eqref{spde:conv-diff-2d} can be viewed as a stochastic convection-diffusion-type equation in a broad setting. The stochastic forcing is multiplicative and is given by the term $g(\omega,x,y,t,u,u_x,u_y)\,\mathrm{d}W_t$.

Throughout this paper, we impose the following assumptions on the data of \eqref{spde:conv-diff-2d}:
\begin{enumerate}[(i)]
    \item (Initial Condition) The initial condition satisfies $u_0 \in H^1([0,2\pi]^2)$.
    \item (Diffusion Coefficient) The leading coefficient matrix $A$ is differentiable of arbitrary order in each variable. In addition, there exist $\alpha, \Lambda>0$, such that for $i,j = 1,2$, 
    \begin{align*}
        \xi^{\top} A(\omega,x,y,t,u)\xi \geq \alpha | \xi |^2, \qquad |a_{ij}(\omega,x,y,t,u)|^2 \leq \Lambda, 
        % \big(A(\omega,x,y,t,u)\xi, \xi\big):=
        % \xi^{\top} A(\omega,x,y,t,u)\xi \geq \alpha \| \xi\|^2, \qquad |a_{ij}(\omega,x,y,t,u)|^2 \leq \Lambda, 
    \end{align*}
    for every $\xi \in \mathbb{R}^2$ and every $(\omega,x,y,t,u) \in \Omega \times [0, 2\pi]^2 \times [0,T] \times \mathbb{R}$.
    
    \item (Lower-order drift term) The function $\psi$ satisfies
    \begin{align*}
        &|\psi(\omega,x,y,t,u,v_1,v_2)|^2 \leq B_2^2(1+|u|^2) + B_3^2 (|v_1|^2+|v_2|^2), \\
        &|\psi(\omega,x,y,t,u,v_1,v_2)-\psi(\omega,x,y,t,u',v_1',v_2')| \leq B_1(|u-u'| + |v_1-v_1'| + |v_2-v_2'|), \\
        &|\psi(\omega,x,y,t,u,v_1,v_2)-\psi(\omega,x,y,t',u,v_1,v_2)| \leq B_4|t-t'|^{\frac{1}{2}}(1 + |u| + |v_1| + |v_2|), 
    \end{align*}
    for some nonnegative constants $B_1,B_2,B_3,B_4$, and for all $(\omega,x,y,t,t',u,u',v_1,v_1',v_2,v_2') \in \Omega \times [0, 2\pi]^2 \times [0,T]^2 \times \mathbb{R}^6$.
    
    \item (Noise Coefficient) There exist nonnegative constants $D_1,\ldots,D_5$ such that
    \begin{align*}
        &|g(\omega,x,y,t,u,v_1,v_2)|^2 \leq D_3^2(1+|u|^2) + D_4^2(|v_1|^2+|v_2|^2), \\
        &|g(\omega,x,y,t,u,v_1,v_2)-g(\omega,x,y,t,u',v_1',v_2')| \leq D_1|u-u'| + D_2(|v_1-v_1'|+|v_2-v_2'|), \\
        &|g(\omega,x,y,t,u,v_1,v_2)-g(\omega,x,y,t',u,v_1,v_2)| \leq D_5|t-t'|^{\frac{1}{2}}(1 + |u| + |v_1|+|v_2|),
    \end{align*}
    for all $(\omega,x,y,t,t',u,u',v_1,v_1',v_2,v_2') \in \Omega \times [0, 2\pi]^2 \times [0,T]^2 \times \mathbb{R}^6$.
\end{enumerate}
Assumptions {\rm(iii)} and {\rm(iv)} are standard Lipschitz continuity and linear growth conditions imposed on the lower-order drift and stochastic diffusion terms. Similar hypotheses were adopted in \cite{LiShuTang2021_ESAIM} to guarantee the existence and uniqueness of the semi-discrete DG solution. In this work, these assumptions will be used repeatedly in the stability and error analysis of the fully-discrete scheme.

\begin{subsection}{Fully-Discrete Numerical Scheme}

Let $ \mathcal{T}_h = \left\{ I_i \times J_j = [x_{i-1/2}, x_{i+1/2}] \times [y_{j-1/2}, y_{j+1/2}],\, i = 1,2,\ldots,N_x;\, \, j = 1,2,\ldots,N_y \right\}$ be a quasi-uniform rectangular partition of the domain $\mathcal{D} = [0, 2\pi]^2$. We denote the cell centers by $x_i = (x_{i-1/2}+ x_{i+1/2})/2$ and $y_j = (y_{j-1/2} + y_{j+1/2})/2$. The mesh sizes in $x,y$ directions are given as $h_{x,i} = x_{i+1/2}- x_{i-1/2}$, $h_{y,j} = y_{j+1/2} - y_{j-1/2}$ with $h_x = \max_i h_{x,i}$, $h_y = \max_j h_{y,j}$, and $h = \max(h_x,h_y)$ being the maximum mesh size. The ratio $\min(h_{x,i},h_{y,j})/h$ is assumed to be bounded below by a positive constant for all $i,j$ as $h \to 0$. 

For a nonnegative integer $r$, let $H^r(\mathcal{T}_h)$ denote the broken Sobolev space associated with the partition $\mathcal{T}_h$. We denote by $P^r(I_i)$ the space of one-dimensional polynomials of degree at most $r$ on the interval $I_i$, and by 
$Q^r(I_i\times J_j) := P^r(I_i)\otimes P^r(J_j)$
the tensor-product polynomial space of degree at most $r$ in each variable on the cell $I_i\times J_j$. The corresponding discontinuous finite element space is then defined by
% , and the two-dimensional tensor product discontinuous piecewise polynomial space of degree at most $r$ in each variable is defined as 
\[ \mathbb{V}_h := \{ v(x,y): v|_{I_i \times J_j} \in Q^r(I_i \times J_j), \quad i = 1,2,\ldots,N_x;\, j = 1,2,\ldots,N_y\}.  \]
For any $v\in\mathbb{V}_h $, we denote by $v_{i+1/2,y}^\pm:=v(x_{i+1/2}^\pm,y)$ and $v_{x,j+1/2}^\pm:=v(x,y_{j+1/2}^\pm)$ the traces of $v$ on the cell interfaces in the $x$- and $y$-directions, respectively.
The jumps at the cell interface along $x$ or $y$ directions are denoted as $[v]_{i+1/2, y} = v_{i+1/2,y}^+ - v_{i+1/2,y}^-, \, y \in J_j$ and $[v]_{x,j+1/2} = v_{x,j+1/2}^+ - v_{x,j+1/2}^-, \, x \in I_i$, respectively. The semi-norm of global jumps is defined as $\| [v]\|_{\Gamma_h}^2 := \sum_{i,j} \left(\int_{J_j} [v]_{i-1/2,y}^2 \, \mathrm{d}y + \int_{I_i} [v]_{x,j-1/2}^2\, \mathrm{d}x \right) $. For the temporal discretization, let $\{ t^n\}_{n=0}^{N_T}$ be a partition of $[0,T]$ and set the time step $k_n = t^{n}-t^{n-1}$ for $n=1,\ldots,N_T$. For simplicity of presentation, we assume that $k_n \equiv k = T/N_T$, and all the results below remain valid for variable time steps. 

To present the LDG method for approximating the SPDE \eqref{spde:conv-diff-2d}, we first rewrite the model into the following first order system by introducing auxiliary variables:
% by setting $w_1 = a_{11}u_x + a_{12}u_y$, $w_2 = a_{21}u_x + a_{22}u_y$.
\begin{gather}\label{eq:auxiliary}
    \begin{cases}
        \mathrm{d}u = \left(w_{1,x} + w_{2,y}  \right)\mathrm{d}t + \psi(\cdot,x,y,t,u,u_x,u_y) \mathrm{d}t + g(\cdot,x,y,t,u,u_x,u_y) \mathrm{d}W_t, \\
        v_1(x,y,t) = u_x(x,y,t),  \\
        v_2(x,y,t) = u_y(x,y,t),  \\
        w_1(x,y,t) = a_{11}(x,y,t,u) v_1(x,y,t) + a_{12}(x,y,t,u) v_2(x,y,t), \\
        w_2(x,y,t) = a_{21}(x,y,t,u) v_1(x,y,t) + a_{22}(x,y,t,u) v_2(x,y,t).
    \end{cases}
\end{gather}
For notational simplicity, we suppress the explicit dependence of random coefficients and stochastic processes on $\omega$. 
%We will omit the dependence of stochastic processes on $\omega$ as they are not defined in a pathwise sense. 
Utilizing the LDG spatial discretization and IMEX-Euler temporal discretization, the fully-discrete scheme for the above system is given as: for any $\omega \in \Omega$, given $u_h^n(\omega,x,y) \in \mathbb{V}_h$, find $(u_h^{n+1}, v_{1,h}^{n+1}, v_{2,h}^{n+1}, w_{1,h}^{n+1},w_{2,h}^{n+1})(\omega,x,y) \in (\mathbb{V}_h)^5$, such that for any test functions $(r_h,z_h, p_h,q_h,\phi_h) \in (\mathbb{V}_h)^5$, we have
\begin{align}
    (u_h^{n+1} - u_h^n,r_h)_{ij} &= kH_{i,j}^+(w_{1,h}^{n+1},w_{2,h}^{n+1},r_h) + k(\psi(x,y,t^n,u_h^n,v_{1,h}^n,v_{2,h}^n),r_h)_{ij} \notag \\
    &\quad + (g(x,y,t^n,u_h^n,v_{1,h}^n,v_{2,h}^n) \Delta W_n, r_h)_{ij} \,,  % \YX{\text{add $\omega$ in } \psi \text{ and } g?} 
    \label{eq:fully-discrete-1-2d} \\
    (v_{1,h}^{n+1},p_h)_{ij} &= L_{i,j}^{x-}(u_h^{n+1},p_h), \quad (v_{2,h}^{n+1},q_h)_{ij} = L_{i,j}^{y-}(u_h^{n+1},q_h),  \label{eq:fully-discrete-2-2d} \\
    (w_{1,h}^{n+1},z_h)_{ij} &= (a_{11}(x,y,t^{n+1},u_h^{n+1})v_{1,h}^{n+1}, z_h)_{ij} + (a_{12}(x,y,t^{n+1},u_h^{n+1})v_{2,h}^{n+1}, z_h)_{ij}\,,  \label{eq:fully-discrete-3-2d} \\
    (w_{2,h}^{n+1},\phi_h)_{ij} &= (a_{21}(x,y,t^{n+1},u_h^{n+1})v_{1,h}^{n+1}, \phi_h)_{ij} + (a_{22}(x,y,t^{n+1},u_h^{n+1})v_{2,h}^{n+1}, \phi_h)_{ij}\,.  \label{eq:fully-discrete-4-2d}
\end{align}
where we define $(\cdot, \cdot)_{ij}$ as the $L^2$ inner product on $I_i \times J_j$, \(\Delta W_n:=W_{t^{n+1}}-W_{t^n}\), and introduce the notations 
\begin{align*}
    H_{i,j}^+(w_{1,h},w_{2,h},r_h) &= -\iint_{I_i \times J_j} w_{1,h}(r_h)_x + w_{2,h}(r_h)_y \, \mathrm{d}x\mathrm{d}y + \int_{J_j} ({w}_{1,h}^+ r_h^-)_{i+\frac{1}{2},y} - ({w}_{1,h}^+ r_h^+)_{i-\frac{1}{2},y} \, \mathrm{d}y \\
    &\quad + \int_{I_i} ({w}_{2,h}^+ r_h^-)_{x,j+\frac{1}{2}} - ({w}_{2,h}^+ r_h^+)_{x,j-\frac{1}{2}} \, \mathrm{d}x, \\
    % L_{i,j}^-(u_h,p_h,q_h) &= -(u_h,(p_h)_x) -(u_h,(q_h)_y) + \int_{J_j} (\hat{u}_{h} p_h^-)_{i+\frac{1}{2},y} - (\hat{u}_{h} p_h^+)_{i-\frac{1}{2},y} \, dy \\
    % &+ \int_{I_i} (\hat{u}_{h} q_h^-)_{x,j+\frac{1}{2}} - (\hat{u}_{h} q_h^+)_{x,j-\frac{1}{2}} \, dx, \\
    L_{i,j}^{x-}(u_h,p_h) &= -\iint_{I_i \times J_j} u_h(p_h)_x \, \mathrm{d}x\mathrm{d}y + \int_{J_j} ({u}_{h}^- p_h^-)_{i+\frac{1}{2},y} - ({u}_{h}^- p_h^+)_{i-\frac{1}{2},y} \, \mathrm{d}y, \\
    L_{i,j}^{y-}(u_h,q_h) &= -\iint_{I_i \times J_j} u_h(q_h)_y \, \mathrm{d}x\mathrm{d}y + \int_{I_i} ({u}_{h}^- q_h^-)_{x,j+\frac{1}{2}} - ({u}_{h}^- q_h^+)_{x,j-\frac{1}{2}} \, \mathrm{d}x.
\end{align*}
For the leading diffusion term, we adopt the alternating fluxes $({w}_{1,h}^+,{u}_{h}^-)$ and $({w}_{2,h}^+,{u}_{h}^-)$, as embedded in the pair of $H_{i,j}^+(w_{1,h},w_{2,h},r_h)$ in \eqref{eq:fully-discrete-1-2d} and $L_{i,j}^{x-}(u_h^{n+1},p_h), L_{i,j}^{y-}(u_h^{n+1},q_h)$ in \eqref{eq:fully-discrete-2-2d}. Other choices of numerical fluxes, including the generalized numerical fluxes studied in \cite{SunShuXing_JCP2022, SunXing_Mathcomp2021,CX2024}, can also be considered.

The initial condition $u_h^0$ can be taken as an approximation of the initial value $u_0$ with sufficient accuracy, for instance, the local Gauss-Radau projections of $u_0$ defined in Section \ref{subsec:prelim-2d}. The global form of the fully-discrete scheme can be obtained by summing over all $i$ and $j$ in \eqref{eq:fully-discrete-1-2d}-\eqref{eq:fully-discrete-4-2d}. For ease of notation, we set $L_{i,j}^-(u_h,p_h,q_h) = L_{i,j}^{x-}(u_h,p_h) + L_{i,j}^{y-}(u_h,q_h)$, $H^+ = \sum_{i,j} H_{i,j}^+$, $L^- = \sum_{i,j} L_{i,j}^-$. 

Lastly, let $\mathbf{X} := \{ X_n(\omega)\}_{n=0}^{N_T}$ be an $H^1(\mathcal{T}_h)$-valued discrete stochastic process. In particular, we write $\mathbf{u_h}=\{ u_h^n\}_{n=0}^{N_T}$, $\mathbf{v_{1,h}}=\{ v_{1,h}^n\}_{n=0}^{N_T}$, and $\mathbf{v_{2,h}}=\{ v_{2,h}^n\}_{n=0}^{N_T}$ for the corresponding sequences of numerical solutions. We define 
\[\langle \mathbf{X} \rangle := \sum_{\ell=0}^{N_T-1} \| X_{\ell+1} - X_{\ell}\|^2\]
to be the discrete temporal variation of $\mathbf{X}$. For each $\omega \in \Omega$, we introduce the discrete norms
\begin{equation}
\label{def:func-norm}
    % \| \{ X_{\ell}\}_{0}^{n} \|_{\mathcal{S}(0,n;L^2)} := \max_{0\leq \ell \leq n} \| X_{\ell}\|, \quad \| \{ X_{\ell}\}_{0}^{n} \|_{\ell^{2}(0,n;L^2)} := \left (k\sum_{\ell=0}^{n} \| X_{\ell}\|^2 \right)^{\frac{1}{2}}. \\
    \| \mathbf{X} \|_{\mathcal{S}} = \| \mathbf{X} \|_{\mathcal{S}(0,N_T;L^2)} := \max_{0\leq \ell \leq N_T} \| X_{\ell}\|, \qquad 
    \| \mathbf{X} \|_{\ell^2} = \| \mathbf{X} \|_{\ell^{2}(0,N_T;L^2)} := \left (k\sum_{\ell=0}^{N_T} \| X_{\ell}\|^2 \right)^{\frac{1}{2}}.
    % \| \mathbf{X} \|_{\mathcal{S}} = \| \{ X_{\ell}\}_{0}^{N_T} \|_{\mathcal{S}(0,N_t;L^2)} := \max_{0\leq \ell \leq N_T} \| X_{\ell}\|, \quad \| \mathbf{X} \|_{\ell^2} = \| \{ X_{\ell}\}_{0}^{N_T} \|_{\ell^{2}(0,N_T;L^2)} := \left (k\sum_{\ell=0}^{N_T} \| X_{\ell}\|^2 \right)^{\frac{1}{2}}.
\end{equation}
To simplify the notation, we will omit the dependence on $N_T$ and the underlying $L^2$-norm from the notations above whenever it is clear from the context. 

\end{subsection}

\begin{subsection}{Preliminaries}
\label{subsec:prelim-2d}    
In this subsection, we present a few lemmas needed for the stability and error analysis of the fully discrete scheme \eqref{eq:fully-discrete-1-2d}-\eqref{eq:fully-discrete-4-2d}, which will be studied in the following sections. 

\begin{lemma}
\label{lem:convex-ineq}
Let $a,b,c$ be three nonnegative real numbers. Given $q > 1$, for any $\epsilon>0$, it holds that
\begin{align}
    &(a+b)^q \leq C^\#(\epsilon)a^q + (1+\epsilon)b^q,  \label{ineq:convex-1} \\
    % & \left|a^q - b^q \right| \leq \frac{q}{2}|a-b|(a^{q-1} + b^{q-1}), \label{ineq:convex-2} \\
    &(a+b+c)^q \leq C^*(\epsilon)a^q + (2^{q-1}+\epsilon)(b^q + c^q),   \label{ineq:convex-3}
\end{align}
where $C^\#(\epsilon) = \left(1-(1+\epsilon)^{\frac{1}{1-q}} \right)^{1-q}$ and $C^*(\epsilon) = \left(1-2 \left(2^{q-1}+\epsilon \right)^{\frac{1}{1-q}} \right)^{1-q}$. 
\end{lemma}
This can be derived using the convexity of $f(x)=x^q$ for $q> 1$ and a suitable weighted application of Jensen's inequality, so the proof is omitted.
% weighted convexity inequality. A short proof is provided in Appendix \ref{appendix-convex-ineq}. 

% The standard inverse property for the finite element space $\mathbb{V}_h$ is given in the following lemma.
% \begin{lemma}[Inverse Inequality]
% \label{lem:inverse-ineq-2d}
% Let $K = I_i \times J_j$. For any $v \in \mathbb{V}_h$, there exists $\mu >0$ independent of $v,h,i,j$ such that 
% \begin{align}
% \label{ineq:inverse-2d}
% h \| \nabla v\|_K + h \| v\|_{\infty}+ h^{\frac{1}{2}}\| v\|_{\partial K} \leq \mu \| v\|_K,
% \end{align}
% where $\| v\|_{\partial K}^2 = \| v(\cdot,y_{j-\frac{1}{2}}^+)\|_{I_i}^2 + \| v(\cdot,y_{j+\frac{1}{2}}^-)\|_{I_i}^2 + \| v(x_{i-\frac{1}{2}}^+, \cdot)\|_{J_j}^2 + \| v(x_{i+\frac{1}{2}}^-, \cdot)\|_{J_j}^2$ is the $L^2$ norm on the boundary. 
% \end{lemma}

The standard $L^2$ projection onto the piecewise polynomial space is denoted by $\mathcal{P}: L^2(\mathcal{T}_h) \to \mathbb{V}_h$. In addition, we define two Gauss-Radau projections $\mathcal{P}^{\pm}: L^2(\mathcal{T}_h) \to \mathbb{V}_h$ as the tensor product of the one-dimensional Gauss-Radau projections $\mathcal{P}^{\pm} = \mathcal{P}_x^{\pm} \otimes \mathcal{P}_y^{\pm}$,
% \begin{align*}
%     \mathcal{P} = \mathcal{P}_x \otimes \mathcal{P}_y\,, \quad \mathcal{P}^{\pm} = \mathcal{P}_x^{\pm} \otimes \mathcal{P}_y^{\pm},
% \end{align*}
where the subscripts indicate that the one-dimensional projections are applied with respect to the corresponding variables. The one-dimensional projections $\mathcal{P}_x^{\pm}$ are defined by (similarly for $\mathcal{P}_y^{\pm}$)
\begin{align}  
    &\int_{I_i} (\mathcal{P}_x^{-}u - u) v \, \mathrm{d}x = 0, \quad \forall v \in P^{r-1}(I_i), \quad \text{and} \quad (\mathcal{P}_x^{-}u)_{i+\frac{1}{2}}^- = u_{i+\frac{1}{2}}^- \,, \label{def:gauss-radau-minus-1d} \\
    &\int_{I_i} (\mathcal{P}_x^{+}u - u) v \, \mathrm{d}x = 0, \quad \forall v \in P^{r-1}(I_i), \quad \text{and} \quad (\mathcal{P}_x^{+}u)_{i-\frac{1}{2}}^+ = u_{i-\frac{1}{2}}^+, \label{def:gauss-radau-plus-1d} 
\end{align}
for $i=1,2,\ldots,N_x$. For clarity, we give the explicit definition of $\mathcal{P}^{-}$, which is also presented in \cite{MengShuWu_MathComp2015}: for any $v \in Q^{r-1}(I_i \times J_j)$, we have
\begin{align*}
    &\iint_{I_i \times J_j} (\mathcal{P}^{-}u - u)v  \, \mathrm{d}y \, \mathrm{d}x = 0, \quad \mathcal{P}^{-}u(x_{i+\frac{1}{2}}^-,y_{j+\frac{1}{2}}^-) = u(x_{i+\frac{1}{2}},y_{j+\frac{1}{2}}), \\
    &\int_{J_j} \left(\mathcal{P}^{-}u(x_{i+\frac{1}{2}}^-,y) - u(x_{i+\frac{1}{2}},y) \right) v(x_{i+\frac{1}{2}}^-,y)\, \mathrm{d}y = 0, \\
    &\int_{I_i} \left(\mathcal{P}^{-}u(x,y_{j+\frac{1}{2}}^-) - u(x,y_{j+\frac{1}{2}}) \right) v(x,y_{j+\frac{1}{2}}^-)\, \mathrm{d}x = 0.  
\end{align*}

% Let $u$ be the exact solution, and $u_h$ be the numerical solution. We denote the error of the numerical scheme by $e_u = u - u_h$, $e_{v_i} = v_i - v_{i,h}$ and $e_{w_i} = w_i - w_{i,h}$, for $i=1,2$. 
%As in \eqref{def:func-norm}, we similarly define $\mathbf{e_{u}}=\{ e_{u}^n\}_{n=0}^{N_T}$, $\mathbf{e_{v_1}}=\{ e_{v_1}^n\}_{n=0}^{N_T}$ and $\mathbf{e_{v_2}}=\{ e_{v_2}^n\}_{n=0}^{N_T}$ to be the vector of numerical error at each time step. 
For the projection error, we have the following approximation result \cite{Cockburn_SIAMNu2001}. 
\begin{lemma}
\label{lem:proj-property-2d}
Let $K = I_i \times J_j$, $\Pi = \mathcal{P}, \mathcal{P}^{\pm}$. For any $u \in H^{r+1}(K)$ with $\eta = \Pi u - u$, it holds that
\begin{align}
\label{ineq:proj-property-1-2d}
\| \eta\|_{L^2(K)} + h^{\frac{1}{2}}\| \eta\|_{L^2( \partial K)} + h\| \eta\|_{H^1(K)} \leq Ch^{r+1}\|u\|_{H^{r+1}(K)}.
\end{align}
Moreover, we have $\| \eta\|_{L^{\infty}(K)} \leq Ch^{r}\|u\|_{H^{r+1}(K)}$ when $r \geq 1$. 
% In addition, for $u \in W^{r+1,\infty}(K)$, we have $\| \eta\|_{L^{\infty}(\partial K)} \leq Ch^{r+1}\|u\|_{W^{r+1,\infty}(K)}$. 
\end{lemma}

In spatial dimensions greater than one, the following two terms in \eqref{ineq:superconvergence-1}, \eqref{ineq:superconvergence-2} associated with the projection error do not vanish in general. To control them, we make use of the super-convergence property summarized below, and refer to \cite{ChengMengZhang_MathComp2016,MengShuWu_MathComp2015} for details. 
\begin{lemma}
\label{lem:superconvergence}
For any $(w,w_1,w_2) \in \mathbb{V}_h^3$ and any $u,v \in H^{r+2}(\mathcal{D})$, the following estimates hold
\begin{align}
    &|L^-(\mathcal{P}^-u - u, w_1,w_2)| \leq Ch^{r+1} \| u\|_{{r+2}}(\| w_1\|+\| w_2\|), \label{ineq:superconvergence-1} \\
    &|H^+(\mathcal{P}^+ u - u, \mathcal{P}^+ v - v, w)| \leq Ch^{r+1} (\| u\|_{{r+2}} + \| v\|_{{r+2}}) \| w\|.  \label{ineq:superconvergence-2}
\end{align}
Moreover, if $u|_K, \,v|_K \in P^{r+1}$ for every $K \in \mathcal{T}_h$, then
\[ L^-(\mathcal{P}^-u - u, w_1,w_2) = H^+(\mathcal{P}^+ u - u, \mathcal{P}^+ v - v, w) = 0.\]
\end{lemma}

For the operators $H^+$ and $L^-$, we summarize the following properties induced by the LDG spatial discretization and the alternating flux. The proof is a straightforward extension of the one-dimensional case given in \cite{WangShuZhang_SIAMNu2015} and is omitted here. 
\begin{lemma}
\label{lem:num-flux-2D}
For any $(v,w_1,w_2) \in \mathbb{V}_h^3$, we have
\begin{align}
    & |L^-(v,w_1,w_2)| \leq \left(\| \nabla v\| + \mu h^{-\frac{1}{2}} \| [v]\|_{\Gamma_h} \right) \left(\| w_1\| + \| w_2\| \right), \label{ineq:alternating-flux-2d} \\
    & H^+(w_1,w_2,v) + L^-(v,w_1,w_2) = 0.  \notag
\end{align}
\end{lemma}

Next, we present three lemmas addressing the stochastic aspects of our numerical analysis. The discrete Kolmogorov lemma below helps establish pathwise error estimates for the numerical scheme. It is closely related to the classic Kolmogorov continuity theorem \cite{DaPrato_Zabczyk_1992}, which gives a criterion for pathwise regularity (H\"older continuity) of a continuous stochastic process. 

\begin{lemma}[Discrete Kolmogorov Theorem]
\label{lem:Kolmogorov}
    Let $\{ X(t)\}_{t \in [0,T]}$ be a stochastic process taking values in a separable Banach space with norm $\| \cdot\|$. Let $0 = t_0 < t_1 < \cdots < t_N = T$ be a uniform partition of the interval $[0,T]$, with $k = {T}/{N}$. Suppose that $\{X_n^N \}_{n=0}^N$ is a discrete stochastic process approximating $X$ at the grid points, in the sense that there exist constants $\nu,\beta >0$, $C>0$, and a positive integer $N_0$ such that, for every $N \geq N_0$,
    \[ \mathbb{E}\left[ \max_{0 \leq n \leq N} \| X(t_n)-X_n^N\|^{\nu} \right] \leq C \left(\frac{1}{N} \right)^{1+\beta}.   \]
    Then, for any $0 \leq \gamma < {\beta}/{\nu}$, there exists a random variable $Z(\omega,\gamma)$ with $\mathbb{E}[|Z|^{\nu}] < \infty$, such that 
    \begin{align}
    \label{ineq:kolmogorov}
        \max_{0 \leq n \leq N} \|X(t_n,\omega)-X_n^N(\omega)\| \leq Z(\omega,\gamma) \left(\frac{1}{N} \right)^{\gamma}, \qquad \forall N \geq N_0.
    \end{align} 
    % for any $N \geq N_0$.
\end{lemma}
\begin{proof}
    We define $ Z(\omega, \gamma) = \displaystyle \max_{N\geq N_0} \max_{0 \leq n \leq N} N^{\gamma} \| X(t_n,\omega)-X_n^N(\omega)\|$, and note that
    \begin{align*}
        \mathbb{E}[Z^{\nu}] &\leq \mathbb{E} \left[ \max_{N\geq N_0} \max_{0 \leq n \leq N} N^{ \nu\gamma} \| X(t_n)-X_n^N\|^{\nu} \right] \\
        &\leq \sum_{N=N_0}^{\infty} N^{ \nu\gamma} \,\mathbb{E}\left[ \max_{0 \leq n \leq N} \| X(t_n)-X_n^N\|^{\nu} \right] \leq C\sum_{N=N_0}^{\infty} \frac{1}{N^{1+\beta - \nu\gamma}} < \infty,
    \end{align*}
    since $\gamma < \beta/\nu$. The desired estimate \eqref{ineq:kolmogorov} follows directly from the definition of $Z(\omega,\gamma)$.
    % By definition, we have $\displaystyle \max_{0 \leq n \leq N} \|X(t_n,\omega)-X_n^N(\omega)\| \leq Z(\omega,\gamma) N^{-\gamma} $ for any $N \geq N_0$. %{\color{blue}more details}
\end{proof}

The classical Burkholder-Davis-Gundy (BDG) inequality relates the maximal function of a local martingale to its quadratic variation. Both the continuous and discrete versions play an important role in the numerical analysis in Sections \ref{sec:stability-2d} and  \ref{sec:error-estimate-2d}. We recall below the discrete version. %The continuous version of the inequality can be found in many stochastic calculus texts, while \cite{ondrejat_Prohl_Walkington_2023} provides a discrete version included below for reader's convenience.  
\begin{lemma}[Discrete Burkholder-Davis-Gundy Inequality \cite{ondrejat_Prohl_Walkington_2023}]
\label{lem:BDG}
    Let $(\Omega, \mathcal{F}, \mathbb{P})$ be a probability space equipped with a (discrete) filtration $\{\mathcal{F}^n\}_{n=0}^{N}$. Let $\{X^n\}_{n=0}^{N}$ with $X^0 \equiv 0$ be an $\{\mathcal{F}^n\}_{n=0}^{N}$-martingale taking values in a separable Hilbert space $H$. Then for each $p \geq 1$ there exist constants $0 < 1/C_b' < C_b$ such that 
    \begin{align*}
        \frac{1}{C_b'} \mathbb{E} \left[ \left( \sum_{n=1}^{N} \| X^n - X^{n-1} \|_H^2 \right)^{p/2} \right] &\leq \mathbb{E} \left[ \max_{0 \leq n \leq N} \| X^n \|_H^p \right] 
        \leq C_b \mathbb{E} \left[ \left( \sum_{n=1}^{N} \| X^n - X^{n-1} \|_H^2 \right)^{p/2} \right]. 
        % \\
        % C_p' \mathbb{E} \left[ \left( \sum_{n=1}^{N} \mathbb{E}\left[\| X_\tau^n - X_\tau^{n-1} \|_H^2 | \mathcal{F}_{n-1} \right] \right)^{p/2} \right] &\leq \mathbb{E} \left[ \max_{0 \leq n \leq N} \| X_\tau^n \|_H^p \right] 
        % \leq C_p \mathbb{E} \left[ \left( \sum_{n=1}^{N} \mathbb{E}\left[\| X_\tau^n - X_\tau^{n-1} \|_H^2 | \mathcal{F}_{n-1} \right] \right)^{p/2} \right].
    \end{align*}
\end{lemma}
\begin{remark}[Continuous and Discrete BDG Inequalities for Discrete Martingales as It\^o Integrals]
\label{rmk:BDG}
    In our fully-discrete numerical analysis, the discrete martingales under consideration are It\^o integrals of the form $I_m = \sum_{n=0}^m G^n \Delta W_n$, where $G^n$ is a function that depends on the numerical solutions at time step $t_n$. Then by virtue of the continuous and discrete BDG inequalities, we have for $p \geq 1$ and $N \leq N_T$ 
    \begin{align*}
        \frac{1}{C_b'}\mathbb{E}\left[ \left( \sum_{n=0}^{N} \| G^n \Delta W_n \|^2 \right)^{p/2} \right] \leq \mathbb{E} \left[ \max_{0 \leq m \leq N} \left\| \sum_{n=0}^m G^n \Delta W_n  \right\|^p \right] \leq C_b \mathbb{E}\left[ \left( kK\sum_{n=0}^{N} \| G^n \|^2 \right)^{p/2} \right],
    \end{align*}
    where the left ``$\leq$'' follows from Lemma \ref{lem:BDG} and the right ``$\leq$" is derived from the continuous BDG inequality, since the discrete martingale $I_m$ can be viewed as an It\^o integral using piecewise constant interpolation. 
\end{remark}

Lastly, we establish a higher-moment H\"older continuity estimate in time for the strong solution $u$ in the spatial $H^{m}$ norm. We also note that Lemma \ref{lem:Holder-High-Moments} and Corollary \ref{coro:Holder-High-Moments} hold for higher spatial dimensions. 
\begin{lemma}[Higher Moment H\"older Continuity Estimate]
\label{lem:Holder-High-Moments}
    Let $u$ be the strong solution of \eqref{spde:conv-diff-2d}, and let $p \geq 1$. Suppose that $u \in L^{2p}\left(\Omega,L^{\infty}[0,T; H^{m+1}] \right)$, $w_1, w_2 \in L^{2p}\left(\Omega \times [0,T]; H^{m+1} \right)$, and $\psi(\cdot,u,\nabla u) \in L^{2p}\left(\Omega \times [0,T]; H^{m} \right)$, where $m$ is a nonnegative integer. In addition, we assume 
    \begin{equation}
        % \| g(\omega,\cdot,t,u,\nabla u)\|_{\mathcal{L}_2^m} \leq C\| g(\omega,\cdot,t,u,\nabla u)\|_m^{2p} \leq C\left(1 + \| u\|_m^{2p} + \| \nabla u \|_m^{2p} \right), \quad \forall (w,t) \in \Omega \times [0,T], 
        \| g(\omega,\cdot,t,u,\nabla u)\|_{\mathcal{L}_2^m}^2 \leq C\left(1 + \| u\|_m^{2} + \| \nabla u \|_m^{2} \right), \qquad \forall (\omega,t) \in \Omega \times [0,T], 
    \label{eq:holder_1}
    \end{equation}
    then 
    \[ \mathbb{E}\left[\| u(t)-u(s)\|_{m}^{2p} \right] \leq C(t-s)^p, \qquad \forall \, 0 \leq s < t \leq T. \] 
\end{lemma}

\begin{proof}
% We first note that 
% \begin{gather*}
%      u\in L^4(\Omega\times [0,T]; W^{m+2,4}]) \Rightarrow u^2 \in L^2(\Omega \times [0,T]; H^{m+2}]) \\
%      u, u^2 \in L^2(\Omega \times [0,T]; H^{m+2}]) \Rightarrow w \in L^2(\Omega \times [0,T]; H^{m+1}])
% \end{gather*}
% \begin{gather*}
%     (u(t)-u(s),z) = \left( \int_s^t w_x \, d\tau , z\right) + \left( \int_s^t f(u)_x \, d\tau , z\right) + \left( \int_s^t \psi \, d\tau , z \right) 
%     + \left( \int_s^t g \, dW_{\tau} ,z \right) \quad \forall z \in L^2
% \end{gather*}
The exact solution satisfies:
\[ u(t)-u(s) = \int_s^t (w_{1,x} + w_{2,y} + \psi) \, \mathrm{d}\tau + \int_s^t g \, \mathrm{d}W_{\tau}.  \]
Utilizing the Jensen's inequality and H\"older inequalities yields
\begin{align*}
    \| u(t)-u(s)\|_m^{2p} &\leq 4^{2p-1} \left\| \int_s^t w_{1,x} \, \mathrm{d}\tau \right\|_m^{2p} + 4^{2p-1} \left\| \int_s^t w_{2,y} \, \mathrm{d}\tau \right\|_m^{2p} + 4^{2p-1} \left\| \int_s^t \psi \, \mathrm{d}\tau \right\|_m^{2p} + 4^{2p-1} \left\| \int_s^t g \, \mathrm{d}W_{\tau} \right\|_m^{2p} \\
    &\leq 4^{2p-1}(t-s)^{2p-1} \int_s^t \| w_{1,x}\|_m^{2p} + \| w_{2,y}\|_m^{2p} + \| \psi\|_m^{2p}\, \mathrm{d}\tau + 4^{2p-1} \left\| \int_s^t g \, \mathrm{d}W_{\tau} \right\|_m^{2p}. 
\end{align*}
Taking the expectation and applying the BDG inequality, H\"older inequality and the assumption \eqref{eq:holder_1}, we get
\begin{align*}
    &\mathbb{E}\left[\| u(t)-u(s)\|_m^{2p} \right] \leq 4^{2p-1}(t-s)^{2p-1} \mathbb{E}\left[ \int_s^t \| w_{1} \|_{m+1}^{2p} + \| w_{2}\|_{m+1}^{2p} + \| \psi\|_m^{2p} \, \mathrm{d}\tau \right] + 4^{2p-1} \mathbb{E}\left[ \left\| \int_s^t g \, \mathrm{d}W_{\tau} \right\|_m^{2p}  \right] \\
    &\qquad \leq 4^{2p-1}(t-s)^{2p-1} \mathbb{E}\left[ \int_s^t \| w_1 \|_{m+1}^{2p} + \| w_2\|_{m+1}^{2p} + \| \psi\|_m^{2p} \, \mathrm{d}\tau \right] + C (t-s)^{p-1} \mathbb{E}\left[ \int_s^t \| g \|_{\mathcal{L}_2^m}^{2p} \, \mathrm{d}\tau \right] \\
    &\qquad \leq C (t-s)^{2p-1} + C(t-s)^{p} \mathbb{E}\left[ \sup_{\tau \in [0,T]} \|u(\cdot,\tau)\|_{m+1}^{2p} \right]  \leq C (t-s)^{p},
\end{align*}
which completes the proof.
\end{proof}

As an immediate result of Lemma \ref{lem:Holder-High-Moments}, we have the following corollary. 
\begin{corollary}
\label{coro:Holder-High-Moments}
    Let $u$ be the strong solution of \eqref{spde:conv-diff-2d}, let $p \geq 1$ and $m$ be a nonnegative integer. If $w_1,w_2 \in L^{2p}(\Omega \times [0,T]; H^{m+1})$, $\psi(\cdot,u,\nabla u) \in L^{2p}(\Omega \times [0,T]; H^{m})$, and if $g(\cdot,u,\nabla u) \in L^{2p}(\Omega \times [0,T]; \mathcal{L}_2^m)$ satisfies \eqref{eq:holder_1},
    %  \begin{equation*}
    %     \| g(\omega,\cdot,t,u,\nabla u)\|_{\mathcal{L}_2^m} \leq C\| g(\omega,\cdot,t,u,\nabla u)\|_m^{}, \quad \forall (w,t) \in \Omega \times [0,T], 
    % \end{equation*}
    then for any uniform partition of $[0,T]$ given by $t_n = nT/N_T$, $n = 0,1,\ldots,N_T$, there exists a constant $C$ independent of $k$, such that
    \[ \frac{1}{k^{p-1}}\sum_{n=0}^{N_T-1} \mathbb{E}\left[ \| u(t_{n+1}) - u(t_n)\|_m^{2p} \right] \leq C. \]
\end{corollary}

\end{subsection}

\end{section}

\begin{section}{High-Moment Stability Estimate}
\label{sec:stability-2d}

% \begin{subsection}{Stability Estimate}
In this section, we show the second- and higher-moment unconditional stability estimates for the numerical solution $u_h$ of the fully discrete numerical scheme \eqref{eq:fully-discrete-1-2d}-\eqref{eq:fully-discrete-4-2d}. 

Before proving the main stability result, we first show that the initial discrete gradients $v_{1,h}^0$ and $v_{2,h}^0$ are well-defined and uniformly bounded. 
% the right-hand-side of \eqref{ineq:high-moment-stability-2D} is deterministic and bounded. 
\begin{lemma}
\label{lem:v_h^0-2d}
    Assume hypothesis {\rm(i)}, and take $u_h^0$ to be a suitable projection of the exact initial condition satisfying the standard approximation \eqref{ineq:proj-property-1-2d}. Then the quantities $v_{1,h}^0$ and $v_{2,h}^0$ computed from \eqref{eq:fully-discrete-2-2d} are deterministic, and satisfy 
    \[\| v_{1,h}^0\| + \| v_{2,h}^0\| < C\| u_0\|_{1}. \] 
\end{lemma}

\begin{proof}
% Since $u_0$ is deterministic, so is its projection $u_h^0=\mathcal{P}^-u_0$. 
Since $u_0$ is deterministic, so is its projection $u_h^0$. It follows from \eqref{eq:fully-discrete-2-2d} that $v_{1,h}^0$ and $v_{2,h}^0$ are also deterministic. 
Taking the test functions $p_h = v_{1,h}^0$ and $q_h = v_{2,h}^0$ in the fully-discrete scheme \eqref{eq:fully-discrete-2-2d}, we get 
$L^-(u_h^0,v_{1,h}^0,v_{2,h}^0) = \| v_{1,h}^0\|^2 + \| v_{2,h}^0\|^2$. Applying inequality \eqref{ineq:alternating-flux-2d} yields
\begin{align*}
    \| v_{1,h}^0\|^2 + \| v_{2,h}^0\|^2 \leq \left(\| \nabla u_h^0\| + \mu h^{-\frac{1}{2}} \| [u_h^0]\|_{\Gamma_h} \right) \left(\| v_{1,h}^0\| + \| v_{2,h}^0\| \right). 
\end{align*}
Utilizing Young's inequality and Lemma \ref{lem:proj-property-2d}, we have
\begin{align*}
    &\| v_{1,h}^0\|^2 + \| v_{2,h}^0\|^2 \leq 2\left(\| \nabla u_h^0\| + \mu h^{-\frac{1}{2}} \| [u_h^0]\|_{\Gamma_h} \right)^2 \\
    &\leq C \left(\| \nabla u_0\|^2 + \| \nabla (u_0-u_h^0)\|^2 + h^{-1}\| [u_0 - u_h^0]\|_{\Gamma_h}^2 \right) \leq C\| u_0\|_{1}^2. 
\end{align*}
This completes the proof. 
\end{proof}

We have the following unconditional stability result for the fully discrete scheme.
\begin{theorem}
\label{thm:stability-estimate-2D}
Assume hypotheses {\rm(i)}-{\rm(iv)}. For any fixed $q\in[1,\infty)$, suppose 
$$\alpha > \alpha_0 := \left(2^{2q-1}C_b' + 2^{5q-4}C_b \right)^{1/q} C_b^{1/q} D_4^2K.$$ 
Then there exists a constant $C$, independent of $h$ and $k$, such that
\begin{align}
\label{ineq:high-moment-stability-2D}
    % \mathbb{E}\left[ \max_{0\leq n \leq N_T} \| u_h^{n+1}\|^{2q} \right] &+ \mathbb{E}\left[\left(\sum_{n=0}^{N_T}  \| u_h^{n+1} - u_h^n\|^2 \right)^q \right] + \mathbb{E}\left[ \left(k\sum_{n=0}^{N_T} (\| v_{1,h}^{n+1}\|^2 + \| v_{2,h}^{n+1}\|^2) \right)^q \right] \notag \\
    % &\leq C\left(1+\| u_h^0\|^{2q} + k^q \| v_{1,h}^0\|^{2q} + k^q \| v_{2,h}^0\|^{2q}\right), 
    % &\mathbb{E}\left[  \| \{u_h^{n}\}_{0}^{N_T} \|_{\mathcal{S}}^{2q} \right] + \mathbb{E}\left[ \langle u_h^{n} \rangle^{q} \right] + \mathbb{E}\left[  \left(\| \{v_{1,h}^{n}\}_{0}^{N_T} \|_{\ell^2}^{2} + \| \{ v_{2,h}^{n}\}_0^{N_T} \|_{\ell^2}^{2} \right)^q \right] \leq C\left(1+\| u_0\|^{2q}+ k^q\| u_0\|_{H^1}^{2q} \right). \\
    \mathbb{E} \!\left[ \| \mathbf{u_h} \|_{\mathcal{S}}^{2q} \right] + \mathbb{E}\!\left[ \langle \mathbf{u_h} \rangle^{q} \right] + \mathbb{E}\!\left[  \left(\| \mathbf{v_{1,h}} \|_{\ell^2}^{2} + \| \mathbf{v_{2,h}} \|_{\ell^2}^{2} \right)^q \right] 
    \leq C\left(1+\| u_0\|^{2q}+ k^q\| u_0\|_{1}^{2q} \right).
\end{align}
\end{theorem}

\begin{proof}
The proof is carried out in the following two steps.

\noindent {\it Step 1.} Choose the test functions $r_h = u_h^{n+1}$, $p_h = w_{1,h}^{n+1}$, $q_h = w_{2,h}^{n+1}$ and $z_h = v_{1,h}^{n+1}$, $\phi_h = v_{2,h}^{n+1}$ in the fully discrete scheme \eqref{eq:fully-discrete-1-2d}-\eqref{eq:fully-discrete-4-2d}. Let $A(\cdot,t_{n+1},u_h^{n+1}) = \{ a_{11}, a_{12}; a_{21}, a_{22}\}$ be the $ 2\times2$ coefficient matrix. After summing over all computational cells and applying Lemma \ref{lem:num-flux-2D}, we have
\begin{align}
\label{eq:fully-discrete-5-2d}
%    (v_{1,h}^{n+1},v_{2,h}^{n+1})A(x,y,t,u_h^{n+1})(v_{1,h}^{n+1},v_{2,h}^{n+1})^T  = L^-(u_h^{n+1},w_{1,h}^{n+1},w_{2,h}^{n+1}) = -H^+(w_{1,h}^{n+1},w_{2,h}^{n+1},u_h^{n+1}).
\bigl(A(\cdot,t_{n+1},u_h^{n+1})(v_{1,h}^{n+1},v_{2,h}^{n+1})^{\top},\,(v_{1,h}^{n+1},v_{2,h}^{n+1})^{\top}\bigr)
=
L^-(u_h^{n+1},w_{1,h}^{n+1},w_{2,h}^{n+1})
=
-H^+(w_{1,h}^{n+1},w_{2,h}^{n+1},u_h^{n+1}).
\end{align}
Adding Eq. \eqref{eq:fully-discrete-5-2d} to Eq. \eqref{eq:fully-discrete-1-2d}, we get
\begin{align}
\label{eq:stability-main-2d}
    &(u_h^{n+1} - u_h^n,u_h^{n+1}) + k\bigl(A(\cdot,t_{n+1},u_h^{n+1})(v_{1,h}^{n+1},v_{2,h}^{n+1})^{\top},\,(v_{1,h}^{n+1},v_{2,h}^{n+1})^{\top}\bigr) \notag \\
    &\qquad = k(\psi(\cdot,t_n,u_h^n,v_{1,h}^n,v_{2,h}^n),u_h^{n+1}) 
    + (g(\cdot,t_n,u_h^n,v_{1,h}^n,v_{2,h}^n) \Delta W_n, u_h^{n+1}). 
\end{align}
By hypothesis {\rm(ii)}, we have for the left-hand side of \eqref{eq:stability-main-2d},
\begin{align}
\label{ineq:LHS-bound-2d}
    LHS \geq \frac{1}{2}\| u_h^{n+1}\|^2 - \frac{1}{2} \|u_h^n\|^2 + \frac{1}{2}  \| u_h^{n+1} - u_h^n\|^2 
    + \alpha k \left(\| v_{1,h}^{n+1}\|^2 +\| v_{2,h}^{n+1}\|^2 \right).
\end{align}

Fix $0\le m\le N_T-1$. Summing \eqref{eq:stability-main-2d} from $n=0$ to $l$, taking the maximum over $0\le l\le m$, and using \eqref{ineq:LHS-bound-2d}, we obtain
\begin{align*}
    &\max_{0 \leq l \leq m} \left\{ \frac12 \|u_h^{l+1}\|^2 - \frac12\|u_h^0\|^2, \,
    \frac{1}{2} \sum_{n=0}^{l} \| u_h^{n+1} - u_h^n\|^2 
    + \alpha k \sum_{n=0}^{l} \left(\| v_{1,h}^{n+1}\|^2 +\| v_{2,h}^{n+1}\|^2 \right) \right\} \notag \\
    &\qquad \leq I_1+ I_2 + I_3,
\end{align*}
where
\begin{align*}
I_1&:=k\max_{0\le l\le m}\left|\sum_{n=0}^{l}(\psi(\cdot,t_n,u_h^n,v_{1,h}^n,v_{2,h}^n),u_h^{n+1})\right|,\\
I_2&:=\max_{0\le l\le m}\left|\sum_{n=0}^{l}(g(\cdot,t_n,u_h^n,v_{1,h}^n,v_{2,h}^n)\Delta W_n,u_h^{n+1}-u_h^n)\right|,\\
I_3&:=\max_{0\le l\le m}\left|\sum_{n=0}^{l}(g(\cdot,t_n,u_h^n,v_{1,h}^n,v_{2,h}^n)\Delta W_n,u_h^n)\right|.
\end{align*}
Then we can bound the sum of the $\max_{0 \leq l \leq m}$ terms by 
\begin{align}
\label{ineq:stability-main-2D}
&\frac12\max_{0\le l\le m}\|u_h^{l+1}\|^2
+\frac12\sum_{n=0}^{m}\|u_h^{n+1}-u_h^n\|^2
+\alpha k\sum_{n=0}^{m}\bigl(\|v_{1,h}^{n+1}\|^2+\|v_{2,h}^{n+1}\|^2\bigr)
\notag\\
&\qquad
\le
\frac12\|u_h^0\|^2 + 2(I_1 + I_2 + I_3),
\end{align}

Let $\epsilon_i>0$, $i=1,\dots,6$, be arbitrary. By hypothesis {\rm(iii)}, Young's inequality, and the Cauchy-Schwarz inequality,
\begin{align*}
2I_1
&\le
\frac{k}{\epsilon_1}\sum_{n=0}^{m}\|u_h^{n+1}\|^2
+\epsilon_1 k\sum_{n=0}^{m}\|\psi(\cdot,t_n,u_h^n,v_{1,h}^n,v_{2,h}^n)\|^2 \\
&\le
Ck(1 + \| u_h^{0}\|^{2} + \| v_{1,h}^{0}\|^{2} + \| v_{2,h}^{0}\|^{2}) + Ck\sum_{n=0}^{m}\|u_h^{n+1}\|^2
+\epsilon_1 B_3^2 k\sum_{n=0}^{m}\bigl(\|v_{1,h}^{n+1}\|^2+\|v_{2,h}^{n+1}\|^2\bigr), \\
2I_2
&\le
\frac{1}{\epsilon_2}\sum_{n=0}^{m}\|u_h^{n+1}-u_h^n\|^2
+\epsilon_2\sum_{n=0}^{m}\|g(\cdot,t_n,u_h^n,v_{1,h}^n,v_{2,h}^n)\Delta W_n\|^2.
\end{align*}
Substituting these bounds into \eqref{ineq:stability-main-2D} leads to
\begin{align*}
&\frac12\max_{0\le l\le m}\|u_h^{l+1}\|^2
+\Bigl(\frac12-\frac{1}{\epsilon_2}\Bigr)\sum_{n=0}^{m}\|u_h^{n+1}-u_h^n\|^2
+\bigl(\alpha-\epsilon_1B_3^2\bigr)k\sum_{n=0}^{m}\bigl(\|v_{1,h}^{n+1}\|^2+\|v_{2,h}^{n+1}\|^2\bigr)
\notag\\
&\le
C\bigl(1+\|u_h^0\|^2+k\|v_{1,h}^0\|^2+k\|v_{2,h}^0\|^2\bigr)
+Ck\sum_{n=0}^{m}\|u_h^{n+1}\|^2
\notag\\
&\qquad
+\epsilon_2\sum_{n=0}^{m}\|g(\cdot,t_n,u_h^n,v_{1,h}^n,v_{2,h}^n)\Delta W_n\|^2
+2I_3.
\end{align*}
Taking the $q$-th power, applying Lemma \ref{lem:convex-ineq}, and then taking the expectation yield
\begin{align}
\label{ineq:stability-main-expect-2D}
    \frac{1}{2^q} & \mathbb{E}\left[\max_{0 \leq l \leq m} \| u_h^{l+1}\|^{2q} \right] + \left(\frac{1}{2} - \frac{1}{\epsilon_2}\right)^q \mathbb{E}\left[ \left(\sum_{n=0}^{m}  \| u_h^{n+1} - u_h^n\|^2 \right)^q \right] \notag \\
    & \quad + \left(\alpha - \epsilon_1 B_3^2 \right)^q \mathbb{E}\left[ \left(k\sum_{n=0}^{m}  \left(\| v_{1,h}^{n+1}\|^2 +\| v_{2,h}^{n+1}\|^2 \right) \right)^q \right] \notag \\
    & \leq C(1 + \| u_h^{0}\|^{2q} + k^q \| v_{1,h}^{0}\|^{2q} + k^q \| v_{2,h}^{0}\|^{2q}) + Ck\sum_{n=0}^{m} \mathbb{E}\left[ \max_{0 \leq l \leq n} \| u_h^{l+1}\|^{2q} \right]  \notag \\
    &\quad + (2^{q-1}+\epsilon_3)\epsilon_2^q \mathbb{E}\left[ \left(\sum_{n=0}^{m}  \| g(\cdot,t_n,u_h^n,v_{1,h}^n,v_{2,h}^n) \Delta W_n\|^2 \right)^q \right] \notag \\
    & \quad + (2^{q-1}+\epsilon_3) 2^q\mathbb{E}\left[ \max_{0 \leq l \leq m} \left| \sum_{n=0}^{l}  (g(\cdot,t_n,u_h^n,v_{1,h}^n,v_{2,h}^n) \Delta W_n,u_h^{n}) \right|^q \right].
\end{align}
% where we use Eq. \eqref{ineq:convex-3} with $\epsilon = \epsilon_3$, and the discrete H\"older inequality.

\noindent {\it Step 2.} Next, we proceed to estimate the last two terms of the right-hand side of \eqref{ineq:stability-main-expect-2D}. By the discrete and continuous BDG (see Remark \ref{rmk:BDG}), we have
\begin{align}
\label{ineq:stability-I1-2D}
    %\mathcal{I}_1 &:= 
    \mathbb{E}&\left[\left(\sum_{n=0}^{m} \| g(\cdot,t_n,u_h^n,v_{1,h}^n,v_{2,h}^n) \Delta W_n\|^2 \right)^q \right] \leq C_{b}'\mathbb{E}\left[ \max_{0 \leq l \leq m} \left\| \sum_{n=0}^{l} g(\cdot,t_n,u_h^n,v_{1,h}^n,v_{2,h}^n) \Delta W_n \right\|^{2q} \right]  \notag \\
    &\hspace{2cm}
    \leq C_b'C_b K^q \mathbb{E}\left[\left( k\sum_{n=0}^{m} \| g(\cdot,t_n,u_h^n,v_{1,h}^n,v_{2,h}^n)\|^2 \right)^q \right] \notag \\
    % &\hspace{1cm}
    % \leq C_b^2 K^q \mathbb{E}\left[ \left( C+ k\sum_{n=0}^{m} C_3^2\| u_h^n\|^2 + kC_4^2 \sum_{n=0}^{m} (\| v_{1,h}^n\|^2 + \| v_{2,h}^n\|^2) \right)^q \right] \notag \\
    &\hspace{2cm}
    \leq C(1 + \| u_h^{0}\|^{2q} + k^q \| v_{1,h}^{0}\|^{2q} + k^q \| v_{2,h}^{0}\|^{2q}) 
    + Ck\sum_{n=0}^{m}\mathbb{E}\left[ \| u_h^{n+1}\|^{2q} \right] \notag \\
    &\hspace{2.6cm}
    + C_b'C_{b} K^qD_4^{2q}(1+\epsilon_4) \mathbb{E}\left[\left(k\sum_{n=0}^{m}  \left(\| v_{1,h}^{n+1}\|^2 +\| v_{2,h}^{n+1}\|^2 \right) \right)^q \right], 
\end{align}
where the last inequality follows from hypothesis {\rm(iv)}, the discrete H\"older inequality and Lemma \ref{lem:convex-ineq}. By the BDG inequality, the Cauchy-Schwarz inequality, discrete H\"older inequality, hypothesis {\rm(iv)} and Lemma \ref{lem:convex-ineq},
\begin{align}
\label{ineq:stability-I2-2D}
\mathbb{E}[I_3^q]
&\le
C_b\mathbb{E}\left[\left(kK\sum_{n=0}^{m}\|g(\cdot,t_n,u_h^n,v_{1,h}^n,v_{2,h}^n)\|^2\|u_h^n\|^2\right)^{q/2}\right]
\notag\\
&\leq C_b \mathbb{E}\left[ \max_{0 \leq n \leq m}  \|u_h^{n}\|^{q} \left( kK\sum_{n=0}^{m}  \|g(\cdot,t_n,u_h^n,v_{1,h}^n,v_{2,h}^n) \|^2 \right)^{q/2} \right] \notag \\
&\le
\frac{1}{4\epsilon_5}\mathbb{E}\left[\max_{0\le n\le m}\|u_h^n\|^{2q}\right]
+\epsilon_5C_b^2K^q\mathbb{E}\left[\left(k\sum_{n=0}^{m}\|g(\cdot,t_n,u_h^n,v_{1,h}^n,v_{2,h}^n)\|^2\right)^q\right]
\notag\\
&\le
C\bigl(1+\|u_h^0\|^{2q}+k^q\|v_{1,h}^0\|^{2q}+k^q\|v_{2,h}^0\|^{2q}\bigr)
\notag\\
&\qquad
+Ck\sum_{n=0}^{m}\mathbb{E}\left[\|u_h^{n+1}\|^{2q}\right]
+\frac{1}{4\epsilon_5}\mathbb{E}\left[\max_{0\le n\le m}\|u_h^{n+1}\|^{2q}\right]
\notag\\
&\qquad
+\epsilon_5C_b^2K^qD_4^{2q}(1+\epsilon_6)
\mathbb{E}\left[\left(k\sum_{n=0}^{m}\bigl(\|v_{1,h}^{n+1}\|^2+\|v_{2,h}^{n+1}\|^2\bigr)\right)^q\right].
\end{align}
Substituting \eqref{ineq:stability-I1-2D} and \eqref{ineq:stability-I2-2D} into \eqref{ineq:stability-main-expect-2D}, we arrive at
\begin{align}
\label{ineq:high-moment-stability-beforefinal-2d}
\mathcal{C}_1&\mathbb{E}\left[\max_{0\le l\le m}\|u_h^{l+1}\|^{2q}\right]
+\mathcal{C}_2\mathbb{E}\left[\left(\sum_{n=0}^{m}\|u_h^{n+1}-u_h^n\|^2\right)^q\right]
+\mathcal{C}_3\mathbb{E}\left[\left(k\sum_{n=0}^{m}\bigl(\|v_{1,h}^{n+1}\|^2+\|v_{2,h}^{n+1}\|^2\bigr)\right)^q\right]
\notag\\
&\le
C\bigl(1+\|u_h^0\|^{2q}+k^q\|v_{1,h}^0\|^{2q}+k^q\|v_{2,h}^0\|^{2q}\bigr)
+Ck\sum_{n=0}^{m}\mathbb{E}\left[\max_{0\le l\le n}\|u_h^{l+1}\|^{2q}\right],
\end{align}
where these three coefficients are given by
\begin{align*}
\mathcal{C}_1
&=\frac{1}{2^q}-\frac{2^q(2^{q-1}+\epsilon_3)}{4\epsilon_5}, \qquad 
\mathcal{C}_2=\Bigl(\frac12-\frac{1}{\epsilon_2}\Bigr)^q,\\
\mathcal{C}_3
&=\bigl(\alpha-\epsilon_1B_3^2\bigr)^q
-(1+\epsilon_4)(2^{q-1}+\epsilon_3)\epsilon_2^q C_b'C_bK^qD_4^{2q}
-(1+\epsilon_6)(2^{q-1}+\epsilon_3)2^q\epsilon_5 C_b^2K^q D_4^{2q}.
\end{align*}
As shown in Appendix \ref{appendix-eps-choice}, the condition $\alpha>\alpha_0$ guarantees that one can choose fixed $\{\epsilon_i\}_{i=1}^6$ such that $\mathcal{C}_1$, $\mathcal{C}_2$, and $\mathcal{C}_3$ are all positive. Applying the discrete Gr\"onwall inequality to \eqref{ineq:high-moment-stability-beforefinal-2d}, then using Lemma \ref{lem:v_h^0-2d} and taking $m=N_T-1$, we obtain \eqref{ineq:high-moment-stability-2D}, which finishes the proof.
\end{proof}

In the following corollary, we show that the lower bound on $\alpha$ can be relaxed in the second-moment case. 
\begin{corollary}[Range of $\alpha$ for Second-Moment Stability Estimate]
\label{coro:stability-2nd-moment-2d}
If we only consider the case $q = 1$ in Theorem \ref{thm:stability-estimate-2D}, we can relax the stochastic parabolicity condition to $\alpha > KD_4^2/2$, which is sharp. More precisely, the second-moment estimate
\begin{align}
\label{ineq:2nd-moment-stability-strong-2d}
   \mathbb{E}\left[  \| \mathbf{u_h} \|_{\mathcal{S}}^{2} \right] + \mathbb{E}\left[ \langle \mathbf{u_h} \rangle \right] + \mathbb{E}\left[ \| \mathbf{v_{1,h}} \|_{\ell^2}^{2} + \| \mathbf{v_{2,h}} \|_{\ell^2}^{2} \right] \leq C\left(1+\| u_0\|^{2}+ k\| u_0\|_{1}^{2} \right).
    % \mathbb{E}\left[ \max_{0\leq n \leq N_T} \| u_h^{n+1}\|^{2} \right] &+ \sum_{n=0}^{N_T} \mathbb{E}\left[\| u_h^{n+1} - u_h^n\|^2 \right] + k\sum_{n=0}^{N_T}\mathbb{E}\left[ \| v_{1,h}^{n+1}\|^2 + \| v_{2,h}^{n+1}\|^2 \right] \notag \\
    % &\leq C\left(1+\| u_h^0\|^{2} + k \| v_{1,h}^0\|^{2} + k\| v_{2,h}^0\|^{2}\right),
\end{align}
holds for any $\alpha > KD_4^2/2$. 
\end{corollary}

\begin{proof}
The key idea to obtain the relaxed bound on $\alpha$ is to first show the weaker stability estimate
\begin{align}
\label{ineq:stability-2nd-moment-weaker-2d}
    \max_{0\leq n \leq N_T-1} \mathbb{E}\left[ \| u_h^{n+1}\|^2 \right] + \mathbb{E}\left[ \langle \mathbf{u_h} \rangle \right] + \mathbb{E}\left[ \| \mathbf{v_{1,h}} \|_{\ell^2}^{2} + \| \mathbf{v_{2,h}} \|_{\ell^2}^{2} \right] \leq C\left(1+\| u_0\|^{2}+ k\| u_0\|_{1}^{2} \right).
    % \max_{0\leq n \leq N_T} \mathbb{E}\left[ \| u_h^{n+1}\|^2 \right] &+ \sum_{n=0}^{N_T} \mathbb{E}\left[  \| u_h^{n+1} - u_h^n\|^2 \right] + k\sum_{n=0}^{N_T} \mathbb{E}\left[  \| v_{1,h}^{n+1}\|^2 + \| v_{2,h}^{n+1}\|^2 \right] \notag \\
    % &\leq C(1+\| u_h^0\|^2 + k\| v_{1,h}^0\|^2 + k\| v_{2,h}^0\|^2).
\end{align}
Starting from \eqref{eq:stability-main-2d} and the left-hand side bound \eqref{ineq:LHS-bound-2d}, for $m=1,\ldots,N_T-1$, we sum from $n=0$ to $m$, and then take the expectation to get
\begin{align}
\label{ineq:2nd-moment-stability-expect-2D}
    \frac{1}{2} &\mathbb{E}\left[ \| u_h^{m+1}\|^2 \right]+ \frac{1}{2} \sum_{n=0}^{m} \mathbb{E}\left[ \| u_h^{n+1} - u_h^n\|^2 \right] + \alpha k \sum_{n=0}^{m} \mathbb{E}\left[ \| v_{1,h}^{n+1}\|^2 + \| v_{2,h}^{n+1}\|^2 \right] -\frac{1}{2}\| u_h^{0}\|^2  \notag \\
    &\leq k \sum_{n=0}^{m} \mathbb{E}\left[ \left(\psi(\cdot,t_n,u_h^n,v_{1,h}^n,v_{2,h}^n), u_h^{n+1} \right) \right] + \sum_{n=0}^{m} \mathbb{E}\left[ \left(g(\cdot,t_n,u_h^n,v_{1,h}^n,v_{2,h}^n) \Delta W_n,u_h^{n+1}-u_h^n \right) \right]  \\
    &:= T_1 + T_2, \notag
\end{align}
where $\sum_{n=0}^{m} \mathbb{E}\left[ \left(g(\cdot,t_n,u_h^n,v_{1,h}^n,v_{2,h}^n) \Delta W_n,u_h^{n} \right)\right] = 0$ by the martingale property of the It\^o integral. Similar to the estimate of $I_1,I_2$ in the proof of Theorem \ref{thm:stability-estimate-2D}, we have
\begin{align*}
    T_1 &\leq C(1+\| u_h^0\|^2 + k\| v_{1,h}^0\|^2 + k\| v_{2,h}^0\|^2) + Ck\sum_{n=0}^{m} \mathbb{E}\left[ \| u_h^{n+1}\|^2 \right] + \epsilon_1 B_3^2 k \sum_{n=0}^{m} \mathbb{E}\left[ \| v_{1,h}^{n+1}\|^2 + \| v_{2,h}^{n+1}\|^2 \right], \\
    T_2 &\leq \frac{1}{4\epsilon_2} \sum_{n=0}^{m} \mathbb{E}\left[ \| u_h^{n+1} - u_h^n\|^2 \right] + \epsilon_2\sum_{n=0}^{m} \mathbb{E}\left[ \| g(\cdot,t_n,u_h^n,v_{1,h}^n,v_{2,h}^n) \Delta W_n\|^2 \right] \\
    &\leq \frac{1}{4\epsilon_2} \sum_{n=0}^{m} \mathbb{E}\left[ \| u_h^{n+1} - u_h^n\|^2 \right] + \epsilon_2 Kk \sum_{n=0}^{m} \mathbb{E}\left[ \| g(\cdot,t_n,u_h^n,v_{1,h}^n,v_{2,h}^n)\|^2 \right] \\
    &\leq \frac{1}{4\epsilon_2} \sum_{n=0}^{m} \mathbb{E}\left[ \| u_h^{n+1} - u_h^n\|^2 \right] +
    Ck\sum_{n=0}^{m} \mathbb{E}\left[ \| u_h^{n+1}\|^2 \right] + \epsilon_2 KD_4^2 k\sum_{n=0}^{m} \mathbb{E}\left[ \| v_{1,h}^{n+1}\|^2 + \| v_{2,h}^{n+1}\|^2 \right] \\
    &\quad + C(1+\| u_h^0\|^2 + k\| v_{1,h}^0\|^2 + k\| v_{2,h}^0\|^2),
\end{align*}
where the estimate for $T_2$ follows from the Cauchy-Schwarz inequality, It\^o isometry and hypothesis {\rm(iv)}. Utilizing the above estimates in \eqref{ineq:2nd-moment-stability-expect-2D}, we get
\begin{align*}
    \frac{1}{2} \mathbb{E}&\left[ \| u_h^{m+1}\|^2 \right]+ \left(\frac{1}{2} - \frac{1}{4\epsilon_2} \right) \sum_{n=0}^{m} \mathbb{E}\left[ \| u_h^{n+1} - u_h^n\|^2 \right] + \left(\alpha-\epsilon_1 B_3^2 - \epsilon_2 KD_4^2 \right) k \sum_{n=0}^{m} \mathbb{E}\left[ \| v_{1,h}^{n+1}\|^2 + \| v_{2,h}^{n+1}\|^2 \right] \\
    &\leq C(1+\| u_h^0\|^2 + k\| v_{1,h}^0\|^2 + k\| v_{2,h}^0\|^2) + Ck\sum_{n=0}^{m} \mathbb{E}\left[ \| u_h^{n+1}\|^2 \right]. 
\end{align*}
Choose $\epsilon_2>1/2$ and $\epsilon_1>0$ sufficiently small. Since $\alpha>KD_4^2/2$, the coefficient $\alpha-\epsilon_1 B_3^2 - \epsilon_2 KD_4^2$ can be made positive. Applying the discrete Gr\"onwall inequality, and then taking $\max_{0 \leq m \leq N_T-1}$ yield \eqref{ineq:stability-2nd-moment-weaker-2d}. In particular, we have 
\[ k\sum_{n=0}^{N_T-1} \mathbb{E}\left[  \| v_{1,h}^{n+1}\|^2 + \| v_{2,h}^{n+1}\|^2 \right] 
    \leq C(1+\| u_h^0\|^2 + k\| v_{1,h}^0\|^2 + k\| v_{2,h}^0\|^2). \]
    
We now return to \eqref{ineq:high-moment-stability-beforefinal-2d} with $q=1$.
Even though the coefficient of the term involving
\[
k\sum_{n=0}^{m} \left(\| v_{1,h}^{n+1}\|^2 + \| v_{2,h}^{n+1}\|^2 \right)
\]
need not be positive under the weaker condition $\alpha>KD_4^2/2$, that term is already controlled by \eqref{ineq:stability-2nd-moment-weaker-2d}. Hence it can be absorbed into the right-hand side, and we obtain
\begin{align*}
    &\mathbb{E}\left[\max_{0 \leq l \leq m} \| u_h^{l+1}\|^{2} \right]
    + \mathbb{E}\left[\sum_{n=0}^{m}  \| u_h^{n+1} - u_h^n\|^2 \right] \\
    &\qquad \leq C(1 + \| u_h^{0}\|^{2} + k\| v_{1,h}^0\|^2 + k\| v_{2,h}^0\|^2)
    + Ck\sum_{n=0}^{m}\mathbb{E}\left[ \max_{0 \leq l \leq n} \| u_h^{l+1}\|^{2} \right].
\end{align*}
A second application of the discrete Gr\"onwall inequality, together with Lemma \ref{lem:v_h^0-2d}, gives \eqref{ineq:2nd-moment-stability-strong-2d}.
\end{proof}

\begin{remark}
\label{rmk:stability-degen-2d}
    As a special case, suppose that the model \eqref{spde:conv-diff-2d} is driven by the standard Brownian motion, so that $K=1$. If $B_3 = 0$, the second-moment stability estimate \eqref{ineq:2nd-moment-stability-strong-2d} holds in the degenerate case $\alpha =D_4^2/2$. 
\end{remark}

\end{section}

\begin{section}{Optimal Error Estimate of the Fully-Discrete Method}
\label{sec:error-estimate-2d}
In this section, we show the high-moment convergence estimates for the fully-discrete numerical scheme \eqref{eq:fully-discrete-1-2d}-\eqref{eq:fully-discrete-4-2d} applied to the strong solution with sufficient regularity. The strong convergence orders in space and time are proved to be the theoretical optimal rates $\mathcal{O}(h^{r+1})$ and $\mathcal{O}(k^{\frac{1}{2}})$, respectively. 

Let $u$ be the variational strong solution of \eqref{spde:conv-diff-2d}, and let $v_1, v_2, w_1, w_2$ be defined as in \eqref{eq:auxiliary}. Let $u^n_h$ be the numerical solution of the LDG-IMEX-Euler method \eqref{eq:fully-discrete-1-2d}--\eqref{eq:fully-discrete-4-2d}. We denote by $e_u^n, e_{v_i}^n, e_{w_i}^n$ ($i=1,2$) the error between the exact and numerical solutions at time $t_n$, and $\mathbf{e_{u}}=\{ e_u^n \}^{N_T}_{n=0}$, $\mathbf{e_{v_i}}=\{ e_{v_i}^n \}^{N_T}_{n=0}$ the discrete processes of the numerical error. In addition to hypotheses {\rm(i)}-{\rm(iv)}, we make the following assumptions on the leading matrix $A=\{a_{ij}\}$ and the stochastic diffusion term $g$. 
%, where hypothesis (v) follows from the hypothesis in Lemma \ref{lem:Holder-High-Moments}, and $\{a_{ij}\}$ may depend on the variables $\{ x,y,t\}$ as opposed to being constant in \cite{LiShuTang2021_ESAIM}.
\begin{enumerate}
    \item[(v)] For any $(\omega,t) \in \Omega \times [0,T]$, 
    \begin{align*}
         \| g(\omega,\cdot,t,u,v_1,v_2)\|_{\mathcal{L}_2^2} \leq C\| g(\omega,\cdot,t,u,v_1,v_2)\|_{2} \leq C\left(1 + \| u\|_{2} + \| v_1\|_{2} + \| v_2\|_{2} \right).
    \end{align*}

    \item[(vi)] The leading coefficient matrix $A = \{a_{ij}\}$ is independent of $u$. Moreover, for any $(\omega,x,y,t,s) \in \Omega \times [0, 2\pi]^2 \times [0,T]^2$, and $j=1,2$, we have
    \begin{align*}
        &|(a_{ij})_x(x,y,t)| + |(a_{ij})_y(x,y,t)|\leq C, \qquad \left|a_{ij}(x,y,t) - a_{ij}(x,y,s) \right| \leq C|t-s|^{1/2}, \\
        &\left|(a_{ij})_x(x,y,t) - (a_{ij})_x(x,y,s) \right| + \left|(a_{ij})_y(x,y,t) - (a_{ij})_y(x,y,s) \right| \leq C|t-s|^{1/2}.
    \end{align*}
\end{enumerate}
Assumption (v) is the growth condition used in Lemma \ref{lem:Holder-High-Moments}, while assumption (vi) specifies the semilinear setting considered in the error analysis. In particular, unlike \cite{LiShuTang2021_ESAIM}, the coefficients $\{a_{ij}\}$ are allowed to depend on the spatial and temporal variables $(x,y,t)$.

The following lemma is used to approximate the individual terms $\mathcal{I}_1$ defined in Theorem \ref{thm:high-moment-error-estimate-2D}. The proof is provided in Appendix \ref{appendix:term1-2}.

\begin{lemma}
\label{lem:error2D-term1-2}
Assume that hypotheses {\rm(iii)} and {\rm(iv)} hold. For any $0\le m\le N_T-1$ and any $\epsilon_0, \epsilon_1 > 0$, we have
\begin{align}
\label{ineq:error2D-psi}
    &\sum_{n=0}^{m}  \int_{t_n}^{t_{n+1}} \| \psi(\cdot,t,u,v_1,v_2) - \psi(\cdot,t_{n},u_h^{n},v_{1,h}^n,v_{2,h}^n) \|^2 \, \mathrm{d}t  \notag \\
    &\leq C\sum_{n=0}^{m} \int_{t_n}^{t_{n+1}} \|u-u^n\|_{1}^2 \, \mathrm{d}t + Ck^2\sum_{n=0}^{m}(1+\| u^n\|_{1}^2) + (2+\epsilon_0) B_1^2 k\sum_{n=0}^{m}  (\| \xi_{v_1}^{n}\|^2 +  \| \xi_{v_2}^{n}\|^2)  \notag \\
    &\quad + Ck \sum_{n=0}^{m}  \| \xi_u^{n}\|^2 + Ch^{2r+2} k \sum_{n=0}^{m} (\| u^{n}\|_{{r+1}}^2 + \| v_1^{n}\|_{{r+1}}^2 + \| v_2^{n}\|_{{r+1}}^2). 
\end{align}
Similarly, we have
\begin{align}
\label{ineq:error2D-g}
    &\sum_{n=0}^{m}  \int_{t_n}^{t_{n+1}} \|g(\cdot,t,u,v_1,v_2) - g(\cdot,t_{n},u_h^{n},v_{1,h}^n,v_{2,h}^n) \|^2 \, \mathrm{d}t  \notag \\
    &\leq C\sum_{n=0}^{m} \int_{t_n}^{t_{n+1}} \|u-u^n\|_{1}^2 \, \mathrm{d}t + Ck^2\sum_{n=0}^{m}(1+\| u^n\|_{1}^2) + (2+\epsilon_1) D_2^2 k\sum_{n=0}^{m}  (\| \xi_{v_1}^{n}\|^2 +  \| \xi_{v_2}^{n}\|^2) \notag \\
    &+ Ck \sum_{n=0}^{m}  \| \xi_u^{n}\|^2 + Ch^{2r+2} k \sum_{n=0}^{m} (\| u^{n}\|_{{r+1}}^2 + \| v_1^{n}\|_{{r+1}}^2 + \| v_2^{n}\|_{{r+1}}^2). 
\end{align}
\end{lemma}

% \begin{lemma}
% \label{lem:error2D-term3}
% Assuming that hypothesis (ii), (vi) hold, for any $0\le m\le N_T-1$, we have
% \begin{align}
% \label{ineq:error-estimate-term3-2D}
%     \mathcal{I}_3 &:= \max_{0 \leq l \leq m} \sum_{n=0}^{l} \int_{t_n}^{t_{n+1}}  H^+(w_1-w_{1}^{n+1},w_2-w_{2}^{n+1},\xi_u^{n+1}) \, dt  \notag \\
%     &\leq k\sum_{n=0}^{m}  \|\xi_u^{n+1}\|^2 + C\sum_{n=0}^{m} \int_{t_n}^{t_{n+1}} \|u(t)-u^{n+1}\|_{H^2}^2 \, dt + Ck^2\sum_{n=0}^{m} (\|v_1^{n+1}\|_{H^1}^2 + \|v_2^{n+1}\|_{H^1}^2).
% \end{align}
% \end{lemma}

Next, we propose a lemma which establishes a relationship between $\| \xi_{v_i}^n\|$ and $\| \xi_{w_i}^n\|$. This will be used to estimate the term $\mathcal{I}_4$ in the proof of Theorem \ref{thm:high-moment-error-estimate-2D}. The proof is provided in Appendix \ref{appendix-w_h-v_h-2d}. 
\begin{lemma}
\label{lem:w_h-v_h-2D}
Assuming that hypothesis {\rm(ii)} holds, for any $0\le m\le N_T-1$, we have 
\begin{align}
\label{ineq:wh-vh-2d}
    k\sum_{n=0}^m (\| \xi_{w_1}^{n+1}\|^2 &+ \| \xi_{w_2}^{n+1}\|^2) \leq 10\Lambda k\sum_{n=0}^m  (\| \xi_{v_1}^{n+1}\|^2 + \| \xi_{v_2}^{n+1}\|^2) \notag  \\
    &+ Ch^{2r+2} k\sum_{n=0}^m \left(\| u^{n+1}\|_{r+1}^2 + \| w_1^{n+1}\|_{r+1}^2 + \| w_2^{n+1}\|_{r+1}^2 + \| v_1^{n+1}\|_{r+1}^2 + \| v_2^{n+1}\|_{r+1}^2 \right). 
\end{align}
\end{lemma}

Lastly, before presenting the main error analysis, we show the high-order approximation to the initial gradient of the solution $u$, provided the initial condition $u_0$ is sufficiently smooth. The proof is provided in Appendix \ref{appendix:xi_v^0-2d}.  
\begin{lemma}
\label{lem:xi_v^0-2d}
Assume $u_0 \in H^{r+2}([0,2\pi]^2)$ and take the initial data $u_h^0$ of the numerical schemes to be any
suitable projections of exact solution satisfying the standard approximation \eqref{ineq:proj-property-1-2d}. 
Let $\xi_{u}^0 := \mathcal{P}^- u_0 - u_{h}^0$, $\xi_{v_i}^0 := \mathcal{P} v_i^0 - v_{i,h}^0$ for $i=1,2$, where $\mathcal{P}^-, \mathcal{P}$ are the two-dimensional Gauss-Radau and $L^2$-projections. Then we have
\begin{align*}
    \| \xi_{u}^0\|^2 + \| \xi_{v_1}^0\|^2 + \| \xi_{v_2}^0\|^2 \leq Ch^{2r+2}. 
\end{align*}
\end{lemma}

Next, we give the optimal convergence theorem for the fully discrete scheme. The theorem hinges on the following regularity assumption ($\mathcal{A}_q$) for the exact solution, and a bound on the stochastic parabolic constant. 

\begin{enumerate}
    \item[($\mathcal{A}_q$)] For any fixed $q\in[1,\infty)$, suppose $u_0 \in H^{r+2}$ and 
    \begin{align*}
        &\alpha > \alpha_1 := \left(2^{3q-1}C_b' + 2^{6q-4}C_b \right)^{1/q}C_b^{1/q}D_2^2K,  \qquad u \in L^{2q}(\Omega, L^{\infty}[0,T; H^{r+3}]), \\
        &w_1,w_2 \in L^{2q}(\Omega \times[0,T]; H^{r+2}), \qquad \{\psi,g\}(\cdot,u,\nabla u) \in L^{2q}(\Omega \times[0,T]; H^{r+1}). 
    \end{align*}
\end{enumerate}

% For arbitrarily fixed $q\in[1,\infty)$, suppose $u_0 \in H^{r+2}$, $u \in L^{2q}(\Omega; L^{\infty}[0,T; H^{r+3}])$, $w_1,w_2 \in L^{2q}(\Omega \times[0,T]; H^{r+2})$, $\{\psi,g\}(\cdot,u,\nabla u) \in L^{2q}(\Omega \times[0,T]; H^{r+1})$, as well as $\alpha > \alpha_1 := \left(1 + 2^{3q-3}  \right)^{1/q}C_b^{2/q}D_2^2K$. 

We have the following optimal error estimates for the fully discrete scheme.
\begin{theorem}[Fully Discrete Optimal Error Estimate]
\label{thm:high-moment-error-estimate-2D}
    Assume hypotheses {\rm(i)}-{\rm(vi)} and ($\mathcal{A}_q$). Then there exists a positive constant $C$, independent of $h$ and $k$, such that
\begin{align}
\label{ineq:high-moment-error-estimate-2D}
    \mathbb{E}\left[  \| \mathbf{e_{u}} \|_{\mathcal{S}}^{2q} \right]^{\frac{1}{2q}} + \mathbb{E}\left[ \langle \mathbf{e_{u}} \rangle^{q} \right]^{\frac{1}{2q}} + \mathbb{E}\left[  \left(\| \mathbf{e_{v_1}} \|_{\ell^2}^{2} + \| \mathbf{e_{v_2}} \|_{\ell^2}^{2} \right)^q \right]^{\frac{1}{2q}} \leq C(k^{1/2} + h^{r+1}). 
    % \mathbb{E}\left[\max_{0 \leq n \leq N_T} \| e_u^{n+1}\|^{2q} \right]^{\frac{1}{2q}} + \mathbb{E}\left[ \left(k\sum_{n=0}^{N_T} (\|e_{v_1}^{n+1}\|^2 + \|e_{v_2}^{n+1}\|^2) \right)^q \right]^{\frac{1}{2q}} \leq C(k^{\frac{1}{2}} + h^{r+1}).
\end{align} 
\end{theorem}

\begin{proof} 
We divide the long proof into two steps. First, recall that
\[
e_u^n=\xi_u^n-\eta_u^n, \qquad
e_{v_i}^n=\xi_{v_i}^n-\eta_{v_i}^n, \qquad
e_{w_i}^n=\xi_{w_i}^n-\eta_{w_i}^n, \qquad i=1,2.
\]
Using the two-dimensional Gauss-Radau projections $\mathcal{P}^{\pm}$ and $L^2$-projection $\mathcal{P}$ defined in Section \ref{subsec:prelim-2d}, we can decompose the numerical error $e_u^n$ into two terms, by setting $\xi_u^n = \mathcal{P}^-u^n - u_h^n$, and $\eta_u^n = \mathcal{P}^-u^n - u^n$, where $\eta_u^n$ represents the projection error. Similarly, we define $\xi_{v_i}^n = \mathcal{P} v_i^n - v_{i,h}^n$, $\eta_{v_i}^n = \mathcal{P} v_i^n - v_i^n$ and $\xi_{w_i}^n = \mathcal{P}^+ w_i^n - w_{i,h}^n$, $\eta_{w_i}^n = \mathcal{P}^+ w_i^n - w_i^n$. 

\noindent {\it Step 1.} Let $a_{ij}^{n+1} = a_{ij}(\cdot,t_{n+1})$ and $A^{n+1} = A(\cdot,t_{n+1})$. Subtracting the fully discrete scheme \eqref{eq:fully-discrete-1-2d}-\eqref{eq:fully-discrete-4-2d} from the corresponding identities satisfied by the exact solution, and summing over all cells $I_i \times J_j$, we get the following error equations
\begin{align}
    (e_u^{n+1},r_h) &= (e_u^{n},r_h) + \int_{t_n}^{t_{n+1}} H^+(w_1-w_{1,h}^{n+1},w_2-w_{2,h}^{n+1},r_h) \, \mathrm{d}t \notag \\
    &\quad + \left( \int_{t_n}^{t_{n+1}} \psi(x,y,t,u,\nabla u) - \psi(x,y,t_{n},u_h^{n},v_{1,h}^n,v_{2,h}^n) \, \mathrm{d}t, r_h \right) \notag \\
    &\quad + \left( \int_{t_n}^{t_{n+1}} g(x,y,t,u,\nabla u) - g(x,y,t_{n},u_h^{n},v_{1,h}^n,v_{2,h}^n) \, \mathrm{d}W_t, r_h \right). \label{eq:error-1-2d} \\
    (e_{v_1}^{n+1},p_h) &+ (e_{v_2}^{n+1},q_h) = L^-(e_u^{n+1},p_h,q_h). \label{eq:error-2-2d} \\
    (e_{w_1}^{n+1},z_h) &= \big(a_{11}^{n+1} e_{v_1}^{n+1},z_h \big) + \big(a_{12}^{n+1} e_{v_2}^{n+1}, z_h \big), \quad (e_{w_2}^{n+1},\phi_h) = \big(a_{21}^{n+1}e_{v_1}^{n+1}, \phi_h \big) + \big(a_{22}^{n+1}e_{v_2}^{n+1} , \phi_h \big). \label{eq:error-3-2d} 
    % \\ (e_{w_2}^{n+1},\phi_h) &= \big(a_{21}^{n+1}e_{v_1}^{n+1}, \phi_h \big) + \big(a_{22}^{n+1}e_{v_1}^{n+1} , \phi_h \big).  \label{eq:error-4-2d}
\end{align}

Taking test functions $p_h = \xi_{w_1}^{n+1}$, $q_h = \xi_{w_2}^{n+1}$, $z_h = \xi_{v_1}^{n+1}$, $\phi_h = \xi_{v_2}^{n+1}$ in \eqref{eq:error-2-2d}-\eqref{eq:error-3-2d}, and using the definition of $L^2$ projection, we get
\begin{align*}
(\xi_{v_1}^{n+1},\xi_{w_1}^{n+1}) + (\xi_{v_2}^{n+1},\xi_{w_2}^{n+1})
&= (\eta_{v_1}^{n+1},\xi_{w_1}^{n+1})+ (\eta_{v_2}^{n+1},\xi_{w_2}^{n+1}) + L^-(e_u^{n+1},\xi_{w_1}^{n+1},\xi_{w_2}^{n+1}) \\
&= L^-(\xi_u^{n+1},\xi_{w_1}^{n+1},\xi_{w_2}^{n+1})
   - L^-(\eta_u^{n+1},\xi_{w_1}^{n+1},\xi_{w_2}^{n+1}),
\end{align*}
and
\begin{align*}
(\xi_{w_1}^{n+1},\xi_{v_1}^{n+1}) + (\xi_{w_2}^{n+1},\xi_{v_2}^{n+1})
&= (\eta_{w_1}^{n+1},\xi_{v_1}^{n+1}) + (\eta_{w_2}^{n+1},\xi_{v_2}^{n+1}) \\
&\quad + \bigl(a_{11}^{n+1}\xi_{v_1}^{n+1}-a_{11}^{n+1}\eta_{v_1}^{n+1},\xi_{v_1}^{n+1}\bigr)
      + \bigl(a_{12}^{n+1}\xi_{v_2}^{n+1}-a_{12}^{n+1}\eta_{v_2}^{n+1},\xi_{v_1}^{n+1}\bigr) \\
&\quad + \bigl(a_{21}^{n+1}\xi_{v_1}^{n+1}-a_{21}^{n+1}\eta_{v_1}^{n+1},\xi_{v_2}^{n+1}\bigr)
      + \bigl(a_{22}^{n+1}\xi_{v_2}^{n+1}-a_{22}^{n+1}\eta_{v_2}^{n+1},\xi_{v_2}^{n+1}\bigr) .
\end{align*}
Combining these two identities gives
\begin{align*}
    \bigl(A^{n+1}&(\xi_{v_1}^{n+1},\xi_{v_2}^{n+1})^{\top}, (\xi_{v_1}^{n+1},\xi_{v_2}^{n+1})^{\top}\bigr)
    = L^-(\xi_u^{n+1},\xi_{w_1}^{n+1},\xi_{w_2}^{n+1}) - L^-(\eta_u^{n+1},\xi_{w_1}^{n+1},\xi_{w_2}^{n+1}) \\ 
    &\quad - (\eta_{w_1}^{n+1},\xi_{v_1}^{n+1}) - (\eta_{w_2}^{n+1},\xi_{v_2}^{n+1})
    + \big(a_{11}^{n+1}\,\eta_{v_1}^{n+1} + a_{12}^{n+1}\,\eta_{v_2}^{n+1},\, \xi_{v_1}^{n+1} \big)
    + \big(a_{21}^{n+1}\,\eta_{v_1}^{n+1} + a_{22}^{n+1}\,\eta_{v_2}^{n+1},\, \xi_{v_2}^{n+1} \big).
\end{align*}
Multiplying the above equation by $k$, adding it to \eqref{eq:error-1-2d} with $r_h = \xi_u^{n+1}$, using the identity $L^-(\xi_u^{n+1},\xi_{w_1}^{n+1},\xi_{w_2}^{n+1}) + H^+(\xi_{w_1}^{n+1},\xi_{w_2}^{n+1},\xi_u^{n+1}) = 0$ from Lemma \ref{lem:num-flux-2D}, and summing over $n=0, \ldots, l$, we get
\begin{align*}
    &\sum_{n=0}^l(\xi_u^{n+1}-\xi_u^{n},\xi_u^{n+1}) + k\sum_{n=0}^l \bigl(A^{n+1}(\xi_{v_1}^{n+1},\xi_{v_2}^{n+1})^{\top}, (\xi_{v_1}^{n+1},\xi_{v_2}^{n+1})^{\top}\bigr) \\
    &= \sum_{n=0}^l(\eta_u^{n+1}-\eta_u^{n},\xi_u^{n+1}-\xi_u^n) + \sum_{n=0}^l(\eta_u^{n+1}-\eta_u^{n},\xi_u^{n}) 
    - k\sum_{n=0}^l  L^-(\eta_u^{n+1},\xi_{w_1}^{n+1},\xi_{w_2}^{n+1}) \\
    &\quad - k\sum_{n=0}^l  H^+(\eta_{w_1}^{n+1},\eta_{w_2}^{n+1},\xi_u^{n+1}) 
    - k\sum_{n=0}^l (\eta_{w_1}^{n+1},\xi_{v_1}^{n+1}) - k\sum_{n=0}^l  (\eta_{w_2}^{n+1},\xi_{v_2}^{n+1}) \\
    &\quad + k\sum_{n=0}^l  \big(a_{11}^{n+1}\,\eta_{v_1}^{n+1}+a_{12}^{n+1}\,\eta_{v_2}^{n+1},\, \xi_{v_1}^{n+1} \big)
    + k\sum_{n=0}^l  \big(a_{21}^{n+1}\,\eta_{v_1}^{n+1}+a_{22}^{n+1}\,\eta_{v_2}^{n+1},\, \xi_{v_2}^{n+1} \big) \\
    &\quad 
    + \sum_{n=0}^l \int_{t_n}^{t_{n+1}}  H^+(w_1-w_{1}^{n+1},w_2-w_{2}^{n+1},\xi_u^{n+1}) \, \mathrm{d}t \\
    &\quad + \sum_{n=0}^l  \left( \int_{t_n}^{t_{n+1}} \psi(x,y,t,u,v_1,v_2) - \psi(x,y,t_{n},u_h^{n},v_{1,h}^n,v_{2,h}^n) \, \mathrm{d}t, \xi_u^{n+1} \right) \\
    &\quad + \sum_{n=0}^l  \left( \int_{t_n}^{t_{n+1}} g(x,y,t,u,v_1,v_2) - g(x,y,t_{n},u_h^{n},v_{1,h}^n,v_{2,h}^n) \, \mathrm{d}W_t, \xi_u^{n+1}-\xi_u^n \right) \\
    &\quad + \sum_{n=0}^l  \left( \int_{t_n}^{t_{n+1}} g(x,y,t,u,v_1,v_2) - g(x,y,t_{n},u_h^{n},v_{1,h}^n,v_{2,h}^n) \, \mathrm{d}W_t, \xi_u^{n} \right).
\end{align*}

By hypothesis {\rm(ii)}, the left-hand side of the above equation is bounded from below by
\begin{align*}
    LHS &= \frac{1}{2}\sum_{n=0}^l (\| \xi_u^{n+1}\|^2 - \| \xi_u^{n}\|^2) + \frac{1}{2} \sum_{n=0}^l \| \xi_u^{n+1} - \xi_u^{n}\|^2 + k\sum_{n=0}^l \bigl(A^{n+1}(\xi_{v_1}^{n+1},\xi_{v_2}^{n+1})^{\top}, (\xi_{v_1}^{n+1},\xi_{v_2}^{n+1})^{\top}\bigr) \\
    &\geq \frac{1}{2}\| \xi_u^{l+1}\|^2 - \frac{1}{2}\| \xi_u^{0}\|^2 + \frac{1}{2} \sum_{n=0}^l \| \xi_u^{n+1} - \xi_u^{n}\|^2 + \alpha k\sum_{n=0}^l  (\| \xi_{v_1}^{n+1}\|^2 + \| \xi_{v_2}^{n+1}\|^2).
\end{align*}
Since $\|\xi_u^0\| + \|\xi_{v_1}^0\|+\|\xi_{v_2}^0\| \leq Ch^{r+1}$ by Lemma \ref{lem:xi_v^0-2d}, these three initial terms will not affect the error accuracy, and are omitted in the following proof. For any $m=0,\ldots,N_T-1$, taking $\max_{0\le l\le m}$ yields
\begin{align}
\label{ineq:error-main-2D}
    &\frac{1}{2} \max_{0 \leq n \leq m} \| \xi_u^{n+1}\|^2 + \frac{1}{2} \sum_{n=0}^{m} \| \xi_u^{n+1} - \xi_u^{n}\|^2 + \alpha k \sum_{n=0}^{m} (\| \xi_{v_1}^{n+1}\|^2 + \| \xi_{v_2}^{n+1}\|^2) \leq Ch^{2r+2}+2(\mathcal{I}_1 + \mathcal{I}_2 + \mathcal{I}_3 + \mathcal{I}_4),
\end{align}
where
\begin{align*}
    \mathcal{I}_1
    &:= \max_{0 \leq l \leq m} \sum_{n=0}^l(\eta_u^{n+1}-\eta_u^{n},\xi_u^{n+1}-\xi_u^n)
    + \max_{0 \leq l \leq m} k\sum_{n=0}^l (-\eta_{w_1}^{n+1},\xi_{v_1}^{n+1})
    + \max_{0 \leq l \leq m} k\sum_{n=0}^l (-\eta_{w_2}^{n+1},\xi_{v_2}^{n+1}) \\
    &\quad + \max_{0 \leq l \leq m} k\sum_{n=0}^l \bigl(a_{11}^{n+1}\eta_{v_1}^{n+1}+a_{12}^{n+1}\eta_{v_2}^{n+1}, \xi_{v_1}^{n+1} \bigr)
    + \max_{0 \leq l \leq m} k\sum_{n=0}^l \bigl(a_{21}^{n+1}\eta_{v_1}^{n+1}+a_{22}^{n+1}\eta_{v_2}^{n+1}, \xi_{v_2}^{n+1} \bigr) \\
    &\quad + \max_{0 \leq l \leq m} \sum_{n=0}^l \left( \int_{t_n}^{t_{n+1}} \bigl[\psi(x,y,t,u,v_1,v_2)-\psi(x,y,t_n,u_h^{n},v_{1,h}^n,v_{2,h}^n)\bigr]\,\mathrm{d}t,\, \xi_u^{n+1} \right) \\
    &\quad + \max_{0 \leq l \leq m} \sum_{n=0}^l \left( \int_{t_n}^{t_{n+1}} \bigl[g(x,y,t,u,v_1,v_2)-g(x,y,t_n,u_h^{n},v_{1,h}^n,v_{2,h}^n)\bigr]\,\mathrm{d}W_t,\, \xi_u^{n+1}-\xi_u^n \right), \\
    \mathcal{I}_2
    &:= \max_{0 \leq l \leq m} \sum_{n=0}^l(\eta_u^{n+1}-\eta_u^{n},\xi_u^n)
    + \max_{0 \leq l \leq m} \sum_{n=0}^l \left( \int_{t_n}^{t_{n+1}} \bigl[g(x,y,t,u,v_1,v_2)-g(x,y,t_n,u_h^{n},v_{1,h}^n,v_{2,h}^n)\bigr]\,\mathrm{d}W_t,\, \xi_u^n \right), \\
    \mathcal{I}_3
    &:= \max_{0 \leq l \leq m} \sum_{n=0}^l \int_{t_n}^{t_{n+1}} H^+(w_1-w_1^{n+1},w_2-w_2^{n+1},\xi_u^{n+1})\,\mathrm{d}t, \\
    \mathcal{I}_4
    &:= \max_{0 \leq l \leq m} k\sum_{n=0}^l -L^-(\eta_u^{n+1},\xi_{w_1}^{n+1},\xi_{w_2}^{n+1})
       + \max_{0 \leq l \leq m} k\sum_{n=0}^l -H^+(\eta_{w_1}^{n+1},\eta_{w_2}^{n+1},\xi_u^{n+1}).
\end{align*}
Let $ \{ \epsilon_i >0: i=0,1,\ldots,9\}$ be arbitrary constants. By the Cauchy-Schwarz inequality, Young's inequality and Lemma \ref{lem:proj-property-2d}, we have
\begin{align}
\label{ineq:error-estimate-term1-2D}
    2\mathcal{I}_1
    &\leq Ch^{2r+2}\sum_{n=0}^{m} \|u^{n+1}-u^n\|_{{r+1}}^2
    + \epsilon_2 \sum_{n=0}^{m} \|\xi_u^{n+1}-\xi_u^n\|^2
    + Ch^{2r+2} k\sum_{n=0}^{m} (\| w_1^{n+1} \|_{{r+1}}^2 + \| w_2^{n+1} \|_{{r+1}}^2) \notag\\
    &\quad + \epsilon_3 k \sum_{n=0}^{m} (\| \xi_{v_1}^{n+1}\|^2 + \| \xi_{v_2}^{n+1}\|^2)
    + Ch^{2r+2} k\sum_{n=0}^{m} (\| v_1^{n+1} \|_{{r+1}}^2 + \| v_2^{n+1} \|_{{r+1}}^2) \notag\\
    &\quad + \frac{k}{\epsilon_4}\sum_{n=0}^{m} \| \xi_u^{n+1}\|^2
    + \epsilon_4 \sum_{n=0}^{m} \int_{t_n}^{t_{n+1}} \| \psi(x,y,t,u,v_1,v_2) - \psi(x,y,t_n,u_h^{n},v_{1,h}^n,v_{2,h}^n) \|^2 \, \mathrm{d}t \notag\\
    &\quad + \frac{1}{\epsilon_5} \sum_{n=0}^{m} \|\xi_u^{n+1}-\xi_u^n\|^2
    + \epsilon_5 \sum_{n=0}^{m} \left\| \int_{t_n}^{t_{n+1}} \bigl[g(x,y,t,u,v_1,v_2)-g(x,y,t_n,u_h^{n},v_{1,h}^n,v_{2,h}^n)\bigr] \, \mathrm{d}W_t \right\|^2.
\end{align}
To estimate $\mathcal{I}_2$, we note that
\[
\eta_u^{n+1}-\eta_u^n
=
(\mathcal P^- - I)\left[
\int_{t_n}^{t_{n+1}} \bigl((w_1)_x+(w_2)_y+\psi\bigr)\,\mathrm{d}t
+
\int_{t_n}^{t_{n+1}} g\,\mathrm{d}W_t
\right].
\]
Hence, by H\"older's inequality and Lemma \ref{lem:proj-property-2d},
\begin{align}
\label{ineq:error-estimate-term2-2D}
    2\mathcal{I}_{2}
    &\leq k\sum_{n=0}^{m} \| \xi_u^{n}\|^2
    + Ch^{2r+2}\int_0^T \bigl(\| (w_1)_x\|_{{r+1}}^2 + \| (w_2)_y\|_{{r+1}}^2 + \| \psi\|_{{r+1}}^2\bigr)\,\mathrm{d}t \notag\\
    &\quad + 2\max_{0 \leq l \leq m} \sum_{n=0}^{l} \int_{t_n}^{t_{n+1}} (\mathcal{P}^-g -g, \xi_u^{n}) \, \mathrm{d}W_t \notag\\
    &\quad + 2\max_{0 \leq l \leq m} \sum_{n=0}^{l} \left( \int_{t_n}^{t_{n+1}} \bigl[g(x,y,t,u,v_1,v_2)-g(x,y,t_n,u_h^{n},v_{1,h}^n,v_{2,h}^n)\bigr] \, \mathrm{d}W_t, \xi_u^{n} \right).
\end{align}
The estimate for $\mathcal{I}_3$ is given as:
\begin{align}
\label{ineq:error-estimate-term3-2D}
    \mathcal{I}_3 &:= \max_{0 \leq l \leq m} \sum_{n=0}^{l} \int_{t_n}^{t_{n+1}}  H^+(w_1-w_{1}^{n+1},w_2-w_{2}^{n+1},\xi_u^{n+1}) \, \mathrm{d}t  \notag \\
    &\leq k\sum_{n=0}^{m}  \|\xi_u^{n+1}\|^2 + C\sum_{n=0}^{m} \int_{t_n}^{t_{n+1}} \|u(t)-u^{n+1}\|_{2}^2 \, \mathrm{d}t + Ck^2\sum_{n=0}^{m} (\|v_1^{n+1}\|_{1}^2 + \|v_2^{n+1}\|_{1}^2),
\end{align}
and we defer its proof to Appendix \ref{appendix:error-estimate-term3}.
% provided in \eqref{ineq:error-estimate-term3-2D} of Lemma \ref{lem:error2D-term3}. 
%Compared to the 1D error estimate in \cite{chen_semi-linear_1d, chen_non-linear_1d}, the only new term is $\mathcal{I}_4$. 
By Lemmas \ref{lem:superconvergence} and \ref{lem:w_h-v_h-2D}
\begin{align}
\label{ineq:error-estimate-term4-2D}
    \mathcal{I}_4 &\leq k\sum_{n=0}^m  |H^+(\eta_{w_1}^{n+1},\eta_{w_2}^{n+1},\xi_u^{n+1})| + k\sum_{n=0}^m  |L^-(\eta_u^{n+1},\xi_{w_1}^{n+1},\xi_{w_2}^{n+1})| \notag \\
    &\leq Ckh^{r+1}\sum_{n=0}^m \left(\| w_1^{n+1}\|_{r+2} + \| w_2^{n+1}\|_{r+2} \right)\| \xi_u^{n+1}\| 
        + Ckh^{r+1} \sum_{n=0}^m  \| u^{n+1}\|_{r+2} \left(\| \xi_{w_1}^{n+1}\| + \| \xi_{w_2}^{n+1}\| \right) \notag\\
    &\leq Ch^{2r+2}k \sum_{n=0}^m (\| u^{n+1}\|_{r+2}^2 + \| w_1^{n+1}\|_{r+2}^2 + \| w_2^{n+1}\|_{r+2}^2) 
        + k\sum_{n=0}^m  \| \xi_u^{n+1}\|^2 %+ Ch^{2r+2} k\sum_{n=0}^m \| u^{n+1}\|_{r+2}^2 \notag\\
        + \epsilon_6 k\sum_{n=0}^m  (\| \xi_{w_1}^{n+1}\|^2 + \| \xi_{w_2}^{n+1}\|^2)\notag\\
    &\leq 10\Lambda \epsilon_6 k\sum_{n=0}^m  (\| \xi_{v_1}^{n+1}\|^2 + \| \xi_{v_2}^{n+1}\|^2) + k\sum_{n=0}^m  \| \xi_{u}^{n+1}\|^2 \notag \\
    &\quad + Ch^{2r+2} k\sum_{n=0}^m (\| u^{n+1}\|_{r+2}^2 + \| w_1^{n+1}\|_{r+2}^2 + \| w_2^{n+1}\|_{r+2}^2 + \| v_1^{n+1}\|_{r+1}^2 + \| v_2^{n+1}\|_{r+1}^2).
\end{align}
Substituting \eqref{ineq:error-estimate-term1-2D}, \eqref{ineq:error-estimate-term2-2D}, \eqref{ineq:error-estimate-term3-2D}, \eqref{ineq:error-estimate-term4-2D} into \eqref{ineq:error-main-2D}, and applying \eqref{ineq:error2D-psi} from Lemma \ref{lem:error2D-term1-2}, we arrive at
\begin{align*}
    \frac{1}{2} \max_{0 \leq n \leq m} &\| \xi_u^{n+1}\|^2
    + \left(\frac{1}{2} - \epsilon_2 - \frac{1}{\epsilon_5} \right) \sum_{n=0}^{m} \| \xi_u^{n+1} - \xi_u^{n}\|^2   \\
    &+ \left(\alpha - \epsilon_3 - 20\Lambda \epsilon_6 -(2+\epsilon_0) B_1^2 \epsilon_4 \right)
    k \sum_{n=0}^{m} \left(\| \xi_{v_1}^{n+1}\|^2 + \| \xi_{v_2}^{n+1}\|^2 \right) \\
    &\leq Ch^{2r+2} + Ck\sum_{n=0}^{m} \max_{0 \leq l \leq n} \|\xi_u^{l+1} \|^2 %+ Ck \sum_{n=0}^{m} \| \xi_u^{n+1}\|^2    
    + Ck^2\sum_{n=0}^{m} (1+\| u^n\|_{1}^2+\|v_1^{n+1}\|_{1}^2 + \|v_2^{n+1}\|_{1}^2)  \\
    &\quad + Ch^{2r+2}\sum_{n=0}^{m} \|u^{n+1}-u^n\|_{r+1}^2 
    + Ch^{2r+2}\int_0^T \bigl(\| (w_1)_x\|_{r+1}^2 + \| (w_2)_y\|_{r+1}^2 + \| \psi\|_{r+1}^2\bigr)\,\mathrm{d}t \\    
    &\quad + Ch^{2r+2} k\sum_{n=0}^m \bigl(\| w_1^{n+1}\|_{r+2}^2 + \| w_2^{n+1}\|_{r+2}^2 + \| v_1^{n+1}\|_{r+1}^2 + \| v_2^{n+1}\|_{r+1}^2 + \| u^{n+1}\|_{r+2}^2 \bigr)    \\
    &\quad  + C \sum_{n=0}^{m} \int_{t_n}^{t_{n+1}} \Bigl( \|u-u^n\|_{1}^2 + \|u-u^{n+1}\|_{2}^2 \Bigr) \, \mathrm{d}t
    + 2\max_{0 \leq l \leq m} \sum_{n=0}^{l} \int_{t_n}^{t_{n+1}} (\mathcal{P}^-g -g, \xi_u^{n}) \, \mathrm{d}W_t\\    
    &\quad + \epsilon_5 \sum_{n=0}^{m} \left\| \int_{t_n}^{t_{n+1}} \bigl[g(x,y,t,u,v_1,v_2)-g(x,y,t_n,u_h^{n},v_{1,h}^n,v_{2,h}^n)\bigr] \, \mathrm{d}W_t \right\|^2 \\
    &\quad + 2\max_{0 \leq l \leq m} \sum_{n=0}^{l} \left( \int_{t_n}^{t_{n+1}} \bigl[g(x,y,t,u,v_1,v_2)-g(x,y,t_n,u_h^{n},v_{1,h}^n,v_{2,h}^n)\bigr] \, \mathrm{d}W_t, \xi_u^{n} \right).
\end{align*}

Taking the $q$-th power of both sides, followed by expectation, and applying the discrete H\"older inequality, Lemma \ref{lem:convex-ineq}, and Lemma \ref{lem:Holder-High-Moments}, we get 
\begin{align}
\label{ineq:error-main-expect-2d}
    \frac{1}{2^q}&\mathbb{E}\left[\max_{0 \leq n \leq m} \| \xi_u^{n+1}\|^{2q} \right]
    + \left(\frac{1}{2}-\epsilon_2-\frac{1}{\epsilon_5} \right)^q \mathbb{E}\left[ \left(\sum_{n=0}^{m} \| \xi_u^{n+1} - \xi_u^{n}\|^2 \right)^q \right] \notag\\
    &\quad + \left(\alpha-\epsilon_3-20\Lambda \epsilon_6-(2+\epsilon_0) \epsilon_4 B_1^2 \right)^q
    \mathbb{E}\left[ \left(k\sum_{n=0}^{m} (\| \xi_{v_1}^{n+1}\|^2 + \| \xi_{v_2}^{n+1}\|^2) \right)^q \right] \notag\\
    &\leq Ck\mathbb{E}\left[\sum_{n=0}^{m} \max_{0 \leq l \leq n} \|\xi_u^{l+1} \|^{2q} \right]  
    + Ck^{2q} \mathbb{E}\left[ \left(\sum_{n=0}^{m}(1+\| u^n\|_{1}^2 + \| v_1^{n+1}\|_{1}^2 + \| v_2^{n+1}\|_{1}^2) \right)^q \right] \notag\\    
    &\quad + C(h^{2r+2})^q
    + C\mathbb{E}\left[ \left(\sum_{n=0}^{m} \int_{t_n}^{t_{n+1}} \|u-u^n\|_{1}^2 + \|u-u^{n+1}\|_{2}^2 \, \mathrm{d}t \right)^q \right] \notag\\
    &\quad + \widetilde{C}^*(\epsilon_{7})2^q \mathcal{T}_1
    + \epsilon_5^q (2^{q-1}+\epsilon_{7})\mathcal{T}_2
    + (2^{q-1}+\epsilon_{7}) 2^q \mathcal{T}_3 \\
    &\leq Ck\mathbb{E}\left[\sum_{n=0}^{m} \max_{0 \leq l \leq n} \|\xi_u^{l+1} \|^{2q}\right]
    + C(h^{2r+2})^q + Ck^q 
    + \widetilde{C}^*(\epsilon_{7})2^q\mathcal{T}_1
    + \epsilon_5^q (2^{q-1}+\epsilon_{7})\mathcal{T}_2
    + (2^{q-1}+\epsilon_{7})2^q \mathcal{T}_3, \notag
\end{align}
where $\widetilde{C}^*(\epsilon_{7}) = (1+ \epsilon_7)\left(1-2 \left(2^{q-1}+\epsilon_7 \right)^{\frac{1}{1-q}} \right)^{1-q}$ is derived by \eqref{ineq:convex-1}, \eqref{ineq:convex-3} 
and 
\begin{align*}
    \mathcal{T}_1 &:= \mathbb{E}\left[\max_{0 \leq l \leq m} \left|\sum_{n=0}^{l} \int_{t_n}^{t_{n+1}} (\mathcal{P}^-g -g, \xi_u^{n}) \, \mathrm{d}W_t \right|^q \right]. \\
    \mathcal{T}_2 &:= \mathbb{E}\left[ \left(\sum_{n=0}^{m}  \left \| \int_{t_n}^{t_{n+1}} g(x,y,t,u,v_1,v_2) - g(x,y,t_{n},u_h^{n},v_{1,h}^n,v_{2,h}^n) \, \mathrm{d}W_t \right \|^2 \right)^q \right]. \\
    \mathcal{T}_3 &:= \mathbb{E}\left[\max_{0 \leq l \leq m} \left|\sum_{n=0}^{l} \left( \int_{t_n}^{t_{n+1}} g(x,y,t,u,v_1,v_2) - g(x,y,t_{n},u_h^{n},v_{1,h}^n,v_{2,h}^n) \, \mathrm{d}W_t, \xi_u^{n} \right) \right|^q \right].
\end{align*}

\noindent {\it Step 2.} 
It remains to estimate the three stochastic terms $\mathcal{T}_1$, $\mathcal{T}_2$, and $\mathcal{T}_3$. By the BDG inequality, we have
\begin{align}
\label{ineq:error-estimate-t1-2D}
    \mathcal{T}_1
    &\leq C_b \mathbb{E}\left[ \left( K\sum_{n=0}^{m} \int_{t_n}^{t_{n+1}} \|\mathcal{P}^-g -g\|^2 \|\xi_u^{n} \|^2 \, \mathrm{d}t \right)^{q/2} \right] \notag\\
    &\leq \epsilon_{8} \mathbb{E}\left[ \max_{0 \leq n \leq m} \|\xi_u^{n} \|^{2q} \right]
    + \frac{C_b^2}{4\epsilon_{8}} \mathbb{E}\left[ \left( K\sum_{n=0}^{m} \int_{t_n}^{t_{n+1}} \|\mathcal{P}^-g -g\|^2 \, \mathrm{d}t \right)^q \right] \notag\\
    &\leq \epsilon_{8} \mathbb{E}\left[ \max_{0 \leq n \leq m}  \|\xi_u^{n+1} \|^{2q} \right] + C(h^{2r+2})^q,
\end{align}
where we used Young's inequality together with the projection estimate in Lemma \ref{lem:proj-property-2d}.
By the same argument as in \eqref{ineq:stability-I1-2D}, applying the continuous and discrete BDG inequalities yields
\begin{align*}
    \mathcal{T}_{2} &\leq C_b'C_b K^q \mathbb{E}\left[ \left( \sum_{n=0}^{m} \int_{t_n}^{t_{n+1}} \|g(x,y,t,u,v_1,v_2) - g(x,y,t_{n},u_h^{n},v_{1,h}^n,v_{2,h}^n) \|^2 \, \mathrm{d}t \right)^{q} \right]. 
\end{align*}
Using \eqref{ineq:error2D-g} from Lemma \ref{lem:error2D-term1-2}, and applying \eqref{ineq:convex-1}, we get
\begin{align}
\label{ineq:error-estimate-t2-2D}
    \mathcal{T}_{2} &\leq C(k^q + (h^{2r+2})^q) + C k\sum_{n=0}^{m} \mathbb{E}\left[ \max_{0 \leq l \leq n} \| \xi_u^{l+1}\|^{2q} \right] \notag \\
    &\quad + C_b'C_b K^q D_2^{2q}(2+\epsilon_{1})^q (1+\epsilon_{1}) \mathbb{E}\left[ \left(k\sum_{n=0}^{m} (\| \xi_{v_1}^{n+1}\|^2 + \| \xi_{v_2}^{n+1}\|^2) \right)^q \right].
\end{align}
Lastly, applying \eqref{ineq:error2D-g}, \eqref{ineq:convex-1}, and the BDG inequality, we obtain
\begin{align}
\label{ineq:error-estimate-t3-2D}
    \mathcal{T}_{3}
    &\leq C_b \mathbb{E}\left[ \left( K\sum_{n=0}^{m} \int_{t_n}^{t_{n+1}} \|g(x,y,t,u,v_1,v_2) - g(x,y,t_{n},u_h^{n},v_{1,h}^n,v_{2,h}^n)\|^2 \| \xi_u^{n}\|^2 \, \mathrm{d}t  \right)^{q/2} \right] \notag \\
    &\leq C_b^2 K^q \epsilon_{9}\mathbb{E}\left[ \left(\sum_{n=0}^{m} \int_{t_n}^{t_{n+1}} \|g(x,y,t,u,v_1,v_2) - g(x,y,t_{n},u_h^{n},v_{1,h}^n,v_{2,h}^n)\|^2 \, \mathrm{d}t  \right)^q \right] 
      + \frac{1}{4\epsilon_{9}}\mathbb{E}\left[ \max_{0 \leq n \leq m} \| \xi_u^{n}\|^{2q} \right] \notag\\
    &\leq \frac{1}{4\epsilon_{9}}\mathbb{E}\left[ \max_{0 \leq n \leq m} \| \xi_u^{n+1}\|^{2q} \right]
     + Ck\sum_{n=0}^{m} \mathbb{E}\left[ \max_{0 \leq l \leq n} \| \xi_u^{l+1}\|^{2q} \right] \notag\\
    &\quad + C_b^2 K^q \epsilon_{9} D_2^{2q}(2+\epsilon_{1})^q (1+\epsilon_{1})
    \mathbb{E}\left[ \left(k\sum_{n=0}^{m} (\| \xi_{v_1}^{n+1}\|^2 + \| \xi_{v_2}^{n+1}\|^2) \right)^q \right] 
     + C\bigl(k^q + (h^{2r+2})^q\bigr).
\end{align}
% \begin{align}
% \label{ineq:error-estimate-t3-2D}
%     \mathcal{T}_{3} &\leq C_b \mathbb{E}\left[ \left( K\sum_{n=0}^{m} \int_{t_n}^{t_{n+1}} \|g(x,y,t,u,v_1,v_2) - g(x,y,t_{n},u_h^{n},v_{1,h}^n,v_{2,h}^n)\|^2 \| \xi_u^{n}\|^2 \, dt  \right)^\frac{q}{2} \right] \notag \\
%     &\leq C_b^2 K^q \epsilon_{9}\mathbb{E}\left[ \left(\sum_{n=0}^{m} \int_{t_n}^{t_{n+1}} \|g(x,y,t,u,v_1,v_2) - g(x,y,t_{n},u_h^{n},v_{1,h}^n,v_{2,h}^n)\|^2 \, dt  \right)^q \right] \notag \\
%     &+ \frac{1}{4\epsilon_{9}}\mathbb{E}\left[ \max_{0 \leq n \leq m} \| \xi_u^{n}\|^{2q} \right] \leq \frac{1}{4\epsilon_{9}}\mathbb{E}\left[ \max_{0 \leq n \leq m} \| \xi_u^{n+1}\|^{2q} \right] + Ck\sum_{n=0}^{m} \mathbb{E}\left[ \max_{0 \leq l \leq n} \| \xi_u^{l+1}\|^{2q} \right] \notag \\
%     &+ C_b^2 K^q \epsilon_{9} C_2^{2q}(2+\epsilon_{1})^q (1+\epsilon_{1}) \mathbb{E}\left[ \left(k\sum_{n=0}^{m} (\| \xi_{v_1}^{n+1}\|^2 + \| \xi_{v_2}^{n+1}\|^2) \right)^q \right] + C(k^q + (h^{2r+2})^q).
% \end{align}

Substituting \eqref{ineq:error-estimate-t1-2D}-\eqref{ineq:error-estimate-t3-2D} into \eqref{ineq:error-main-expect-2d}, we arrive at
\begin{align}
\label{ineq:high-moment-error-strong-final-2d}
    \mathcal{C}_1 &\mathbb{E}\left[\max_{0 \leq n \leq m} \| \xi_u^{n+1}\|^{2q} \right]  + \mathcal{C}_2 \mathbb{E}\left[ \left(\sum_{n=0}^{m} \| \xi_u^{n+1} - \xi_u^{n}\|^2 \right)^q \right] + \mathcal{C}_3 \mathbb{E}\left[ \left(k\sum_{n=0}^{m} (\| \xi_{v_1}^{n+1}\|^2 + \| \xi_{v_2}^{n+1}\|^2) \right)^q \right] \notag \\
    &\leq Ck^q + C(h^{2r+2})^q + Ck\sum_{n=0}^{m} \mathbb{E}\left[ \max_{0 \leq l \leq n}\| \xi_u^{l+1}\|^{2q} \right],
\end{align}
where the three coefficients on the left-hand side terms are given by 
\begin{align*}
    \mathcal{C}_1 &= \frac{1}{2^q}-\widetilde{C}^*(\epsilon_{7})2^q\epsilon_{8}-\frac{2^q(2^{q-1}+\epsilon_{7})}{4\epsilon_{9}}, \quad  \mathcal{C}_2 = \left(\frac{1}{2}-\epsilon_2-\frac{1}{\epsilon_5} \right)^q,  \\
    \mathcal{C}_3 &= \left(\alpha-\epsilon_3-20\Lambda \epsilon_6-(2+\epsilon_0) \epsilon_4 B_1^2 \right)^q - \epsilon_5^q (2^{q-1}+\epsilon_{7}) C_b'C_b K^q D_2^{2q}(1+\epsilon_1)(2+\epsilon_{1})^{q} \\
    &- (2^{q-1}+\epsilon_{7})2^q C_b^2 K^q \epsilon_{9} D_2^{2q}(1+\epsilon_1)(2+\epsilon_{1})^{q}. 
\end{align*}
As in the proof of Theorem \ref{thm:stability-estimate-2D}, one can choose fixed constants $\{\epsilon_i>0: i=0,1,\ldots,9\}$ such that $\mathcal{C}_1$, $\mathcal{C}_2$, and $\mathcal{C}_3$ are all positive 
if $\alpha > \Bigl(2^{3q-1}C_b' + 2^{6q-4}C_b \Bigr)^{1/q}C_b^{1/q}D_2^2K$. Applying the discrete Gr\"onwall inequality to \eqref{ineq:high-moment-error-strong-final-2d}, we obtain, for every $m=0,\ldots,N_T-1$,
\begin{align*}
    \mathbb{E}\left[\max_{0 \leq n \leq m} \| \xi_u^{n+1}\|^{2q} \right]
    &+ \mathbb{E}\left[ \left(\sum_{n=0}^{m} \| \xi_u^{n+1} - \xi_u^{n}\|^2 \right)^q \right] \\
    &+ \mathbb{E}\left[ \left(k\sum_{n=0}^{m} (\| \xi_{v_1}^{n+1}\|^2 + \| \xi_{v_2}^{n+1}\|^2) \right)^q \right]
    \leq C\left(k^q + h^{q(2r+2)} \right).
\end{align*}
Taking the $2q$-th root of the above inequality with $m = N_T-1$, we get 
\begin{align*}
    \mathbb{E}\left[  \| \mathbf{\xi_{u}} \|_{\mathcal{S}}^{2q} \right]^{\frac{1}{2q}} + \mathbb{E}\left[ \langle \mathbf{\xi_{u}} \rangle^{q} \right]^{\frac{1}{2q}} + \mathbb{E}\left[  \left(\| \mathbf{\xi_{v_1}} \|_{\ell^2}^{2} + \| \mathbf{\xi_{v_2}} \|_{\ell^2}^{2} \right)^q \right]^{\frac{1}{2q}} \leq C(k^{1/2} + h^{r+1}). 
    % \mathbb{E}&\left[\max_{0 \leq n \leq m} \| \xi_u^{n+1}\|^{2q} \right]^{\frac{1}{2q}} 
    % + \mathbb{E}\left[ \left(\sum_{n=0}^{m} \| \xi_u^{n+1} - \xi_u^{n}\|^2 \right)^q \right]^{\frac{1}{2q}}  \\ 
    % &+ \mathbb{E}\left[ \left(k\sum_{n=0}^{m} (\| \xi_{v_1}^{n+1}\|^2 + \| \xi_{v_2}^{n+1}\|^2) \right)^q \right]^{\frac{1}{2q}} \leq C\left(k^{\frac{1}{2}} + h^{r+1} \right). 
\end{align*}
Finally, combining this bound with the projection estimate in Lemma \ref{lem:proj-property-2d} yields \eqref{ineq:high-moment-error-estimate-2D}.
\end{proof}

As a result of Lemma \ref{lem:Kolmogorov} and Theorem \ref{thm:high-moment-error-estimate-2D}, we establish the following pathwise error estimate.
\begin{corollary}[Pathwise Error Estimate]
\label{coro:pathwise-error-2d}
For any $q > 1$, there exists $0 < \zeta < (q-1)/(2q)$, and a random variable $Z(\omega,\zeta)$ with $\mathbb{E}[|Z|^{2q}] < \infty$, such that the following holds almost surely:
\begin{align*}
    \| \mathbf{e_{u}} \|_{\mathcal{S}} + \langle \mathbf{e_{u}} \rangle^{1/2} +   \left(\| \mathbf{e_{v_1}} \|_{\ell^2}^{2} + \| \mathbf{e_{v_2}} \|_{\ell^2}^{2} \right)^{1/2} 
    \leq Z(\omega,\zeta) \bigl(k^{\zeta} + h^{\zeta(2r+2)}\bigr).
    % \max_{0 \leq n \leq N_T} \| e_u^{n+1}\| + \left(k\sum_{n=0}^{N_T} (\|e_{v_1}^{n+1}\|^2 + \|e_{v_2}^{n+1}\|^2) \right)^\frac{1}{2} \leq Z(k^{\zeta} + h^{r+\frac{1}{2}+\zeta}).
\end{align*}  
\end{corollary}     
\begin{proof}
Let
\[
X_N :=
\| \mathbf{e_u} \|_{\mathcal S}
+ \langle \mathbf{e_u} \rangle^{1/2}
+ \left(\| \mathbf{e_{v_1}} \|_{\ell^2}^{2} + \| \mathbf{e_{v_2}} \|_{\ell^2}^{2} \right)^{1/2},
\qquad k=\frac{T}{N}.
\]
By Theorem \ref{thm:high-moment-error-estimate-2D},
\[
\mathbb{E}[X_N^{2q}] \le C\bigl(k^{1/2}+h^{r+1}\bigr)^{2q}.
\]
Applying Lemma \ref{lem:Kolmogorov} with $\nu=2q$ and $\beta=q-1$, we obtain that for any
\[
0<\zeta<\frac{q-1}{2q}=\frac12-\frac{1}{2q},
\]
there exists a random variable $Z(\omega,\zeta)$ with $\mathbb{E}[|Z|^{2q}]<\infty$ such that
\[
X_N \le Z(\omega,\zeta)\bigl(k^\zeta+h^{\zeta(2r+2)}\bigr)
\]
almost surely. This proves the result.
\end{proof}

Two special cases of Theorem \ref{thm:high-moment-error-estimate-2D} are summarized in the following remarks. 
\begin{remark}[Simplified $g$]
\label{rmk:simplified-g-2d}    
Suppose that $g(\omega,x,y,t,u,v_1,v_2)$ is independent of the last two variables. Equivalently, $D_2=D_4=0$ in hypothesis {\rm(iv)}. Then both the stability and error estimates hold for any $\alpha>0$.
\end{remark}

\begin{remark}[Range of $\alpha$ for Second-Moment Error Estimate]
\label{rmk:error-second-moment-2d}
If we only consider the case $q=1$, then the lower bound on $\alpha$ can be relaxed to $\alpha > KD_2^2$. Similar to the proof of Corollary \ref{coro:stability-2nd-moment-2d}, we first establish the following weaker second-moment estimate (with the stronger counterpart being \eqref{ineq:high-moment-error-strong-final-2d}):
\begin{align}
\label{ineq:error-2nd-moment-weak-2d}
    \max_{0 \leq n \leq N_T-1} &\mathbb{E}\left[ \| \xi_u^{n+1}\|^{2} \right]  + \sum_{n=0}^{N_T-1} \mathbb{E}\left[ \| \xi_u^{n+1} - \xi_u^{n}\|^2 \right] + k\sum_{n=0}^{N_T-1} \mathbb{E}\left[ \| \xi_{v_1}^{n+1}\|^2 + \| \xi_{v_2}^{n+1}\|^2 \right] \leq C(h^{2r+2} + k),
    % \notag \\  &\leq C(h^{2r+2} + k) + Ck\sum_{n=0}^{N_T} \mathbb{E}\left[ \max_{0 \leq l \leq n}\| \xi_u^{l+1}\|^{2} \right],
\end{align}
which holds for any $\alpha > KD_2^2$. The key point is that, in contrast to \eqref{ineq:error-main-2D}, we first take expectation and then the maximum in time. In this way, the It\^o integrals appearing in $\mathcal{I}_2$ have zero expectation and need not be estimated, which leads to the relaxed condition on $\alpha$. Finally, applying \eqref{ineq:error-2nd-moment-weak-2d} to \eqref{ineq:high-moment-error-strong-final-2d} with $q=1$, and then using the discrete Gr\"onwall inequality, yields the desired conclusion.
\end{remark}

We also note that all analytical results hold if the source term is treated implicitly in time. 
\begin{remark}
\label{rmk:semi-linear-implicit-2d}
If the fully-implicit Euler scheme is used, that is, if the source term $\psi$ in \eqref{spde:conv-diff-2d} is treated implicitly, then all the results established above in Sections \ref{sec:stability-2d} and \ref{sec:error-estimate-2d} remain valid, including Theorems \ref{thm:stability-estimate-2D} and \ref{thm:high-moment-error-estimate-2D}, Corollaries \ref{coro:stability-2nd-moment-2d} and \ref{coro:pathwise-error-2d}, and Remark \ref{rmk:error-second-moment-2d}.
\end{remark}

% Lastly, we note that both Theorem \ref{thm:high-moment-error-estimate-2D} and Remark \ref{rmk:error-second-moment-2d} can be generalized if model \eqref{spde:conv-diff-2d} incorporates nonlinear convection or source terms. 
% \begin{remark}
% \label{rmk:error-nonlinear-generalization-2d}
% As shown in \cite{chen_non-linear_2d}, both Theorem Theorem \ref{thm:high-moment-error-estimate-2D} and Remark \ref{rmk:error-second-moment-2d} still hold if nonlinear convection term $\nabla F(u)$ or source term $S(u) = -u^{\lambda}$ is added to model \eqref{spde:conv-diff-2d}, provided that we choose an appropriate monotone flux for $F(u)$ and $\lambda$ is a positive odd integer. Specifically, the theorems are proved on a subset of the sample space whose probability converges to one for vanishing discretization width. The proof of same results for one-dimensional case is provided in \cite{chen_non-linear_1d}.  
% \end{remark}
\end{section}

\begin{section}{Numerical Tests}
\label{sec:numerical-test}
In this section, we apply our numerical methods to several model problems. Since the expectation of the error cannot be calculated exactly in general, we approximate the error in $L^2(\Omega,L^{\infty}[0,T;L^2(\mathcal{D})])$-norm by Monte Carlo sampling. Let $M$ be the number of sample paths and define
\[
z_i := \max_{0 \leq n \leq N_T} \| u_h^n(\omega_i,\cdot) - u(\omega_i,\cdot,t_n) \|^2,
\qquad i=1,\ldots,M,
\]
the squared maximum-in-time \(L^2\)-error of one simulation along the sample path $\omega_i$, between the exact solution $u$ and the fully-discrete numerical solution $\{ u_h^n\}_{n=0}^{N_T}$.
Then we have the Monte Carlo approximation
% \begin{gather*}
%     \mathbb{E}\left[\max_{0 \leq n \leq N_T} \| u(\omega,\cdot,t_n) - u_h^n(\omega,\cdot)\|^2 \right] \approx e_2^2 + \mathcal{V},
% \end{gather*}
% and
% \begin{gather*}
%     e_2^2 = \frac{1}{M} \sum_{i=1}^M z_i  \qquad \mathcal{V} := \frac{2}{\sqrt{M}} \left[ \frac{1}{M} \sum_{i=1}^{M} z_i^2 - \left( \frac{1}{M} \sum_{i=1}^{M} z_i \right)^2 \right]^{\frac{1}{2}},
% \end{gather*}
\begin{gather*}
    e_2^2 :=  \frac{1}{M} \sum_{i=1}^M z_i   \approx \mathbb{E} [z] = \mathbb{E}\left[\max_{0 \leq n \leq N_T} \| u(\omega,\cdot,t_n) - u_h^n(\omega,\cdot)\|^2 \right].
\end{gather*}
Hence $e_2$ approximates the error in the norm $L^2(\Omega,L^{\infty}[0,T;L^2(\mathcal{D})])$. 
% \[
% \frac{1}{M} \sum_{i=1}^M z_i \mathbb{E}\left[\max_{0 \leq n \leq N_T} \| u(\omega,\cdot,t_n) - u_h^n(\omega,\cdot)\|^2 \right] \approx e_2^2 ,
% \]
% \begin{gather*}
%     e_2^2 = \frac{1}{M} \sum_{i=1}^M z_i  \qquad \mathcal{V} := \frac{2}{\sqrt{M}} \left[ \frac{1}{M} \sum_{i=1}^{M} z_i^2 - \left( \frac{1}{M} \sum_{i=1}^{M} z_i \right)^2 \right]^{\frac{1}{2}},
% \end{gather*}
% where $\mathcal{V}$ is the Monte-Calo error, \[ z_i := \max_{0 \leq n \leq N_T} \| u_h^n(\omega_i, \cdot) - u(\omega_i, \cdot, t_n) \|^2 \]
% is the numerical error of one simulation from the path $\omega_i$, between exact solution $u$ and fully-discrete solution $\{ u_h^n\}_{n=0}^{N_T}$. We thus use $e_2$ to approximate the second moment of the $L^2$-error. 
Similarly, we denote by $e_4$ the Monte Carlo approximation of the errors in the norm $L^4(\Omega,L^{\infty}[0,T;L^2(\mathcal{D})])$. 
% to investigate the convergence of higher moments, we denote the approximated third and fourth moment of $L^2$-error by $e_3,e_4$. That is, $e_3,e_4$ are the errors in the norm $L^3(\Omega,L^{\infty}[0,T;L^2(\mathcal{D})])$ and $L^4(\Omega,L^{\infty}[0,T;L^2(\mathcal{D})])$, respectively. 
In addition, we denote by $e_{\infty}$ the error in the norm $L^{\infty}(\Omega,L^{\infty}([0,T] \times \mathcal{D}))$. 

In all experiments, $r$ is the degree of the piecewise polynomial DG space $\mathbb{V}_h$ and $T$ is the final time. We apply periodic boundary conditions, with $N$ being the number of spatial cells in each direction of the rectangular domain, and $N_t$ being the number of time steps. Unless otherwise specified, we choose $W_t$ to be a standard Brownian motion. When examining the temporal convergence rate, we take $N_t \sim N$; when examining the spatial convergence rate, we take $N_t \sim N^4$ for $r=1$ and $N_t \sim N^6$ for $r = 2$, so that the convergence rate is dominated by the spatial order. 
% As pointed out in Remark \ref{rmk:stability-nonlinear-generalization-2d} and \ref{rmk:error-nonlinear-generalization-2d}, the stability and error estimates for the LDG-Euler scheme hold on subsets of the sample space when model \eqref{spde:conv-diff-2d} admits nonlinear convection or source terms. We thus include numerical tests on nonlinear equation \eqref{spde:sburgers-2d}. 

\begin{subsection}{One-dimensional Stochastic Heat Equation with Gradient-Type Multiplicative Noise}
We first test the stability and accuracy of the scheme on the one-dimensional linear stochastic heat equation
\begin{gather}
\label{spde:linear-1d}
    \begin{cases}
        \mathrm{d}u = a u_{xx} \, \mathrm{d}t + bu_x \, \mathrm{d}W_t, \quad & (\omega,x,t)\in \Omega \times [0,2\pi]\times (0,T], \\
        u(\omega,x,0) = u_0(x),  & (\omega,x) \in \Omega \times [0,2\pi].
    \end{cases}
\end{gather}
Assuming that $u_0+(u_0)_{xx} = 0$, the exact solution to Eq. \eqref{spde:linear-1d} is given by 
\[ u(\omega,x,t) = e^{\frac{b^2}{2}t - at} u_0(x + bW_t). \]

To verify the order of accuracy of the numerical methods, we take a smooth initial condition $u_0(x) = \sin(x)$ with fixed parameters $a = 0.5$, $b = 1$. We compute the solution at final time $T = 1$ with $M=4000$ samples. Tables \ref{table:linear-1d-timerate} and \ref{table:linear-1d-spacerate} show that the scheme is of $(r+1)$-th order in space, and ${1}/{2}$-th order in time, for the $L^{\infty}$-norm and the second- and fourth-moment $L^2$ errors. These results are consistent with the corresponding one-dimensional theoretical estimate of Theorem \ref{thm:high-moment-error-estimate-2D}. 

Next, we demonstrate the role of the stochastic parabolic condition $2a > b^2$,
%(see \cite{chen_semi-linear_1d} for the analysis in one dimension), 
by testing different sets of $a$ and $b$. According to Table \ref{table:linear-1d-timerate}, the numerical error decreases as $a$ increases or as the noise intensity $b$ decreases. Note that for $(a,b) = (0.25,1)$ and $(a,b)=(1,2)$, the stochastic parabolic condition $2a > b^2$ is not satisfied, and the numerical solution becomes unstable as $N$ increases. 
% On the other hand, for $(a,b)=(1,1.5)$, we still obtain the expected convergence rate even though $2a < b^2$. This shows that the instability may not be visible if $b^2-2a$ is close to zero, unless for very fine mesh. 
For the degenerate case $(a,b)=(0.5,1)$, similar to the conclusion in Remark \ref{rmk:stability-degen-2d}, the one-dimensional version of Theorems \ref{thm:stability-estimate-2D} and \ref{thm:high-moment-error-estimate-2D} still hold for $2a=b^2$. Hence we observe the stability of the numerical solution and optimal error estimates, which is confirmed by the numerical results.

\begin{table}[H]
\centering
\renewcommand\arraystretch{1.25}
\caption{Numerical error and temporal convergence rate for the linear SPDE \eqref{spde:linear-1d} with different pairs of $(a,b)$ and $r=1$.}
\begin{tabular}{c c c c c c c c c}
\hline\hline
$a$ & $b$ & $N$ & $e_2$ & rate & $e_4$ & rate & $e_{\infty}$ & rate \\
\hline\hline
\multirow{5}{*}{1} & \multirow{5}{*}{1}
  & 10  & 3.32E-02 & --   & 4.20E-02 & --   & 5.21E-02 & -- \\
& & 20  & 2.21E-02 & 0.58 & 2.90E-02 & 0.53 & 2.95E-02 & 0.82 \\
& & 40  & 1.53E-02 & 0.53 & 2.04E-02 & 0.51 & 1.86E-02 & 0.67 \\
& & 80  & 1.08E-02 & 0.51 & 1.43E-02 & 0.51 & 1.25E-02 & 0.57 \\
& & 160 & 7.65E-03 & 0.49 & 1.02E-02 & 0.49 & 8.74E-03 & 0.52 \\
\hline
\multirow{5}{*}{0.5} & \multirow{5}{*}{1}
  & 10  & 5.49E-02 & --   & 6.91E-02 & --   & 8.62E-02 & -- \\
& & 20  & 3.65E-02 & 0.59 & 4.75E-02 & 0.54 & 4.87E-02 & 0.82 \\
& & 40  & 2.53E-02 & 0.53 & 3.35E-02 & 0.51 & 3.06E-02 & 0.67 \\
& & 80  & 1.77E-02 & 0.51 & 2.35E-02 & 0.51 & 2.06E-02 & 0.57 \\
& & 160 & 1.26E-02 & 0.49 & 1.67E-02 & 0.49 & 1.44E-02 & 0.52 \\
\hline
\multirow{5}{*}{0.25} & \multirow{5}{*}{1}
  & 10  & 2.96E-01 & --     & 2.23E+00 & --    & 1.25E-01 & -- \\
& & 20  & 4.70E-02 & 2.65   & 6.13E-02 & 5.18  & 6.30E-02 & 0.99 \\
& & 40  & 3.25E-02 & 0.54   & 4.30E-02 & 0.51  & 3.93E-02 & 0.68 \\
& & 80  & 1.30E+06 & -25.26 & 8.02E+06 & -27.47 & 1.53E+05 & -21.90 \\
& & 160 & 2.11E+27 & -70.46 & 1.51E+28 & -70.67 & 1.82E+26 & -70.01 \\
\hline
\multirow{5}{*}{1} & \multirow{5}{*}{2}
  & 10  & 6.81E-01 & --     & 3.26E+00 & --     & 7.06E-01 & -- \\
& & 20  & 3.86E-01 & 0.82   & 5.13E-01 & 2.67   & 4.68E-01 & 0.59 \\
& & 40  & 2.73E-01 & 0.50   & 3.67E-01 & 0.48   & 3.20E-01 & 0.55 \\
& & 80  & 1.16E+08 & -28.66 & 7.79E+08 & -30.98 & 8.98E+06 & -24.74 \\
& & 160 & 1.16E+29 & -69.77 & 8.42E+29 & -69.87 & 7.43E+27 & -69.49 \\
\hline\hline
\end{tabular}
\label{table:linear-1d-timerate}
\end{table}

\begin{table}[H]
\centering
\renewcommand\arraystretch{1.25}
\caption{Numerical error and spatial convergence rate for the linear SPDE \eqref{spde:linear-1d} with $a=0.5$, $b=1$.}
\begin{tabular}{c c c c c c c c}
\hline\hline
$r$ & $N$ & $e_2$ & rate & $e_4$ & rate & $e_{\infty}$ & rate \\
\hline\hline
\multirow{4}{*}{1}
& 10 & 1.55E-01 & --   & 2.34E-01 & --   & 2.02E-01 & -- \\
& 20 & 3.95E-02 & 1.98 & 5.11E-02 & 2.19 & 5.26E-02 & 1.94 \\
& 40 & 9.98E-03 & 1.98 & 1.30E-02 & 1.98 & 1.32E-02 & 1.99 \\
& 80 & 2.51E-03 & 1.99 & 3.25E-03 & 1.99 & 3.31E-03 & 2.00 \\
\hline
\multirow{3}{*}{2}
& 10 & 1.54E-01 & --   & 2.34E-01 & --   & 1.81E-01 & -- \\
& 20 & 1.98E-02 & 2.96 & 2.61E-02 & 3.16 & 2.26E-02 & 3.00 \\
& 40 & 2.48E-03 & 3.00 & 3.24E-03 & 3.01 & 2.80E-03 & 3.01 \\
\hline\hline
\end{tabular}
 \label{table:linear-1d-spacerate}
\end{table}

% Table below contains e_3 (third moment)

% \begin{table}[H]
% \centering
% \renewcommand\arraystretch{1.25}
% \caption{Numerical error and spatial convergence rate for the linear SPDE \eqref{spde:linear-1d} with $a=0.5$, $b=1$.}
% \begin{tabular}{c c c c c c c c c c}
% \hline\hline
% $r$ & $N$ & $e_2$ & rate & $e_3$ & rate & $e_4$ & rate & $e_{\infty}$ & rate \\
% \hline\hline
% \multirow{4}{*}{1}
% & 10 & 1.55E-01 & --   & 1.89E-01 & --   & 2.34E-01 & --   & 2.02E-01 & -- \\
% & 20 & 3.95E-02 & 1.98 & 4.55E-02 & 2.06 & 5.11E-02 & 2.19 & 5.26E-02 & 1.94 \\
% & 40 & 9.98E-03 & 1.98 & 1.16E-02 & 1.98 & 1.30E-02 & 1.98 & 1.32E-02 & 1.99 \\
% & 80 & 2.51E-03 & 1.99 & 2.91E-03 & 1.99 & 3.25E-03 & 1.99 & 3.31E-03 & 2.00 \\
% \hline
% \multirow{3}{*}{2}
% & 10 & 1.54E-01 & --   & 1.89E-01 & --   & 2.34E-01 & --   & 1.81E-01 & -- \\
% & 20 & 1.98E-02 & 2.96 & 2.32E-02 & 3.03 & 2.61E-02 & 3.16 & 2.26E-02 & 3.00 \\
% & 40 & 2.48E-03 & 3.00 & 2.89E-03 & 3.00 & 3.24E-03 & 3.01 & 2.80E-03 & 3.01 \\
% \hline\hline
% \end{tabular}
%  \label{table:linear-1d-spacerate}
% \end{table}

\end{subsection}

\begin{subsection}{One-dimensional Stochastic Advection-Diffusion Equation with Spatially Colored Noise}
In this example, we apply our IMEX-LDG scheme to the one-dimensional stochastic advection-diffusion equation with a nonlinear reaction term driven by a space-time $\mathcal{Q}$-Wiener process:
\begin{gather}
\label{spde:advec-diffu3-1d}
    \begin{cases}
        \mathrm{d}u = \left( a u_{xx} + u_x + \sin(u) \right) \,\mathrm{d}t + bu \, \mathrm{d}W_t, \quad & (\omega,x,t)\in \Omega \times [0,2\pi] \times (0,T], \\
        u(x,0) = \sin(x),  & x \in [0,2\pi],
    \end{cases}
\end{gather}
where the noise is given by 
\begin{align}
\label{def:Q-wiener-1d-1}
    W_t = \sum_{m=1}^{\infty} \sqrt{\gamma_m} \frac{\sin(mx)}{\sqrt{\pi}}\mathcal{B}_m(t), \qquad \gamma_m = \frac{1}{m^3}.
\end{align}
In the simulation below, we fix $a=1$ and truncate the noise term in \eqref{def:Q-wiener-1d-1} by summing up to $m = 25$ with $M=4000$ Monte Carlo sample paths. 
Since the analytic solution is unavailable, the convergence rates are approximated by comparing the numerical solutions on two successive meshes, namely, by measuring the corresponding norms of $u_h-u_{h/2}$ on a coarse mesh and a refined mesh.
%successively computing the corresponding norm of $ u_h - u_{h/2}$, the difference of numerical solution on the coarse and fine mesh $N_c, N_f$. 

To study temporal convergence rate, we fix $N_t/N = 20$ and run the simulation up to $T=1$. To study the spatial convergence rate, we fix $N_t = 10^4$ and take the final time $T=0.1$. Tables \ref{table:advec-diffu3-1d-r=1-dt~h}--\ref{table:advec-diffu3-1d-spacerate} show that, for $b = 1$, the observed temporal and spatial convergence rates are close to the theoretical $1/2$-th, and $(r+1)$-th rates, respectively. When we raise the noise coefficient to $b=1.5$, the convergence rate deteriorates, especially for the higher moments of the $L^2$-error. In this case, a finer choice of $N_t/N$ and a larger number of sample paths are needed in order to recover a clear half-order temporal rate. %We thus need fine adjustment of $N_t/N$ and $M$ to recover a clean half-order rate in time, for higher noise levels. 

We would like to point out that the noise defined in \eqref{def:Q-wiener-1d-1} lacks the spatial regularity needed to prove the error estimates in Theorem \ref{thm:high-moment-error-estimate-2D}. In particular, $W_t(x) \notin L^2(\Omega, H^1(0,2\pi))$ and the regularity assumptions imposed on the exact solution in Theorem \ref{thm:high-moment-error-estimate-2D} may not be satisfied. Numerically, since we only take a finite truncation of the infinite-dimensional colored noise, the expected convergence rate is still observed, although the $e_{\infty}$ error order may deteriorate because of the highly-oscillatory modes. To further illustrate this issue, we also consider the operator eigenvalues $\gamma_m = e^{-2m}$ for which $W_t(x) \in L^2(\Omega,C^{\infty}(0,2\pi))$. The corresponding spatial accuracy test is provided in Table \ref{table:advec-diffu3-1d-spacerate}, where we observe smaller error and a clear $(r+1)$-th convergence rate.

\begin{table}[H]
\centering
\renewcommand\arraystretch{1.25}
\caption{Numerical error and temporal convergence rate for Eq. \eqref{spde:advec-diffu3-1d} with $a=1$, $r=2$ at $T=1$. } 
\begin{tabular}{c c c c c c c c}
\hline\hline
$b$ & $N$ & $e_2$ & rate & $e_4$ & rate & $e_{\infty}$ & rate \\
\hline\hline
\multirow{5}{*}{1}
& 10 & 1.28E-02 & -- & 1.52E-02 & -- & 2.44E-02 & -- \\
& 20 & 4.20E-03 & 1.60 & 5.20E-03 & 1.54 & 7.92E-03 & 1.62 \\
& 40 & 2.37E-03 & 0.82 & 3.21E-03 & 0.70 & 4.00E-03 & 0.99 \\
& 80 & 1.66E-03 & 0.52 & 2.18E-03 & 0.56 & 2.69E-03 & 0.57 \\
& 160 & 1.16E-03 & 0.52 & 1.52E-03 & 0.53 & 1.87E-03 & 0.53 \\
\hline 
\multirow{5}{*}{1.5}
& 10 & 2.11E-02 & -- & 2.64E-02 & -- & 4.06E-02 & -- \\
& 20 & 8.79E-03 & 1.26 & 1.23E-02 & 1.10 & 1.59E-02 & 1.36 \\
& 40 & 5.67E-03 & 0.63 & 9.87E-03 & 0.32 & 9.20E-03 & 0.79 \\
& 80 & 3.93E-03 & 0.53 & 5.70E-03 & 0.79 & 6.32E-03 & 0.54 \\
& 160 & 2.78E-03 & 0.50 & 4.66E-03 & 0.29 & 4.41E-03 & 0.52 \\
\hline\hline
\end{tabular}
\label{table:advec-diffu3-1d-r=1-dt~h}
\end{table}

% \begin{table}[H]
% \centering
% \renewcommand\arraystretch{1.25}
% \caption{Numerical error and temporal convergence rate for Eq. \eqref{spde:advec-diffu3-1d} with $a=1$, $r=2$ at $T=1$. } 
% \begin{tabular}{c c c c c c c c c c}
% \hline\hline
% $b$ & $N$ & $e_2$ & rate & $e_3$ & rate & $e_4$ & rate & $e_{\infty}$ & rate \\
% \hline\hline
% \multirow{5}{*}{1}
% & 10 & 1.28E-02 & -- & 1.40E-02 & -- & 1.52E-02 & -- & 2.44E-02 & -- \\
% & 20 & 4.20E-03 & 1.60 & 4.68E-03 & 1.58 & 5.20E-03 & 1.54 & 7.92E-03 & 1.62 \\
% & 40 & 2.37E-03 & 0.82 & 2.77E-03 & 0.76 & 3.21E-03 & 0.70 & 4.00E-03 & 0.99 \\
% & 80 & 1.66E-03 & 0.52 & 1.93E-03 & 0.52 & 2.18E-03 & 0.56 & 2.69E-03 & 0.57 \\
% & 160 & 1.16E-03 & 0.52 & 1.34E-03 & 0.52 & 1.52E-03 & 0.53 & 1.87E-03 & 0.53 \\
% \hline 
% \multirow{5}{*}{1.5}
% & 10 & 2.11E-02 & -- & 2.37E-02 & -- & 2.64E-02 & -- & 4.06E-02 & -- \\
% & 20 & 8.79E-03 & 1.26 & 1.04E-02 & 1.18 & 1.23E-02 & 1.10 & 1.59E-02 & 1.36 \\
% & 40 & 5.67E-03 & 0.63 & 7.30E-03 & 0.52 & 9.87E-03 & 0.32 & 9.20E-03 & 0.79 \\
% & 80 & 3.93E-03 & 0.53 & 4.78E-03 & 0.61 & 5.70E-03 & 0.79 & 6.32E-03 & 0.54 \\
% & 160 & 2.78E-03 & 0.50 & 3.50E-03 & 0.45 & 4.66E-03 & 0.29 & 4.41E-03 & 0.52 \\
% \hline\hline
% \end{tabular}
% \label{table:advec-diffu3-1d-r=1-dt~h}
% \end{table}

\begin{table}[H]
\centering
\renewcommand\arraystretch{1.25}
\caption{Numerical error and spatial convergence rate for Eq. \eqref{spde:advec-diffu3-1d} with $a=b=1$ at $T=0.1$.} 
\begin{tabular}{c | c c c c c c c c}
\hline\hline
\multicolumn{1}{c}{$\gamma_m$} & $r$ & $N$ & $e_2$ & rate & $e_4$ & rate & $e_{\infty}$ & rate \\
\hline\hline
\multirow{8}{*}{$1/m^3$} & \multirow{4}{*}{1}
 & 10 & 3.00E-02 & -- & 3.03E-02 & -- & 7.87E-02 & -- \\
& & 20 & 8.05E-03 & 1.90 & 8.12E-03 & 1.90 & 2.37E-02 & 1.73 \\
& & 40 & 1.95E-03 & 2.05 & 1.96E-03 & 2.05 & 6.41E-03 & 1.89 \\
& & 80 & 4.92E-04 & 1.99 & 4.95E-04 & 1.99 & 1.71E-03 & 1.90 \\
\cline{2-9}
& \multirow{4}{*}{2}
  & 10 & 4.96E-03 & -- & 5.37E-03 & -- & 1.20E-02 & -- \\
& & 20 & 1.02E-03 & 2.28 & 1.09E-03 & 2.30 & 2.76E-03 & 2.12 \\
& & 40 & 8.90E-05 & 3.52 & 9.36E-05 & 3.55 & 3.24E-04 & 3.09 \\
& & 80 & 1.02E-05 & 3.12 & 1.05E-05 & 3.15 & 6.63E-05 & 2.29 \\
\hline
\multirow{8}{*}{$e^{-2m}$} & \multirow{4}{*}{1}
 & 10 & 2.73E-02 & -- & 2.73E-02 & -- & 5.97E-02 & -- \\
& & 20 & 6.92E-03 & 1.98 & 6.92E-03 & 1.98 & 1.54E-02 & 1.95 \\
& & 40 & 1.74E-03 & 1.99 & 1.74E-03 & 1.99 & 3.89E-03 & 1.99 \\
& & 80 & 4.35E-04 & 2.00 & 4.35E-04 & 2.00 & 9.76E-04 & 2.00 \\
\cline{2-9}
& \multirow{4}{*}{2}
 & 10 & 1.05E-03 & -- & 1.05E-03 & -- & 4.15E-03 & -- \\
& & 20 & 1.30E-04 & 3.01 & 1.31E-04 & 3.01 & 5.69E-04 & 2.87 \\
& & 40 & 1.63E-05 & 3.00 & 1.63E-05 & 3.00 & 7.41E-05 & 2.94 \\
& & 80 & 2.03E-06 & 3.00 & 2.04E-06 & 3.00 & 9.35E-06 & 2.99 \\
\hline\hline
\end{tabular}
\label{table:advec-diffu3-1d-spacerate}
\end{table}

\end{subsection}

\begin{subsection}{One-dimensional Stochastic Advection-Diffusion Equation with Nonlinear Source and Noise Terms}
Next, we apply the IMEX-LDG scheme to the one-dimensional stochastic advection-diffusion equation with nonlinear source and noise terms:
\begin{gather}
\label{spde:advec-diffu2-1d}
    \begin{cases}
        \mathrm{d}u = \left( a u_{xx} + \sin(u_x) + \ln(1+u^2) \right) \,\mathrm{d}t + b\cos(u_x)\, \mathrm{d}W_t, \quad & (\omega,x,t)\in \Omega \times [0,2\pi] \times (0,T], \\
        u(x,0) = \sin(x),  & x \in [0,2\pi].
    \end{cases}
\end{gather}
We fix $a=1$ and test the stability and accuracy of the proposed scheme using $M=4000$ samples. To study the temporal and spatial convergence rates separately, we compute the numerical error with $N_t/N = 20$ at $T=1$ for the temporal tests, and with $N_t = 2\times 10^4$ at $ T=0.2$ for the spatial tests. Since the analytic solution is unavailable, the convergence rates are estimated by comparing numerical solutions on two successive meshes, that is, by measuring the corresponding norms of $u_h-u_{h/2}$ on a coarse mesh and a refined mesh. 

Tables \ref{table:advec-diffu2-1d-timerate}--\ref{table:advec-diffu2-1d-spacerate} show the expected optimal $(r+1)$-th order in space and half order in time. As the noise intensity $b$ increases, the numerical error becomes much larger and the scheme converges more slowly. Moreover, we observe the scheme is stable even when $b=2$, in contrast to the behavior observed for \eqref{spde:linear-1d} and \eqref{spde:linear-2d}. This is because the stability condition $2a > b^2$ only applies when $b$ is the linear growth coefficient of the stochastic diffusion term. In this example, since the term $g= b\cos(u_x)$ is bounded, the numerical scheme is stable for finite values of $b$. However, the condition $2a > b^2$ is still required to obtain optimal error estimates. This is consistent with the numerical results, where the error no longer decreases at the expected rate when the mesh is refined for $b=2$.

\begin{table}[H]
\centering
\renewcommand\arraystretch{1.25}
\caption{Numerical error and temporal convergence rate for Eq. \eqref{spde:advec-diffu2-1d} with $a=r=1$ at $T=1$.} 
\begin{tabular}{c c c c c c c c}
\hline\hline
$b$ & $N$ & $e_2$ & rate & $e_4$ & rate & $e_{\infty}$ & rate \\
\hline\hline
\multirow{5}{*}{1}
& 10 & 1.35E-02 & -- & 1.50E-02 & -- & 2.90E-02 & -- \\
& 20 & 3.83E-03 & 1.82 & 4.39E-03 & 1.77 & 8.30E-03 & 1.80 \\
& 40 & 1.49E-03 & 1.36 & 1.92E-03 & 1.19 & 2.93E-03 & 1.50 \\
& 80 & 8.92E-04 & 0.74 & 1.28E-03 & 0.59 & 1.42E-03 & 1.05 \\
& 160 & 6.09E-04 & 0.55 & 8.81E-04 & 0.53 & 8.45E-04 & 0.75 \\
\hline 
\multirow{5}{*}{1.5}
& 10 & 1.73E-02 & -- & 2.24E-02 & -- & 3.44E-02 & -- \\
& 20 & 7.29E-03 & 1.24 & 1.11E-02 & 1.02 & 1.26E-02 & 1.44 \\
& 40 & 4.71E-03 & 0.63 & 8.95E-03 & 0.31 & 6.39E-03 & 0.99 \\
& 80 & 3.32E-03 & 0.50 & 5.55E-03 & 0.69 & 4.03E-03 & 0.67 \\
& 160 & 2.38E-03 & 0.48 & 4.27E-03 & 0.38 & 2.72E-03 & 0.57 \\
\hline
\multirow{5}{*}{2}
& 10 & 1.16E-01 & -- & 1.84E-01 & -- & 1.72E-01 & -- \\
& 20 & 1.19E-01 & -0.04 & 1.71E-01 & 0.11 & 2.01E-01 & -0.23 \\
& 40 & 1.04E-01 & 0.20 & 1.51E-01 & 0.18 & 1.96E-01 & 0.04 \\
& 80 & 9.69E-02 & 0.11 & 1.36E-01 & 0.16 & 1.96E-01 & 0.00 \\
& 160 & 9.07E-02 & 0.10 & 1.34E-01 & 0.02 & 1.93E-01 & 0.02 \\
\hline\hline
\end{tabular}
\label{table:advec-diffu2-1d-timerate}
\end{table}

\begin{table}[H]
\centering
\renewcommand\arraystretch{1.25}
\caption{Numerical error and spatial convergence rate for Eq. \eqref{spde:advec-diffu2-1d} with $a=b=1$ at $T=0.2$.} 
\begin{tabular}{c c c c c c c c}
\hline\hline
$r$ & $N$ & $e_2$ & rate & $e_4$ & rate & $e_{\infty}$ & rate \\
\hline\hline
\multirow{4}{*}{1}
& 10 & 2.39E-02 & -- & 2.40E-02 & -- & 5.53E-02 & -- \\
& 20 & 6.09E-03 & 1.97 & 6.14E-03 & 1.97 & 1.49E-02 & 1.89 \\
& 40 & 1.53E-03 & 1.99 & 1.55E-03 & 1.99 & 3.78E-03 & 1.98 \\
& 80 & 3.85E-04 & 2.00 & 3.88E-04 & 2.00 & 9.49E-04 & 2.00 \\
\hline
\multirow{4}{*}{2}
& 10 & 1.09E-03 & -- & 1.25E-03 & -- & 4.58E-03 & -- \\
& 20 & 1.35E-04 & 3.02 & 1.52E-04 & 3.04 & 6.77E-04 & 2.76 \\
& 40 & 1.69E-05 & 3.00 & 1.92E-05 & 2.99 & 9.08E-05 & 2.90 \\
& 80 & 2.11E-06 & 3.00 & 2.39E-06 & 3.00 & 1.16E-05 & 2.97 \\
\hline\hline
\end{tabular}
\label{table:advec-diffu2-1d-spacerate}
\end{table}

\end{subsection}

\begin{subsection}{Two-dimensional Stochastic Heat Equation with Gradient-Type Multiplicative Noise}
We next consider the following two-dimensional analog of \eqref{spde:linear-1d} to verify the stability and accuracy of the proposed scheme:
\begin{gather}
\label{spde:linear-2d}
    \begin{cases}
        \mathrm{d}u = a (u_{xx} +u_{yy})\, \mathrm{d}t + b(u_x + u_y) \mathrm{d}W_t, \quad & (\omega,x,y,t)\in \Omega \times [0,2\pi]^2\times (0,T], \\
        u(\omega,x,y,0) = \sin(x+y),  & (\omega,x,y) \in \Omega \times [0,2\pi]^2.
    \end{cases}
\end{gather}
Equation \eqref{spde:linear-2d} admits the exact solution of the form $u(\omega,x,y,t) = e^{2b^2t-2at}\sin(x+y + 2bW_t)$. 

The computational domain $[0,2\pi]^2$ is uniformly divided into $N^2$ square cells. We fix the leading coefficient $a = 1$, and compute the numerical solution up to the final time $T = 1$ with $M=1000$ realizations. Tables \ref{table:linear-2d-r=1-dt~h}--\ref{table:linear-2d-spacerate} show that, for $b=1$, the spatial and temporal convergence rates are optimal and the error estimates in Theorem \ref{thm:high-moment-error-estimate-2D} are sharp. 

Next, we vary the size of the noise $b$ and investigate its effect on the numerical error and stability. Table \ref{table:linear-2d-r=1-dt~h} shows that the error increases as the noise magnitude becomes larger. Furthermore, the stochastic parabolic condition for Eq. \eqref{spde:linear-2d} is $a > b^2$, as opposed to $2a > b^2$ for the one-dimensional case \eqref{spde:linear-1d}. Hence we observe instability for $b=1.2$ and large $N$. For the degenerate case $a=b=1$, as noted in Remark \ref{rmk:stability-degen-2d}, Theorems \ref{thm:stability-estimate-2D} and \ref{thm:high-moment-error-estimate-2D} remain valid, because Equation \eqref{spde:linear-2d} contains neither a convection nor a source (drift) term. Numerical convergence rates verify this analytic result.

\begin{table}[H]
\centering
\renewcommand\arraystretch{1.25}
\caption{Numerical error and temporal convergence rate for Eq. \eqref{spde:linear-2d} with $a=r=1$. } 
\begin{tabular}{c c c c c c c c}
\hline\hline
$b$ & $N$ & $e_2$ & rate & $e_4$ & rate & $e_{\infty}$ & rate \\
\hline\hline
\multirow{5}{*}{0.5}
& 10  & 1.08E-02 & --   & 1.22E-02 & --    & 2.29E-02 & --    \\
& 20  & 5.99E-03 & 0.85 & 7.70E-03 & 0.66 & 1.01E-02 & 1.19 \\
& 40  & 3.99E-03 & 0.59 & 5.09E-03 & 0.60 & 5.43E-03 & 0.89 \\
& 80  & 2.72E-03 & 0.55 & 3.56E-03 & 0.51 & 3.30E-03 & 0.72 \\
& 160 & 1.92E-03 & 0.51 & 2.53E-03 & 0.50 & 2.25E-03 & 0.55 \\
\hline
\multirow{5}{*}{1}
& 10  & 1.43E-01 & --   & 1.83E-01 & --   & 2.16E-01 & --    \\
& 20  & 1.01E-01 & 0.50 & 1.34E-01 & 0.45 & 1.31E-01 & 0.72 \\
& 40  & 7.04E-02 & 0.52 & 8.95E-02 & 0.58 & 8.54E-02 & 0.62 \\
& 80  & 4.85E-02 & 0.54 & 6.35E-02 & 0.49 & 5.60E-02 & 0.61 \\
& 160 & 3.42E-02 & 0.50 & 4.51E-02 & 0.49 & 3.96E-02 & 0.50 \\
\hline
\multirow{5}{*}{1.2}
& 10  & 4.80E-01 & --   & 6.17E-01 & --   & 6.81E-01 & --  \\
& 20  & 3.47E-01 & 0.47 & 4.62E-01 & 0.42 & 4.38E-01 & 0.64 \\
& 40  & 2.43E-01 & 0.51 & 3.08E-01 & 0.59 & 2.93E-01 & 0.58 \\
& 80  & 1.68E-01 & 0.53 & 2.21E-01 & 0.48 & 1.95E-01 & 0.59 \\
& 160 & 2.02E+10 & -36.81 & 1.07E+11 & -38.81 & 1.69E+09 & -33.01 \\
\hline\hline
\end{tabular}
 \label{table:linear-2d-r=1-dt~h}
\end{table}

\begin{table}[H]
\centering
\renewcommand\arraystretch{1.25}
\caption{Numerical error and spatial convergence rate for Eq. \eqref{spde:linear-2d} with $a=b=1$.}
\begin{tabular}{c c c c c c c c}
\hline\hline
$r$ & $N$ & $e_2$ & rate & $e_4$ & rate & $e_{\infty}$ & rate \\
\hline\hline
\multirow{4}{*}{1}
& 10 & 4.68E-01 & --   & 7.04E-01 & --   & 6.03E-01 & -- \\
& 20 & 1.52E-01 & 1.63 & 1.97E-01 & 1.84 & 1.93E-01 & 1.64 \\
& 40 & 3.91E-02 & 1.96 & 5.16E-02 & 1.93 & 4.84E-02 & 1.99 \\
& 80 & 1.00E-02 & 1.96 & 1.31E-02 & 1.97 & 1.24E-02 & 1.96 \\
\hline
\multirow{3}{*}{2}
& 10 & 4.67E-01 & --   & 7.06E-01 & --   & 5.79E-01 & -- \\
& 20 & 8.01E-02 & 2.55 & 1.08E-01 & 2.71 & 9.10E-02 & 2.67 \\
& 40 & 1.00E-02 & 3.00 & 1.31E-02 & 3.04 & 1.14E-02 & 2.99 \\
\hline\hline
\end{tabular}
 \label{table:linear-2d-spacerate}
\end{table}

% \begin{table}[H]
% \centering
% \renewcommand\arraystretch{1.25}
% \caption{Numerical error and spatial convergence rate for Eq. \eqref{spde:linear-2d} with $a=b=1$.}
% \begin{tabular}{c c c c c c c c c c}
% \hline\hline
% $r$ & $N$ & $e_2$ & rate & $e_3$ & rate & $e_4$ & rate & $e_{\infty}$ & rate \\
% \hline\hline
% \multirow{4}{*}{1}
% & 10 & 4.68E-01 & --   & 5.61E-01 & --   & 7.04E-01 & --   & 6.03E-01 & -- \\
% & 20 & 1.52E-01 & 1.63 & 1.74E-01 & 1.69 & 1.97E-01 & 1.84 & 1.93E-01 & 1.64 \\
% & 40 & 3.91E-02 & 1.96 & 4.56E-02 & 1.93 & 5.16E-02 & 1.93 & 4.84E-02 & 1.99 \\
% & 80 & 1.00E-02 & 1.96 & 1.17E-02 & 1.96 & 1.31E-02 & 1.97 & 1.24E-02 & 1.96 \\
% \hline
% \multirow{3}{*}{2}
% & 10 & 4.67E-01 & --    & 5.62E-01 & --    & 7.06E-01 & --    & 5.79E-01 & -- \\
% & 20 & 8.01E-02 & 2.55 & 9.42E-02 & 2.58 & 1.08E-01 & 2.71 & 9.10E-02 & 2.67 \\
% & 40 & 1.00E-02 & 3.00 & 1.17E-02 & 3.01 & 1.31E-02 & 3.04 & 1.14E-02 & 2.99 \\
% \hline\hline
% \end{tabular}
%  \label{table:linear-2d-spacerate}
% \end{table}

\end{subsection}

\begin{subsection}{Two-dimensional Stochastic Advection-Diffusion Equation}
Finally, we consider the two-dimensional stochastic advection-diffusion equation with a nonlinear reaction/source term: 
\begin{gather}
\label{spde:advec-diffu-2d}
    \begin{cases}
        \mathrm{d}u = [ a(u_{xx} + u_{yy}) + b(u_x+u_y) + R(u)] \,\mathrm{d}t + cu \, \mathrm{d}W_t, \quad & (\omega,x,y,t)\in \Omega \times [0,2\pi]^2 \times (0,T], \\
        u(x,y,0) = \sin(x)\sin(y),  & (\omega,x,y) \in \Omega \times [0,2\pi]^2,
    \end{cases}
\end{gather}
where $R(u) = \sin(u) - h(\omega,x,y,t)$, and $h$ is chosen so that the exact solution of \eqref{spde:advec-diffu-2d} is 
$$
u(\omega,x,y,t) = e^{cW_t-0.5c^2t-2at} \sin(x+bt)\sin(y+bt).
$$ 

We partition the domain $[0,2\pi]^2$ into $N^2$ uniform square cells, fix $a=b=1$, and compute the numerical solution up to the final time $T=1$ using $M=5000$ sample paths. Tables \ref{table:advec-diffu-2d-r=1-dt~h}--\ref{table:advec-diffu-2d-spacerate} show that, for $c=1$, the method achieves the optimal $(r+1)$-th order spatial and $0.5$-th order temporal convergence rates, as predicted in Theorem \ref{thm:high-moment-error-estimate-2D}. We also investigate the effect of the noise intensity. As $c$ increases from $0.5$ to $1.5$ while the ratio $N_t/N$ is kept the same, all numerical errors become larger and the convergence becomes slower. For these larger noise levels, the scheme is still stable which is consistent with the stability result in Theorem \ref{thm:stability-estimate-2D}. For the smaller noise case $c=0.5$, a finer space-time resolution is needed in order to observe the theoretical half-order temporal convergence rate more clearly.

\begin{table}[H]
\centering
\renewcommand\arraystretch{1.25}
\caption{Numerical error and temporal convergence rate for Eq. \eqref{spde:advec-diffu-2d} with $a=b=r=1$. } 
\begin{tabular}{c c c c c c c c}
\hline\hline
$c$ & $N$ & $e_2$ & rate & $e_4$ & rate & $e_{\infty}$ & rate \\
\hline\hline
\multirow{5}{*}{0.5}
& 10 & 1.02E-02 & --   & 1.31E-02 & --   & 2.41E-02 & -- \\
& 20 & 4.94E-03 & 1.05 & 6.88E-03 & 0.93 & 9.58E-03 & 1.33 \\
& 40 & 2.66E-03 & 0.89 & 3.93E-03 & 0.81 & 4.48E-03 & 1.10 \\
& 80 & 1.52E-03 & 0.81 & 2.28E-03 & 0.79 & 2.37E-03 & 0.92 \\
& 160 & 9.59E-04 & 0.66 & 1.49E-03 & 0.61 & 1.42E-03 & 0.73 \\
\hline
\multirow{5}{*}{1}
& 10 & 2.16E-02 & --   & 4.39E-02 & --   & 3.21E-02 & -- \\
& 20 & 1.39E-02 & 0.63 & 3.08E-02 & 0.51 & 1.70E-02 & 0.92 \\
& 40 & 9.21E-03 & 0.60 & 2.19E-02 & 0.49 & 1.03E-02 & 0.73 \\
& 80 & 6.06E-03 & 0.61 & 1.42E-02 & 0.62 & 6.58E-03 & 0.64 \\
& 160 & 4.38E-03 & 0.47 & 1.05E-02 & 0.43 & 4.58E-03 & 0.52 \\
\hline
\multirow{5}{*}{1.5}
& 10 & 5.20E-02 & --   & 1.37E-01 & --   & 4.99E-02 & -- \\
& 20 & 3.73E-02 & 0.48 & 1.04E-01 & 0.40 & 2.99E-02 & 0.74 \\
& 40 & 2.63E-02 & 0.51 & 8.21E-02 & 0.34 & 1.94E-02 & 0.63 \\
& 80 & 1.72E-02 & 0.62 & 5.04E-02 & 0.70 & 1.28E-02 & 0.60 \\
& 160 & 1.26E-02 & 0.44 & 3.72E-02 & 0.44 & 9.07E-03 & 0.49 \\
\hline\hline
\end{tabular}
\label{table:advec-diffu-2d-r=1-dt~h}
\end{table}

\begin{table}[H]
\centering
\renewcommand\arraystretch{1.25}
\caption{Numerical error and spatial convergence rate for Eq. \eqref{spde:advec-diffu-2d} with $a=b=c=1$.}
\begin{tabular}{c c c c c c c c}
\hline\hline
$r$ & $N$ & $e_2$ & rate & $e_4$ & rate & $e_{\infty}$ & rate \\
\hline\hline
\multirow{4}{*}{1} 
& 10 & 7.15E-02 & --   & 1.26E-01 & --   & 9.97E-02 & -- \\
& 20 & 1.07E-02 & 2.74 & 2.56E-02 & 2.29 & 1.31E-02 & 2.93 \\
& 40 & 2.30E-03 & 2.21 & 4.88E-03 & 2.39 & 2.98E-03 & 2.14 \\
& 80 & 6.05E-04 & 1.93 & 1.39E-03 & 1.82 & 7.62E-04 & 1.97 \\
\hline
\multirow{3}{*}{2} 
& 10 & 7.13E-02 & --   & 1.26E-01 & --   & 9.02E-02 & -- \\
& 20 & 4.88E-03 & 3.87 & 1.11E-02 & 3.50 & 5.16E-03 & 4.13 \\
& 40 & 6.00E-04 & 3.02 & 1.38E-03 & 3.01 & 6.26E-04 & 3.04 \\
\hline\hline
\end{tabular}
 \label{table:advec-diffu-2d-spacerate}
\end{table}

% \begin{table}[H]
% \centering
% \renewcommand\arraystretch{1.25}
% \caption{Numerical error and spatial convergence rate for Eq. \eqref{spde:advec-diffu-2d} with $a=b=c=1$ at $T=1$.}
% \begin{tabular}{c c c c c c c c c c}
% \hline\hline
% $r$ & $N$ & $e_2$ & rate & $e_3$ & rate & $e_4$ & rate & $e_{\infty}$ & rate \\
% \hline\hline
% \multirow{4}{*}{1} 
% & 10 & 7.15E-02 & -- & 9.76E-02 & -- & 1.26E-01 & -- & 9.97E-02 & -- \\
% & 20 & 1.07E-02 & 2.74 & 1.76E-02 & 2.47 & 2.56E-02 & 2.29 & 1.31E-02 & 2.93 \\
% & 40 & 2.30E-03 & 2.21 & 3.58E-03 & 2.30 & 4.88E-03 & 2.39 & 2.98E-03 & 2.14 \\
% & 80 & 6.05E-04 & 1.93 & 9.77E-04 & 1.87 & 1.39E-03 & 1.82 & 7.62E-04 & 1.97 \\
% \hline
% \multirow{3}{*}{2} 
% & 10 & 7.13E-02 & -- & 9.76E-02 & -- & 1.26E-01 & -- & 9.02E-02 & -- \\
% & 20 & 4.88E-03 & 3.87 & 7.85E-03 & 3.64 & 1.11E-02 & 3.50 & 5.16E-03 & 4.13 \\
% & 40 & 6.00E-04 & 3.02 & 9.72E-04 & 3.01 & 1.38E-03 & 3.01 & 6.26E-04 & 3.04 \\
% \hline\hline
% \end{tabular}
%  \label{table:advec-diffu-2d-spacerate}
% \end{table}

\end{subsection}

\end{section}

\begin{section}{Concluding Remarks}
    In this paper, we proposed and analyzed the fully-discrete IMEX-LDG scheme for a class of quasilinear stochastic convection-diffusion equations on two-dimensional Cartesian meshes. We proved high-moment stability estimates for the fully-discrete scheme, together with optimal error estimates of order $\mathcal{O}(h^{r+1})$ in space and $\mathcal{O}(k^{{1}/{2}})$ in time in the semilinear setting. Pathwise error estimates were also established by combining the high-moment error estimate and the discrete Kolmogorov lemma. Numerical experiments were performed to confirm the theoretical results of the numerical methods. Although the analysis was carried out in two spatial dimensions for clarity, the framework extends to higher-dimensional Cartesian meshes.
\end{section}

\begin{appendices}

\section{Proofs for High-Moment Stability and Error Estimates}

\label{appendixA}

% \subsection{Proof of Lemma \ref{lem:convex-ineq}}
% \label{appendix-convex-ineq}
% \begin{proof}
% The two inequalities are clearly true for $q=1$. Now we consider $q > 1$. Since $f(x) = x^q$ is a convex function, by standard Jensen's inequality, we have for any $ 0 < \lambda < 1$, 
% \[ f( \lambda a + (1-\lambda)b) \leq \lambda f(a) + (1-\lambda)f(b) \quad \forall a, b \geq 0. \]
% To show inequality \eqref{ineq:convex-1}, $\forall a, b \geq 0$, $\forall \epsilon > 0$, we have
% \begin{align*}
%     (a+b)^q &= (\lambda \frac{a}{\lambda} + (1-\lambda) \frac{b}{1-\lambda})^q \leq \lambda  \frac{a^q}{\lambda^q} + (1-\lambda) \frac{b^q}{(1-\lambda)^q} \\
%     &=\lambda^{1-q} \,a^q + (1-\lambda)^{1-q} \,b^q := C_1(\epsilon)a^q + (1+\epsilon)b^q.  
% \end{align*}
% Thus we get $C_1(\epsilon) = \left(1 - (1+\epsilon)^{\frac{1}{1-q}} \right)^{1-q}$. For inequality \eqref{ineq:convex-3}, we similarly have
% \begin{align*}
%     (a+b+c)^q &\leq \lambda^{1-q} \,a^q + (1-\lambda)^{1-q}\,(b+c)^q \\
%     &\leq \lambda^{1-q} \,a^q + 2^{q-1}(1-\lambda)^{1-q}\,(b^q + c^q) := C_2(\epsilon)a^q + (2^{q-1} + \epsilon)(b^q + c^q).
% \end{align*}
% Hence we get $C_2(\epsilon) = \left(1-2 \left(2^{q-1}+\epsilon \right)^{\frac{1}{1-q}} \right)^{1-q}$.
% \end{proof}

\subsection{Choice of $\{ \epsilon_i: i = 1,\ldots,6 \}$ in Theorem \ref{thm:stability-estimate-2D}}
\label{appendix-eps-choice}
We claim that if $\alpha > \alpha_0 := \Bigl(2^{2q-1}C_b' + 2^{5q-4}C_b \Bigr)^{1/q}C_b^{1/q}D_4^2K$,
%KD_4^2(\frac{1}{2}C_b^2 + 2^{3q-4}C_b^2)^{\frac{1}{q}}
then there exist $\{\epsilon_i >0: i=1,\ldots,6\}$, such that the coefficients
\begin{align*}
    \mathcal{C}_1 &= \frac{1}{2^q}-\frac{2^q(2^{q-1}+\epsilon_3)}{4\epsilon_5}, \quad \mathcal{C}_2 = \left(\frac{1}{2}-\frac{1}{\epsilon_2} \right)^q, \\
    \mathcal{C}_3 &= \left(\alpha - \epsilon_1 B_3^2 \right)^q
    - (1+\epsilon_4)(2^{q-1}+\epsilon_3)\epsilon_2^q C_b'C_{b}K^qD_4^{2q} 
    - (1+\epsilon_6)(2^{q-1}+\epsilon_3) 2^q\epsilon_5C_b^2 K^q D_4^{2q} 
\end{align*}
are all positive.
\begin{proof} %\hfill
% \begin{itemize}

% \item $\Leftarrow$: Assume there exist such $\epsilon_i$'s such that $\mathcal{C}_1, \mathcal{C}_2, \mathcal{C}_3 > 0$. Then $\mathcal{C}_3 > 0$ implies 
% \begin{align*}
%     \alpha^q > \left(\alpha - \epsilon_1 B_3^2 \right)^q > (2^{q-1}+\epsilon_3)\epsilon_2^q C_{b}^2K^qD_4^{2q}(1+\epsilon_4) + (2^{q-1}+\epsilon_3) \epsilon_5C_b^2 K^q D_4^{2q} (1+\epsilon_6) > \alpha_0^q,
% \end{align*}
% where $\epsilon_5 > 2^{2q-3}$ and $\epsilon_2 > \frac{1}{2}$ due to the condition $\mathcal{C}_1, \mathcal{C}_2 > 0$. 

% \item $\Rightarrow$: 
% Given $ \alpha >\alpha_0$, we show the existence of $\{\epsilon_i: i = 1,\ldots,6\}$. For example, we can set
First choose $\epsilon_2>2$ and $\epsilon_5>2^{3q-3}$. Then $\mathcal{C}_2>0$, and by taking $\epsilon_3>0$ sufficiently small we also ensure $\mathcal{C}_1>0$.

Next choose $\epsilon_1>0$ so that
\[
\epsilon_1 B_3^2 = \frac{\alpha-\alpha_0}{2}
\qquad\text{if } B_3>0.
\]
If $B_3=0$, we simply choose any $\epsilon_1>0$, since in that case $\epsilon_1B_3^2=0$. Then
\[
\alpha-\epsilon_1B_3^2 \ge \frac{\alpha+\alpha_0}{2},
\]
and hence
\[
(\alpha-\epsilon_1B_3^2)^q
\ge \alpha_0^q + \Bigl(\frac{\alpha-\alpha_0}{2}\Bigr)^q.
\]
If $\epsilon_2$ and $\epsilon_5$ are chosen arbitrarily close to their lower bounds $2$ and $2^{3q-3}$, respectively, and if $\epsilon_3,\epsilon_4,\epsilon_6$ are taken sufficiently small, then
\begin{align*}
&(1+\epsilon_4)(2^{q-1}+\epsilon_3)\epsilon_2^q C_b'C_bK^qD_4^{2q}
+ (1+\epsilon_6)(2^{q-1}+\epsilon_3)2^q\epsilon_5 C_b^2K^qD_4^{2q}
\end{align*}
can be made arbitrarily close to
\[
2^{2q-1} C_b'C_bK^qD_4^{2q} + 2^{5q-4} C_b^2K^qD_4^{2q}
= \alpha_0^q,
\]
which leads to
\begin{align*}
&(1+\epsilon_4)(2^{q-1}+\epsilon_3)\epsilon_2^q C_b'C_bK^qD_4^{2q}
+ (1+\epsilon_6)(2^{q-1}+\epsilon_3)2^q\epsilon_5 C_b^2K^qD_4^{2q} 
< \alpha_0^q + \Bigl(\frac{\alpha-\alpha_0}{2}\Bigr)^q.
\end{align*}
Consequently, we obtain $\mathcal{C}_3>0$.

Therefore, there exist $\{\epsilon_i>0: i=1,\ldots,6\}$ such that $\mathcal{C}_1,\mathcal{C}_2,\mathcal{C}_3$ are all positive. 
\end{proof}

\subsection{Proof of Lemma \ref{lem:error2D-term1-2}}
\label{appendix:term1-2}

\begin{proof}
Let $\epsilon_0,\epsilon_1>0$ be arbitrary. We first estimate
\[
\sum_{n=0}^{m}\int_{t_n}^{t_{n+1}}
\|\psi(\cdot,t,u,v_1,v_2)-\psi(\cdot,t_n,u_h^n,v_{1,h}^n,v_{2,h}^n)\|^2\, \mathrm{d}t.
\]
For each $t\in[t_n,t_{n+1}]$, we decompose
\begin{align*}
&\psi(\cdot,t,u,v_1,v_2)-\psi(\cdot,t_n,u_h^n,v_{1,h}^n,v_{2,h}^n) 
=
\bigl[\psi(\cdot,t,u,v_1,v_2)-\psi(\cdot,t,u^n,v_1^n,v_2^n)\bigr] \\
&\qquad\quad +
\bigl[\psi(\cdot,t,u^n,v_1^n,v_2^n)-\psi(\cdot,t_n,u^n,v_1^n,v_2^n)\bigr] 
+
\bigl[\psi(\cdot,t_n,u^n,v_1^n,v_2^n)-\psi(\cdot,t_n,u_h^n,v_{1,h}^n,v_{2,h}^n)\bigr].
\end{align*}
Using the elementary inequality
\[
\|a+b+c\|^2 \le C\bigl(\|a\|^2+\|b\|^2\bigr) + (1+\epsilon_0)\|c\|^2,
\]
we obtain
\begin{align*}
&\|\psi(\cdot,t,u,v_1,v_2)-\psi(\cdot,t_n,u_h^n,v_{1,h}^n,v_{2,h}^n)\|^2 \\
&\qquad \le
C\|\psi(\cdot,t,u,v_1,v_2)-\psi(\cdot,t,u^n,v_1^n,v_2^n)\|^2 \\
&\qquad\quad + C\|\psi(\cdot,t,u^n,v_1^n,v_2^n)-\psi(\cdot,t_n,u^n,v_1^n,v_2^n)\|^2 \\
&\qquad\quad + (1+\epsilon_0)\|\psi(\cdot,t_n,u^n,v_1^n,v_2^n)-\psi(\cdot,t_n,u_h^n,v_{1,h}^n,v_{2,h}^n)\|^2.
\end{align*}
By hypothesis {\rm(iii)},
\begin{align*}
\int_{t_n}^{t_{n+1}}
\|\psi(\cdot,t,u,v_1,v_2)-\psi(\cdot,t,u^n,v_1^n,v_2^n)\|^2\, \mathrm{d}t
&\le 3B_1^2\int_{t_n}^{t_{n+1}}\|u-u^n\|_1^2\, \mathrm{d}t,
\end{align*}
and
\begin{align*}
\int_{t_n}^{t_{n+1}}
\|\psi(\cdot,t,u^n,v_1^n,v_2^n)-\psi(\cdot,t_n,u^n,v_1^n,v_2^n)\|^2\,\mathrm{d}t
&\le C k^2(1+\|u^n\|_1^2).
\end{align*}
For the last term, hypothesis {\rm(iii)} again yields
\begin{align*}
&\|\psi(\cdot,t_n,u^n,v_1^n,v_2^n)-\psi(\cdot,t_n,u_h^n,v_{1,h}^n,v_{2,h}^n)\|^2 
\le B_1^2\bigl(\|e_u^n\|+\|e_{v_1}^n\|+\|e_{v_2}^n\|\bigr)^2.
\end{align*}
Using $e_u^n=\xi_u^n-\eta_u^n$, $e_{v_i}^n=\xi_{v_i}^n-\eta_{v_i}^n$, together with Young's inequality and Lemma \ref{lem:proj-property-2d}, we infer
\begin{align*}
&\|\psi(\cdot,t_n,u^n,v_1^n,v_2^n)-\psi(\cdot,t_n,u_h^n,v_{1,h}^n,v_{2,h}^n)\|^2 \\
&\qquad \le
C\|\xi_u^n\|^2
+ (2+\epsilon_0)B_1^2\bigl(\|\xi_{v_1}^n\|^2+\|\xi_{v_2}^n\|^2\bigr)
+ Ch^{2r+2}\bigl(\|u^n\|_{r+1}^2+\|v_1^n\|_{r+1}^2+\|v_2^n\|_{r+1}^2\bigr).
\end{align*}
Combining the above bounds and summing over $n=0,\ldots,m$, we obtain
\begin{align*}
&\sum_{n=0}^{m}\int_{t_n}^{t_{n+1}}
\|\psi(\cdot,t,u,v_1,v_2)-\psi(\cdot,t_n,u_h^n,v_{1,h}^n,v_{2,h}^n)\|^2\,\mathrm{d}t \\
&\qquad \le
C\sum_{n=0}^{m}\int_{t_n}^{t_{n+1}}\|u-u^n\|_1^2\,\mathrm{d}t
+ Ck^2\sum_{n=0}^{m}(1+\|u^n\|_1^2)
+ Ck\sum_{n=0}^{m}\|\xi_u^n\|^2 \\
&\qquad\quad
+ (2+\epsilon_0)B_1^2\,k\sum_{n=0}^{m}\bigl(\|\xi_{v_1}^n\|^2+\|\xi_{v_2}^n\|^2\bigr)
+ Ch^{2r+2}k\sum_{n=0}^{m}\bigl(\|u^n\|_{r+1}^2+\|v_1^n\|_{r+1}^2+\|v_2^n\|_{r+1}^2\bigr),
\end{align*}
which proves \eqref{ineq:error2D-psi}. The proof of \eqref{ineq:error2D-g} is analogous by using hypothesis {\rm(iv)} in place of {\rm(iii)}.
\end{proof}

\subsection{Proof of Lemma \ref{lem:w_h-v_h-2D}}
\label{appendix-w_h-v_h-2d}

\begin{proof} 
We only prove the estimate for $\|\xi_{w_1}^{n+1}\|^2$, since the estimate for $\|\xi_{w_2}^{n+1}\|^2$ is obtained in the same way. By \eqref{eq:error-3-2d}, we have
\begin{align*}
    (\xi_{w_1}^{n+1},z_h)
    &=
    (\eta_{w_1}^{n+1},z_h)
    + \bigl(a_{11}^{n+1}\xi_{v_1}^{n+1},z_h\bigr)
    - \bigl(a_{11}^{n+1}\eta_{v_1}^{n+1},z_h\bigr) 
    + \bigl(a_{12}^{n+1}\xi_{v_2}^{n+1},z_h\bigr)
    - \bigl(a_{12}^{n+1}\eta_{v_2}^{n+1},z_h\bigr).
\end{align*}
Choosing $z_h=\xi_{w_1}^{n+1}$ and applying the Cauchy-Schwarz inequality, Young's inequality, hypothesis {\rm(ii)}, and Lemma \ref{lem:proj-property-2d}, we obtain
\begin{align*}
    \|\xi_{w_1}^{n+1}\|^2
    &\leq
    5\|\eta_{w_1}^{n+1}\|^2
    + 5\Lambda\|\xi_{v_1}^{n+1}\|^2
    + 5\Lambda\|\eta_{v_1}^{n+1}\|^2
    + 5\Lambda\|\xi_{v_2}^{n+1}\|^2
    + 5\Lambda\|\eta_{v_2}^{n+1}\|^2 \\
    &\leq
    Ch^{2r+2}\bigl(\|w_1^{n+1}\|_{r+1}^2 + \|v_1^{n+1}\|_{r+1}^2 + \|v_2^{n+1}\|_{r+1}^2\bigr)
    + 5\Lambda\bigl(\|\xi_{v_1}^{n+1}\|^2 + \|\xi_{v_2}^{n+1}\|^2\bigr).
\end{align*}
Summing over $n=0,\ldots,m$ gives
\begin{align*}
    k\sum_{n=0}^m \|\xi_{w_1}^{n+1}\|^2
    &\leq
    5\Lambda\,k\sum_{n=0}^m \bigl(\|\xi_{v_1}^{n+1}\|^2 + \|\xi_{v_2}^{n+1}\|^2\bigr) \\
    &\quad
    + Ch^{2r+2}k\sum_{n=0}^m
    \bigl(\|w_1^{n+1}\|_{r+1}^2 + \|v_1^{n+1}\|_{r+1}^2 + \|v_2^{n+1}\|_{r+1}^2\bigr).
\end{align*}

By the same argument,
\begin{align*}
    k\sum_{n=0}^m \|\xi_{w_2}^{n+1}\|^2
    &\leq
    5\Lambda\,k\sum_{n=0}^m \bigl(\|\xi_{v_1}^{n+1}\|^2 + \|\xi_{v_2}^{n+1}\|^2\bigr) \\
    &\quad
    + Ch^{2r+2}k\sum_{n=0}^m
    \bigl(\|w_2^{n+1}\|_{r+1}^2 + \|v_1^{n+1}\|_{r+1}^2 + \|v_2^{n+1}\|_{r+1}^2\bigr).
\end{align*}
Adding the two inequalities and inserting the nonnegative term
$
Ch^{2r+2}k\sum_{n=0}^m \|u^{n+1}\|_{r+1}^2
$
for later convenience, we obtain \eqref{ineq:wh-vh-2d}.
\end{proof}

\subsection{Proof of the estimate for $\mathcal{I}_3$ in \eqref{ineq:error-estimate-term3-2D}}
\label{appendix:error-estimate-term3}
\begin{proof}
Since $w_1$ and $w_2$ are smooth exact solutions, the definition of $H^+$ and integration by parts give
\begin{align*}
    \mathcal{I}_3 &\leq \sum_{n=0}^{m} \int_{t_n}^{t_{n+1}}  \left|H^+(w_1-w_1^{n+1}, w_2-w_2^{n+1}, \xi_u^{n+1}) \right| \, \mathrm{d}t \\
    &= \sum_{n=0}^{m} \int_{t_n}^{t_{n+1}}  \left|((w_1-w_1^{n+1})_x + (w_2-w_2^{n+1})_y, \xi_u^{n+1}) \right| \, \mathrm{d}t \\
    &\leq \sum_{n=0}^{m} \int_{t_n}^{t_{n+1}}  \left\| (w_1-w_1^{n+1})_x \right\|^2 + \left\| (w_2-w_2^{n+1})_y \right\|^2 \, \mathrm{d}t + k\sum_{n=0}^{m}  \|\xi_u^{n+1}\|^2. 
\end{align*}

We only provide a detailed estimate for $\| (w_1-w_1^{n+1})_x \|^2$, as the estimate for $\| (w_2-w_2^{n+1})_y\|^2$ can be obtained similarly. Let $a_{ij}^{n+1} = a_{ij}(x,y,t_{n+1})$. Since $w_1(x,y,t)=a_{11}(x,y,t)v_1 + a_{12}(x,y,t)v_2$, we decompose
\begin{align*}
    (w_1(t)-w_1^{n+1})_x
    &= \bigl[a_{11}(v_1)_x-a_{11}^{n+1}(v_1)_x^{n+1}+(a_{11})_xv_1-(a_{11}^{n+1})_xv_1^{n+1}\bigr] \\
    &\quad + \bigl[a_{12}(v_2)_x-a_{12}^{n+1}(v_2)_x^{n+1}+(a_{12})_xv_2-(a_{12}^{n+1})_xv_2^{n+1}\bigr] \\
    &=: R_1 + R_2.
\end{align*}
By hypothesis {\rm(vi)}, we have
\begin{align*}
    R_1 
    % &:= a_{11}(v_1)_x - a_{11}^{n+1}(v_1)_x^{n+1} + (a_{11})_x v_1 - (a_{11}^{n+1})_x v_1^{n+1} \\
    &= a_{11}(v_1 -v_1^{n+1})_x + (a_{11}-a_{11}^{n+1})(v_1^{n+1})_x + (a_{11})_x(v_1 - v_1^{n+1}) + (a_{11} - a_{11}^{n+1})_x v_1^{n+1} \\
    &\leq C|(v_1 -v_1^{n+1})_x| + C|t-t_{n+1}|^{\frac{1}{2}}|(v_1^{n+1})_x| + C|v_1-v_1^{n+1}| + C|t-t_{n+1}|^{\frac{1}{2}}|v_1^{n+1}|. 
\end{align*}
Using $v_1=u_x$, we have
\begin{align*}
    \sum_{n=0}^{m} \int_{t_n}^{t_{n+1}}  \| R_1 \|^2 \, \mathrm{d}t \leq C\sum_{n=0}^{m} \int_{t_n}^{t_{n+1}} \|u(t)-u^{n+1}\|_{2}^2 \, \mathrm{d}t + C k^2\sum_{n=0}^{m}\|v_1^{n+1}\|_{1}^2.
\end{align*}
By the same argument, we have
\begin{align*}
    \sum_{n=0}^{m} \int_{t_n}^{t_{n+1}}  \| R_2 \|^2 \, \mathrm{d}t \leq C\sum_{n=0}^{m} \int_{t_n}^{t_{n+1}} \|u(t)-u^{n+1}\|_{2}^2 \, \mathrm{d}t + C k^2\sum_{n=0}^{m}\|v_2^{n+1}\|_{1}^2.
\end{align*}
Hence
\begin{align*}
    \sum_{n=0}^{m} \int_{t_n}^{t_{n+1}}  \| (w_1-w_1^{n+1})_x \|^2 \, \mathrm{d}t \leq C\sum_{n=0}^{m} \int_{t_n}^{t_{n+1}} \|u(t)-u^{n+1}\|_{2}^2 \, \mathrm{d}t + Ck^2\sum_{n=0}^{m} (\|v_1^{n+1}\|_{1}^2 + \|v_2^{n+1}\|_{1}^2).
\end{align*}
Clearly, the same estimate holds for $\sum_{n=0}^{m} \int_{t_n}^{t_{n+1}}  \| (w_2-w_2^{n+1})_y \|^2 \, \mathrm{d}t$. Combining the above bounds, we obtain the estimate of $\mathcal{I}_3$ in \eqref{ineq:error-estimate-term3-2D}. 
% \begin{align*}
%     \mathcal{I}_3 \leq k\sum_{n=0}^{m}  \|\xi_u^{n+1}\|^2 + C\sum_{n=0}^{m} \int_{t_n}^{t_{n+1}} \|u(t)-u^{n+1}\|_{2}^2 \, dt + Ck^2\sum_{n=0}^{m} (\|v_1^{n+1}\|_{1}^2 + \|v_2^{n+1}\|_{1}^2).
% \end{align*}
\end{proof}

\subsection{Proof of Lemma \ref{lem:xi_v^0-2d}}
\label{appendix:xi_v^0-2d}
\begin{proof}
    We split the numerical error $e_u^0 = u_0 - u_h^0$ and $e_{v_i}^0 = v_i^0 - v_{i,h}^0$ as
    $e_u^0 = \xi_u^0 - \eta_u^0$ and $e_{v_i}^0 = \xi_{v_i}^0 - \eta_{v_i}^0$ for $i = 1,2,$,
    where $\eta_u^0,\eta_{v_i}^0$ denote the projection error. Clearly, we have $\| \xi_{u}^0\| = \|\mathcal{P}^- u_0 - u_{h}^0\| \leq Ch^{r+1}$ by \eqref{ineq:proj-property-1-2d}. Taking the error equation \eqref{eq:error-2-2d} at $t=t^0$, we have 
    \[ (e_{v_1}^{0},p_h) + (e_{v_2}^{0},q_h) = L^-(e_u^{0},p_h,q_h). \]
    Taking $p_h = \xi_{v_1}^0$ and $q_h = \xi_{v_2}^0$, and by the definition of $L^2$-projection, we get
    \[ \| \xi_{v_1}^0\|^2 + \| \xi_{v_2}^0\|^2 = L^-(\xi_u^{0},\xi_{v_1}^0,\xi_{v_2}^0) - L^-(\eta_u^{0},\xi_{v_1}^0,\xi_{v_2}^0). \]
    By \eqref{ineq:alternating-flux-2d} of Lemma \ref{lem:num-flux-2D}, Cauchy's inequality and \eqref{ineq:proj-property-1-2d}
    \begin{align*}
        |L^-(\xi_u^{0},\xi_{v_1}^0,\xi_{v_2}^0)| &\leq \left(\| \nabla \xi_u^{0}\| + \mu h^{-\frac{1}{2}} \| [\xi_u^{0}]\|_{\Gamma_h} \right) \left(\| \xi_{v_1}^0\| + \| \xi_{v_2}^0\| \right) \\
        &\leq \frac{1}{8}\left(\| \xi_{v_1}^0\| + \| \xi_{v_2}^0\| \right)^2 + 4\left(\| \nabla \xi_u^{0}\|^2 + \mu^2 h^{-1} \| [\xi_u^{0}]\|_{\Gamma_h}^2 \right) \leq \frac{1}{8}\left(\| \xi_{v_1}^0\| + \| \xi_{v_2}^0\| \right)^2 + Ch^{2r+2}. 
    \end{align*}
    By \eqref{ineq:superconvergence-1} of Lemma \ref{lem:superconvergence}
    \begin{align*}
        |L^-(\eta_u^{0},\xi_{v_1}^0,\xi_{v_2}^0)| \leq Ch^{r+1} \| u_0\|_{{r+2}} \left(\| \xi_{v_1}^0\|+\| \xi_{v_2}^0\| \right) \leq \frac{1}{8}\left(\| \xi_{v_1}^0\| + \| \xi_{v_2}^0\| \right)^2 + Ch^{2r+2}. 
    \end{align*}
    Combining the estimates for the above two terms, we obtain 
    \[ \| \xi_{v_1}^0\|^2 + \| \xi_{v_2}^0\|^2 \leq Ch^{2r+2}. \]
\end{proof}

\end{appendices}

% \newpage

\printbibliography

\end{document}